\documentclass[reqno,11pt]{amsart}
\usepackage[all]{xy}

\usepackage{amsthm,amsfonts, amssymb, amscd}
\usepackage{euscript}

\usepackage{tikz}
\usetikzlibrary{matrix,arrows}
\usetikzlibrary{positioning}

\usepackage[normalem]{ulem}

\usepackage{hyperref}
\hypersetup{%
    linktoc=page
}

\let\oldtocsection=\tocsection
\let\oldtocsubsection=\tocsubsection
\renewcommand{\tocsection}[2]{\hspace{0em}\oldtocsection{#1}{#2}}
\renewcommand{\tocsubsection}[2]{\hspace{1em}\oldtocsubsection{#1}{#2}}

\newtheorem{Thm}{Theorem}[section]
\newtheorem{Lem}[Thm]{Lemma}
\newtheorem{Cor}[Thm]{Corollary}
\newtheorem{Prop}[Thm]{Proposition}

\theoremstyle{remark}
\newtheorem{Rem}[Thm]{Remark}
\theoremstyle{remark}
\newtheorem{example}[Thm]{Example}
\theoremstyle{definition}
\newtheorem{Def}[Thm]{Definition}

\numberwithin{equation}{section}

\newcommand{\C}{\mathbb{C}}           
\newcommand{\N}{\mathbb{ N}}           
\newcommand{\Z}{\mathbb{ Z}}           

\newcommand{\tr}{\operatorname{trace}}

\newcommand{\spec}{\operatorname{Spec}}
\newcommand{\Char}{\operatorname{char}}

\newcommand{\diag}{\operatorname{diag}}
\newcommand{\res}{\operatorname{res}}

\newcommand{\Gr}{\operatorname{Gr }}
\newcommand{\ev}{\operatorname{ev}}

\newcommand{\mult}{\operatorname{mult}}
\newcommand{\codim}{\operatorname{codim}}

\newcommand{\GC}{\operatorname{GC}}
\newcommand{\GD}{\operatorname{GD}}
\newcommand{\GL}{\operatorname{GL}}
\newcommand{\SL}{\operatorname{SL}}
\newcommand{\SO}{\operatorname{SO}}
\newcommand{\Sp}{\operatorname{Sp}}
\newcommand{\Grass}{\operatorname{Grass }}
\newcommand{\OG}{\operatorname{OG}}

\newcommand{\fb}{{\mathfrak b}}

\newcommand{\fg}{{\mathfrak g}}
\newcommand{\fh}{{\mathfrak h}}

\newcommand{\fl}{{\mathfrak l}}
\newcommand{\fm}{{\mathfrak m}}
\newcommand{\fn}{{\mathfrak n}}
\newcommand{\fo}{{\mathfrak o}}

\newcommand{\fp}{{\mathfrak p}}

\newcommand{\ft}{{\mathfrak t}}
\newcommand{\fu}{{\mathfrak u}}
\newcommand{\fs}{{\mathfrak s}}

\newcommand{\ga}{\alpha}
\newcommand{\gb}{\beta}
\newcommand{\gd}{\delta}
\newcommand{\gD}{\Delta}

\newcommand{\gre}{\epsilon}

\newcommand{\gl}{\lambda}

\newcommand{\gm}{\mu}
\newcommand{\gn}{\nu}

\newcommand{\gP}{\Phi}

 \newcommand{\co}{\mathcal{O}}
 \newcommand{\cp}{\mathcal{P}}
 \newcommand{\cs}{\mathcal{S}}

 \newcommand{\cf}{\mathcal{F}}
 
 \newcommand{\ce}{\mathcal{E}}
 \newcommand{\ct}{\mathcal{T}}
 \newcommand{\cb}{\mathcal{B}}

\newcommand{\vj}{\mathbf{j}}
\newcommand{\vq}{\mathbf{q}}
\newcommand{\vr}{\mathbf{r}}
\newcommand{\vs}{\mathbf{s}}
\newcommand{\vt}{\mathbf{t}}

\renewcommand{\tilde}{\widetilde}

\renewcommand{\bar}[1]{\overline{#1}}

\newcommand{\exco}{{\ \tikz[baseline={([yshift=-.7ex]current bounding box.center)}]{\draw[>=latex,->] (0,0) -- (.85,0);}\ }}
\newcommand{\exct}{{\ \tikz[baseline={([yshift=-.7ex]current bounding box.center)}]{\draw[double distance = 1.5pt,>=latex,->] (0,0) -- (.85,0);}\ }}

\begin{document}
\parskip=4pt
\baselineskip=14pt

\title[Excited Young diagrams]{Excited Young diagrams, equivariant K-theory, and Schubert varieties}

\author{William Graham}
\address{
Department of Mathematics,
University of Georgia,
Boyd Graduate Studies Research Center,
Athens, GA 30602
}
\email{wag@math.uga.edu}

\author{Victor Kreiman}
\address{
Department of Mathematics,
University of Wisconsin - Parkside,
Kenosha, WI 53140
}
\email{kreiman@uwp.edu}

\date{\today}

\begin{abstract}
We give combinatorial descriptions of the restrictions to $T$-fixed points of the
classes of structure sheaves of Schubert varieties in the $T$-equivariant $K$-theory
of Grassmannians and of maximal isotropic Grassmannians of orthogonal and symplectic types.
We also give formulas, based on these descriptions, for the Hilbert series and Hilbert polynomials at $T$-fixed points of the corresponding
Schubert varieties. These
descriptions and formulas are given
in terms of two equivalent combinatorial models: excited Young diagrams and set-valued tableaux.   The restriction fomulas are positive, in that for a Schubert variety of codimension $d$, the formula equals $(-1)^d$ times a sum, with nonnegative coefficients,
of monomials in the expressions $(e^{-\ga} -1)$,
as $\alpha$ runs over the positive roots.   In types $A_n$ and $C_n$ the restriction formulas
had been proved earlier by \cite{Kre:05}, \cite{Kre:06} by a different method.
In type $A_n$, the formula for the Hilbert series had been proved earlier by \cite{LiYo:12}.
The method of this paper,
which relies on a restriction formula of Graham \cite{Gra:02} and Willems \cite{Wil:06},
is based on the method used by Ikeda and Naruse \cite{IkNa:09} to obtain the analogous formulas in equivariant cohomology.
The formulas we give differ from the $K$-theoretic restriction formulas given by Ikeda and Naruse \cite{IkNa:11}, which use different versions
of excited Young diagrams and set-valued tableaux.
We also give Hilbert series and Hilbert polynomial formulas which are valid for Schubert varieties in any cominuscule flag variety, in terms
of the $0$-Hecke algebra.
\end{abstract}

\maketitle

\tableofcontents

\noindent

\section{Introduction}
In this paper we use equivariant $K$-theory to obtain information about the local structure
of Schubert varieties in Grassmannians or maximal isotropic Grassmannians of orthogonal or symplectic
type.  Such a Grassmannian is a generalized flag variety of the form $X = G/P$, where $G$ is one
of the groups $G = \SL_n(\C)$, $G = \SO_n(\C)$ or $G = \Sp_{2n}(\C)$, and $P$ is a parabolic subgroup.
These varieties have long attracted attention because of their connections
with combinatorics and representation theory.   Let $B \supset T$ denote a Borel subgroup
and maximal torus of $G$, and let $B^-$ denote the opposite Borel subgroup to $B$.
 If $Y$ is a Schubert variety in $X$
(that is, the closure of a $B^-$-orbit in $X$), then the structure sheaf $\co_Y$ defines an
element $[\co_Y]$ in $K_T(X)$, the Grothendieck group of $T$-equivariant coherent sheaves on $X$.
If $i: \{ x \} \hookrightarrow X$ is the inclusion of a $T$-fixed point, then the class $i_x^*[\co_Y]$
is an element in the representation ring $R(T)$ of $T$.  The restriction
$i_x^*[\co_Y]$ enables one to describe the ring of functions on the tangent cone of $Y$ at $x$ as a representation of $T$. We will refer to the class $i_x^*[\co_Y]$ as the restriction or
pullback of the class $[\co_Y]$ to the fixed point $x$.  The main results of this paper are
combinatorial formulas for these pullback classes.  Such formulas have particular interest
because in all of these cases (except for the odd orthogonal case)
the generalized flag variety is cominuscule, which means that the restriction formulas yield formulas for the
Hilbert series and Hilbert functions of the local rings $\co_{Y,x}$.\footnote{The maximal isotropic Grassmannians in the $SO(2n+1)$ and $SO(2n+2)$ cases are isomorphic, so one can obtain Hilbert series
and multiplicity formulas in the odd orthogonal case as well.  See Section \ref{ss.orthog_grassmannians}.}    As a consequence of our formulas we
deduce that in these examples (and in fact for Schubert varieties in any cominuscule flag variety),
 the Hilbert function coincides with the Hilbert polynomial.  Note that the restriction fomulas in this paper are positive, in that
 for a Schubert variety of codimension $d$, the formula is $(-1)^d$ times a sum, with nonnegative coefficients,
of monomials in the expressions $(e^{-\ga} -1)$,
as $\alpha$ runs over the positive roots.  Similar positivity results occur in the structure constants for the equivariant
$K$-theory of flag varieties; see \cite{GrRa:04}, \cite{GrKu:08}.
One consequence of the positivity in our restriction formulas is that the constants defining the Hilbert series and Hilbert polynomial
are given by positive termed enumerative formulas, i.e., one obtains them by counting well defined algebraic or combinatorial objects.

The combinatorial formulas for the $i_x^*[\co_Y]$ were obtained earlier, in the cases of the
Grassmannian and the Lagrangian Grassmannian, by Kreiman \cite{Kre:05}, \cite{Kre:06}.
The formulas were derived there by using equivariant Gr\"{o}bner degenerations of Schubert varieties in the neighborhood of a $T$-fixed point; these degenerations were obtained in \cite{KoRa:03}, \cite{Kre:03}, \cite{KrLa:04}, and \cite{GhRa:06}.  This method is discussed in more detail in
Section \ref{s.Grassmannian}.
The approach taken in this paper is different, and modeled on the
approach taken by Ikeda and Naruse \cite{IkNa:09}, who obtained restriction formulas
in equivariant cohomology.  The main tool
of Ikeda and Naruse is a formula of Andersen-Jantzen-Soergel \cite{AJS:94} and Billey \cite{Bil:99}.  This gives the
restriction of a Schubert class to a $T$-fixed point in terms of expressions in what is
called the nil-Coxeter (or nil-Hecke) algebra.  The nil-Coxeter formula works for the full flag variety (and hence for any generalized
flag variety).  In the classical cominuscule cases, however, it can be used to give formulas in terms
of combinatorics related to Young diagrams.   To obtain the formulas in equivariant $K$-theory,
we replace the cohomology formula by an analogous $K$-theory formula in the $0$-Hecke algebra, obtained
by Graham and Willems.   We use this to obtain general formulas for the Hilbert series and
Hilbert polynomials for Schubert varieties of cominuscule flag varieties at $T$-fixed points.
In the Grassmanian cases, we again relate these formulas to Young diagrams.
Our formulas are in terms of excited Young diagrams (the term is due to Ikeda and Naruse \cite{IkNa:09};
these were called subsets of Young diagrams in \cite{Kre:05}, \cite{Kre:06}).  We have generalized the definition
of excited Young diagrams for the $K$-theory formulas; the earlier definitions, which were used for
the equivariant cohomology formulas, are what we call reduced excited
Young diagrams.  Reduced excited Young diagrams were discovered independently by Kreiman \cite{Kre:05}, \cite{Kre:06} and Ikeda and Naruse \cite{IkNa:09}.
In type $A_n$, excited Young diagrams are the same combinatorial objects as the pipe dreams of \cite{WoYo:12} (see also \cite{LiYo:12}).
We also give formulas in terms of an alternative, but equivalent, combinatorial model,
namely set-valued tableaux. 
The set-valued tableaux which we use were introduced in \cite{Kre:05}, \cite{Kre:06}. In type $A_n$, they also appeared in \cite{WoYo:12}, where they were identified as special types of flagged set-valued tableaux. Flagged set-valued tableaux, which were introduced in  \cite{KMY:09}, generalize both set-valued tableaux \cite{Buc:02} and flagged tableaux \cite{Wac:85}.

 Ikeda and Naruse \cite{IkNa:11} used somewhat different versions of excited Young diagrams and set-valued tableaux
to obtain combinatorial formulas for functions $G_\gl(x|b)$, $GB^{(n)}_{\lambda}(x|b)$,  $GC^{(n)}_{\lambda}(x|b)$,
and  $GD^{(n)}_{\lambda}(x|b)$.  These are functions in variables $x_1,x_2, \ldots, x_n$ and $b_1,b_2, \ldots$, which depend
on a parameter $\gb$.   The function $G_\gl(x|b)$ is equal to the factorial Grothendieck polynomial of Macnamara \cite{Mcn:06}. These functions represent the classes of structure sheaves of Schubert varieties
in equivariant $K$-theory, in types $A$, $B$, $C$, and $D$, respectively.
In particular, Ikeda and Naruse prove that if $G$ is of type $A$, $B$, $C$, or $D$, then if one takes a function of the
appropriate type, sets $\gb = -1$, and chooses an appropriate specialization
of the variables (depending on $\mu$), the result is the restriction of $[\co_{X^{\lambda}}]$ to the point
corresponding to $\mu$.   Thus, their results lead to combinatorial formulas for the pullbacks of the
structure sheaves of Schubert varieties.  These formulas are different from the formulas given
in this paper.  See Section \ref{ss.previous} for examples comparing the formulas in this paper with
the results of \cite{IkNa:11}.

Since the $0$-Hecke restriction formula is valid for any generalized flag variety, it is natural
to ask why we focus on the Grassmannians and maximal isotropic Grassmannians.   There are
are two important properties which are relevant to these cases.  First,
given any $T$-fixed
point in a cominuscule flag variety, there exists a vector $\xi$ in the Lie algebra of $T$,
such that for any weight $\ga$ of $T$ on the tangent space of the flag variety at that point,
$\ga(\xi) = -1$ (see Proposition \ref{p.cominuscule}).  This property implies that in the cominuscule cases, restrictions in
equivariant $K$-theory give information about the Hilbert series (see Proposition \ref{p.fixedhilbert}).
Second, in the cases we consider---that is, in all cominuscule flag varieties,
as well as the maximal isotropic Grassmannians in the odd orthogonal case---the Schubert varieties and $T$-fixed points are
parametrized by elements of the Weyl group which are fully commutative in the sense
of Stembridge \cite{Ste:96}.  This property is used when we connect the $0$-Hecke formula to Young diagrams.

The contents of the paper are as follows.  Section \ref{s.background} describes the relation between
restriction formulas in equivariant $K$-theory and the tangent cone at $T$-fixed points, and
explains how, in the case of cominuscule flag varieties, this is connected with multiplicities and
Hilbert series.  This connection has been known for some time; we learned about it from
Michel Brion, who pointed out that in the cominuscule case,  equivariant multiplicities can
be used to compute multiplicities.
Section \ref{ss.mult} contains
Proposition \ref{p.fixedchar}, which is a version of
a result in an unpublished paper of Bressler \cite{Bre}, who used it to give a proof of a formula
of Kumar \cite[Theorem 2.2]{Kum:96} describing multiplicities in the tangent cone of Schubert varieties
in terms of the $0$-Hecke algebra.
This is is also related to work of Rossmann
\cite{Ros:89}.   The connection to Hilbert series is given
in  Proposition \ref{p.fixedhilbert}.   This connection to the Hilbert series is known---see for example
\cite[Section 9]{IkNa:09} or \cite{LiYo:12}; we have given some details not explained in these references.
Section \ref{ss.backgroundHecke} recalls some definitions
about the equivariant $K$-theory of the flag variety, and states the $0$-Hecke pullback formula
(Theorem \ref{t.pullback}); related formulas were given by Graham \cite{Gra:02} and
Willems \cite{Wil:06}.  Section \ref{ss.cominuscule} contains Proposition \ref{p.cominuscule},
which shows that cominuscule flag varieties have the geometric property needed to apply
Proposition \ref{p.fixedhilbert}.  We originally learned this result from Brion; it is also used
in \cite[Section 9]{IkNa:09}.
Combining the above results yields
 Hilbert series and Hilbert
polynomial formulas (Theorem \ref{t.cominusculeformula}) and a formula for the multiplicity
of a $T$-fixed point (Corollary \ref{c.cominusculemult}), which are valid for arbitrary
(not necessarily classical) cominuscule flag varieties; these formulas are given in Section \ref{ss.Hilbertcominuscule}.

Section \ref{s.nilHecke} proves some results about fully commutative elements and the $0$-Hecke
algebra which we need to obtain the connection with Young diagrams.

Section \ref{s.Grassmannian} concerns the case where $G = \SL(n,\C)$,
so $X = \Grass(d,n)$, the Grassmannian variety of all $d$-dimensional subspaces of $\C^n$.
This case is the foundation of all the classical cases.  Section \ref{ss.partitions_An}
gives background about Grassmannian permutations and partitions, which index the Schubert
varieties and the $T$-fixed points.  Section \ref{ss.eyd} defines excited
Young diagrams.   The first main result is Theorem
\ref{t.ktheory_eyd}, the restriction formula in terms of excited Young diagrams.  This theorem is proved
by finding a reduced expression for a Grassmannian permutation which is related to the Young diagram
of the corresponding partition (Section \ref{ss.eyd}), and then interpreting the terms
of the $0$-Hecke restriction formula in terms of excited Young diagrams (Proposition \ref{p.subsequences}).  Section \ref{ss.setvalued_tableaux} defines set-valued tableaux, which are in some ways
easier to work with than excited Young diagrams; Theorem
\ref{t.ktheory_svt} gives the restriction formula in terms of set-valued tableaux.  This theorem is proved by establishing
a bijection between appropriate collections of set-valued tableaux and excited Young diagrams.

Section \ref{s.orthosymplectic} deals with the remaining classical cases, the maximal isotropic
Grassmannians of orthogonal or symplectic types.  In these cases, the Schubert varieties and $T$-fixed points are indexed by
shifted Young diagrams.  The restriction formulas, which are obtained by adapting the methods of the previous
section, are in terms of excited shifted Young
diagrams (Theorem \ref{t.ktheory_eyd_bcd})  and set-valued shifted tableaux (Theorem \ref{t.ktheory_syt_bcd}).

Appendix \ref{s.appendix_roots} reviews some facts about root systems and Weyl groups.
Appendix \ref{s.appendix_restriction-opp} explains the relationship between different versions of the
$0$-Hecke restriction formula.

\section{Equivariant $K$-theory of the flag variety, Hilbert series and multiplicities}\label{s.background}

\subsection{Equivariant $K$-theory, Hilbert series and multiplicities} \label{ss.mult}
In this section we review some results relating equivariant $K$-theory to the local
rings of functions at $T$-fixed points.  We have included some proofs for the convenience
of the reader.   Let $T = (\C^*)^m$ denote a complex torus.
Let $R(T)$ denote the representation ring of $T$; this is the set of all $\Z$-linear combinations
of $e^{\gl}$, where $\gl$ is a weight of $T$.

If $M$ is a scheme with a $T$-action, let $K_T(M)$ denote the Grothendieck
group of coherent sheaves on $X$.  If $M$ is smooth, then $K_T(M)$ can be identified
with the Grothendieck group of vector bundles on $M$.  A $T$-equivariant
coherent sheaf $\cf$ on $M$ defines a class $[\cf] \in K_T(M)$.  In particular,
a closed $T$-invariant subscheme $Z$ of $M$ defines a class $[\co_Z] \in K_T(M)$.
If $M$ is a point then $K_T(M)$ is identified with $R(T)$.

Let $X$ be a smooth $T$-variety and $x \in X^T$ be an
isolated fixed point.  Let $\Phi(T_x)$ denote the set of weights of $T$ on the tangent space
$T_x X$.  The fixed point $x$ is said to be attractive if there is a half-space in
$\ft^*$ containing $\Phi(T_x)$.  This implies that $x$ has a $T$-stable neighborhood in $X$
which is $T$-equivariantly isomorphic to $T_x X$, and such that $x$ corresponds to the origin in $T_x X$.
Using this fact we can prove results about the pullbacks of classes in $K_T(X)$ to $x$
by reducing to the case where $X$ is a vector space with a linear $T$-action and $x$ is the
origin (see Proposition \ref{p.fixedchar}).

Let $Y \subset X$ be a $T$-stable subscheme containing $x$, and let $ \co_{Y,x}$ denote the local
ring of $Y$ at $x$ with maximal ideal $\fm = \fm_{Y,x}$.  Let $ \Gr \co_{Y,x} = \oplus_{i=0}^{\infty} \fm^i / \fm^{i+1}$.
By definition, the tangent cone of $Y$ at $x$ is $\spec ( \Gr \co_{Y,x})$ (see \cite[Section 2]{Kum:96}).
Let $\hat{R}$ denote the set of expressions of the form $\sum_{\mu \in \hat{T}} c_{\mu} e^{\mu}$.
The group $T$ acts on $\Gr \co_{Y,x}$ with finite multiplicities, so we can define $\Char ( \Gr \co_{Y,x} ) \in \hat{R}$
as $\Char ( \Gr \co_{Y,x} )  =  \sum m_{\mu} e^{\mu}$, where $m_{\mu}$ is the multiplicity of the weight $\mu$ in $Gr \co_{Y,x} $.

Let $f$ be an element of the quotient field of $R(T)$ of the form
$$
f = \frac{r}{\prod_{\mu \in \hat{T}} (1 - e^{\mu})^{n_{\mu}}},
$$
where $r \in R(T)$, $n_{\mu} \in \Z_{\ge 0}$, and such that there is a half-space in $\ft^*$
containing all the $\mu$ with $n_{\mu} \neq 0$.  Define $F(f) \in \hat{R}$ to be the series
$$
r \prod_{\mu \in \hat{T}} ( \sum 1 + e^{\mu} + e^{2 \mu} + \cdots )^{n_{\mu}}).
$$

The following proposition is a version of a result in an unpublished paper of Bressler, and
is also related to \cite[Lemma 1.1]{Ros:89}.  Bressler \cite{Bre} used this result to give a proof of a formula
of Kumar \cite[Theorem 2.2]{Kum:96} describing the multiplicities in the ring of functions on
the tangent cone to a Schubert variety at a $T$-fixed point in terms of the $0$-Hecke algebra
(see \cite[Remark 2.13]{Kum:96}).

\begin{Prop} \label{p.fixedchar}
Let $x$ be an attractive fixed point in the smooth $T$-variety $X$, and let $Y \subset X$ be a $T$-stable
subscheme containing $X$.  Let $i: \{ x \} \hookrightarrow X$ denote the inclusion,
and $[\co_Y] \in K_T(X)$ the class of the structure sheaf of $Y$.  Then
$$
\Char (\Gr \co_{Y,x}) = F\left(\frac{i^*[\co_Y]}{\prod_{\mu \in \Phi(T_x)}(1 - e^{-\mu})}\right).
$$
\end{Prop}

\begin{proof}
There is a $T$-stable affine open neighborhood of $x$ in $X$ which is $T$-equivariantly
isomorphic to $V = T_x X$ (see \cite[Corollary 2]{Bia:73}).  We can replace $X$ by this neighborhood and therefore assume
$X = V = \spec A$, where $A = S(V^*)$.  Let $I$ denote the ideal of $Y$ in $A$ and $B = A/I$,
so $Y = \spec B$.  Let $\fn$ denote the maximal ideal of $x$ in $B$.  Then $\co_{Y,x} = B_{\fn}$,
and $\fm = \fn B_{\fn} \subset B_{\fn}$.  Define $\Gr B = \oplus_{i=0}^{\infty} \fn^i / \fn^{i+1}$.
The natural map $B / \fn^i \to B_{\fn} / \fm^i$ is an isomorphism for all $i$.  This implies
that the natural map $\fn^i / \fn^{i+1} \to \fm^i / \fm^{i+1}$ is an isomorphism for all $i$, so
we obtain a $T$-equivariant isomorphism $\Gr \co_{Y,x} \to \Gr B$.  Therefore
$\Char ( \Gr \co_{Y,x} ) = \Char (\Gr B)$, which in turn is equal to $\Char (B)$.

There exists a $T$-equivariant resolution of $B$ by finite free $A$-modules
$$
0 \to F_d \to F_{d-1} \to \cdots \to F_0 \to B \to 0
$$
where each $F_j$ is isomorphic to $\oplus_i A \otimes \C_{\lambda_{i,j}}$.  Here $A \otimes \C_{\lambda_{i,j}}$ denotes the $A$-module $A$ with $T$-action twisted by $\lambda_{i,j} \in \hat{T}$.   (See \cite[Lemma 1.1]{Ros:89}.)
This resolution corresponds to the resolution of $\co_Y$ over $\co_X$:
$$
0 \to \cf_d \to \cf_{d-1} \to \cdots \to \cf_0 \to \co_Y \to 0
$$
where $\cf_j$ is isomorphic to $\oplus_i \co_X \otimes \C_{\lambda_{i,j}}$
In $R(T)$, $i^*[\co_X] = 1$, and therefore
$$
i^*[\co_Y] = \sum_{i,j} (-1)^j e^{\lambda_{i,j}} i^*[\co_X] = \sum_{i,j} (-1)^j e^{\lambda_{i,j}}.
$$
On the other hand,
\begin{eqnarray*}
\Char (B) & = & \sum_j (-1)^j \Char (F_j) = \sum_{i,j} (-1)^j e^{\lambda_{i,j}} \Char (A) \\
                  & = & \sum_{i,j} (-1)^j e^{\lambda_{i,j}} F\left(\frac{1}{\prod_{\mu \in \Phi(T_x)}(1 - e^{-\mu})}\right) \\
                  & = & F\left(\frac{i^*[\co_Y]}{\prod_{\mu \in \Phi(T_x)}(1 - e^{-\mu})}\right),
\end{eqnarray*}
as desired.
\end{proof}

The Hilbert function of $\Gr \co_{Y,x}$ is by definition the function
$n \mapsto \dim (\fm^n / \fm^{n+1})$.  For $n >> 0$, this function is a polynomial,
which we denote by $h(Y,x)(n)$.   Let $r$ denote the degree of $h(Y,x)(n)$. The multiplicity of $Y$ at $x$, which we denote by $\mult(Y,x)$, is $r!$ times the leading coefficient of $h(Y,x)(n)$. The Hilbert series $H(Y,x)(t)$
is the generating function associated to the Hilbert function of $\Gr \co_{Y,x}$.
By definition,
$$
H(Y,x)(t)= \sum_{i = 0}^{\infty} \dim (\fm^i / \fm^{i+1}) t^i.
$$

Let $S$ denote a formal sum $S = \sum_{\mu \in \hat{T}} c_{\mu} e^{\mu}$.
Suppose that there exists an element $\xi \in \ft$ such that $\mu(\xi)$ is a nonnegative integer
for each $\mu$ with $c_{\mu} \neq 0$, and such that for each $n \in \Z_{\ge 0}$,
there exist only finitely many $\mu$ with $c_{\mu} \neq 0$ and $\mu(\xi) = n$.  Then
define $ev_{\xi} S$ to be the power series $ev_{\xi} S = \sum_{\mu \in \hat{T}} t^{\mu(\xi)}$.

\begin{Prop} \label{p.fixedhilbert}
Keep the hypotheses of Proposition \ref{p.fixedchar}, and assume in addition that there exists
$\xi \in \ft$ such that $\ga(\xi) = -1$ for each $\ga \in \Phi(T_x)$.  Then
$$
H(Y,x)(t) = \ev_{\xi} \Char(\Gr \co_{Y,x}) = \frac{\ev_{\xi} i^* [\co_Y]}{(1-t)^d},
$$
where $d = \dim X$.
\end{Prop}
\begin{proof}
Let $z_1,\ldots,z_q$ be a basis for $\fm/\fm^2\subseteq (T_xX)^*$, with weights $-\ga_1,\ldots,-\ga_q$, $\ga_k\in \Phi(T_x)$. For $\vj=(j_1,\ldots,j_i)$ a sequence of integers between 1 and $q$, define $z_\vj=z_{j_1}\cdots z_{j_i}\in \fm^i/\fm^{i+1}$, and let $\gm_\vj=-\ga_{j_1}-\cdots-\ga_{j_i}$, the weight of $z_\vj$. Let $\cb_i$ be a collection of $z_\vj$'s which forms a basis for $\fm^i/\fm^{i+1}$. Then $\Char(\Gr\co_{Y,x})=\sum_{i=0}^{\infty}\sum_{z_\vj\in \cb_i}e^{\gm_\vj}$, and
\begin{equation*}
    \ev_{\xi} \Char(\Gr \co_{Y,x})=\sum_{i=0}^{\infty}\sum_{z_\vj\in \cb_i} t^{\gm_\vj(\xi)}=\sum_{i=0}^{\infty}\sum_{z_\vj\in \cb_i}t^i=H(Y,x)(t).
\end{equation*}
\end{proof}

\subsection{Equivariant K-theory of flag varieties} \label{ss.backgroundHecke}
In this section we recall some background about equivariant $K$-theory and flag
varieties.  Let $G$ be a complex semisimple algebraic group, $B$ a Borel subgroup of $G$, and
$T$ a maximal torus in $B$.  Let $P \supset B$ be a parabolic subgroup of $G$.
Let $\fg, \fb, \ft$ and $\fp$ denote the Lie algebras of these groups.  Given a representation $V$ of
$T$, let $\Phi(V) \subset \ft^*$ denote the set of weights of $V$.
Let $\Phi = \Phi(\fg)$
denote the set of roots of $\fg$ with respect to $\ft$, and let $\Phi^+$ denote the set of positive
roots, chosen so that the positive root spaces are in $\fb$, i.e.~, so that $\Phi^+ = \Phi(\fb)$.  The set of negative roots is $\Phi^- = -\Phi^+$.
Let $L$ be a Levi subgroup of $P$ containing $T$, and let $\Phi_{\fl} = \Phi(\fl)$, and $\Phi^+_{\fl} = \Phi_{\fl} \cap \Phi^+$,
$\Phi^-_{\fl} = \Phi_{\fl} \cap \Phi^-$.
Let $B^-$ denote the opposite Borel subgroup
to $B$ and $\fb^-$ its Lie algebra.  Let $W = N_G(T)/T$ denote the Weyl group; we
will often use the same letter to denote an element of $W$ and a representative in $N_G(T)$.
Let $S$ denote the set of simple reflections in $W$, so $(W,S)$ is a Coxeter system.
Let $W_P = W_L$ denote the Weyl group of $L$.  Then $W_P$ is
a subgroup of $W$.  Each coset $wW_P$ in $W$ contains
a unique minimal length element and we let $W^P$ denote the set of minimal length
coset representatives in $W$.  The element $w$ is in $W^P$ if and only if $w(\Phi^+_{\fl}) \subset \Phi^+$
(cf.~\cite[2.5.3]{BiLa:00}).

Let $Y = G/B$, $X = G/P$, and let $\pi: Y \to X$ denote the projection.  We will need a formula
for the pullback of the class in $K_T(X)$ of the structure sheaf of a Schubert variety to a fixed point.
We explain how to obtain this formula from the known formula for the corresponding problem on $Y$.
If $w \in W$, we define the Schubert varieties $X^w = \overline{B^- \cdot wP} \subset X$ and
$Y^w =  \overline{B^- \cdot wB} \subset Y$.
The variety $X^w$ only depends on the coset $wW_P$, and if we take $w \in W^P$, then
$\codim X^w = l(w)$.   Since $\pi$ is a flat map, it induces a map $\pi^*: K_T(X) \to K_T(Y)$
satisfying $\pi^*[\co_{X^w}] = [\co_{\pi^{-1}(X^w)}] = [\co_{Y^w}]$.
If $v \in W^P$, let $i_v: \{ pt \} \to X$ (resp.~$j_v: \{ pt \} \to Y$) denote the map taking the point to
$vP$ (resp.~$vB$).  Because $X$ and $Y$ are smooth, there are induced pullback maps
$i_v^*: K_T(X) \to R(T)$ and $j_v^*: K_T(Y) \to R(T)$.  As $i_v = \pi \circ j_v$, we have
$i_v^* = j_v^* \circ \pi^*$, and therefore
$$
i_v^*[\co_{X^w}] = j_v^* \pi^* [\co_{X^w}] = j_v^* [\co_{Y^w}].
$$

There is a general formula for $i_v^*[\co_{X^w}] $.  To state it we need to define the $0$-Hecke algebra
associated to the Coxeter system $(W,S)$, over the ring $R$.  This algebra is a free
$R$-algebra with basis $H_w$, for $w \in W$, and the multiplication is characterized by the following
relations: $H_1$ is the identity element (here $1$ denotes the identity element of $W$); for $s \in S$, $w \in W$,
we have $H_s H_w = H_{sw}$ if $l(sw)> l(w)$,  $H_s H_w = H_{w}$ if $l(sw)< l(w)$, and $H_s^2 = H_s$.

\begin{Rem}\label{r.nilHecke}
The term $0$-Hecke algebra has been used (see, for example, \cite{Car:86}, \cite{Fay:05})
for the algebra with basis $J_w$ (for $w \in W$),
characterized
by the properties $J_s J_w = J_{sw}$ if $l(sw)> l(w)$, $J_s J_w = J_{w}$ if $l(sw)< l(w)$, and $J_s^2 = -J_s$
(and the formula in \cite{Gra:02} is given in terms of this algebra).  If we set $H_s = - J_s$, we see
that this algebra is the same as the algebra defined above, and we can translate between the two presentations,
since
 $$
 J_{s_1} J_{s_2} \ldots J_{s_k} = (-1)^{l(u) - k} J_u
 \Leftrightarrow H_{s_1} H_{s_2} \ldots H_{s_k} = H_u.
 $$
In the case of equivariant cohomology, the appropriate algebra is one in which the relations $J_s^2 = - J_s$ are replaced by relations of the form
$T_s^2 = 0$.  This algebra has been called the nil-Coxeter algebra or the nil-Hecke ring or algebra (see
for example
\cite{FoSt:94}, \cite{KoKu:86},  \cite{Gin:97}).
\end{Rem}

\begin{Def} \label{d.tvj}
Let $\vs = (s_1, s_2, \ldots, s_l)$ be a sequence of simple reflections.
Define
$T(w, \vs)$ to be the set of subsequences $\vt = (s_{i_1}, \ldots, s_{i_m})$, $1\leq i_1<\cdots<i_m\leq l$ such that
$H_{s_{i_1}} H_{s_{i_2}} \cdots H_{s_{i_m}} = H_w$.  Define $l(\vt) = m$, and define $e(\vt) = l(\vt) - l(w)$.
\end{Def}

\begin{Rem}
In the above definition of $T(w,\vs)$, if $(i_1,\ldots,i_m)$ and $(j_1,\ldots,j_m)$ are different subsequences of $(1,\ldots,l)$, then we regard $(s_{i_1}, \ldots, s_{i_m})$ and $(s_{j_1}, \ldots, s_{j_m})$ as different subsequences of $\vs$, even if they have the same entries.
\end{Rem}

We can now state the restriction formula.

\begin{Thm} \label{t.pullback}
Let $v,w \in W^P$.
Fix a reduced expression $\vs = (s_{1}, \ldots s_{l})$ for $v$, and for $c = 1, \ldots, l$, let
$r(c) = s_{1} s_{2} \cdots s_{c-1}(\ga_c)$.  Then
\begin{eqnarray} \label{e.pullback}
 i_v^*[\co_{X^w}]  & = & (-1)^{l(w)}  \sum_{\vt \in T_{w,\vs} }(e^{-r(i_1)}-1) (e^{-r(i_2)} - 1) \cdots (e^{-r(i_m)} - 1) \\
  & = & \sum_{\vt \in T_{w,\vs}} (-1)^{e(\vt)} (1 - e^{-r(i_1)}) (1 - e^{-r(i_2)}) \cdots (1 - e^{-r(i_m)})
 \end{eqnarray}
 where $T_{w,\vs}$ and $e(\vt)$ are as in Definition \ref{d.tvj}.
 \end{Thm}

 This formula can be deduced from formulas for restrictions to fixed points
 given by Graham and Willems (see \cite{Gra:02}, \cite{Wil:06}).  (The paper \cite{Wil:06}
 gives formulas for restrictions of a different basis than the basis of structure sheaves
 of Schubert varieties, but
 the relationship between bases given in \cite[Proposition~2.2]{GrKu:08} allows one to deduce formulas
 for the restrictions of one basis from the formulas for another basis.)   See also
 \cite{Knu:08}.

 \begin{Rem} \label{r.tangent}
Observe that $\{ r(1), r(2), \ldots, r(l) \} \subset \Phi(T_v(G/P))$ (in
 the notation of Theorem \ref{t.pullback}).
Indeed,
 the set $\{ r(1), r(2), \ldots, r(l) \}$ of Theorem \ref{t.pullback} is equal to
 $v(\Phi^-) \cap \Phi^+$, and the set $\Phi(T_v(G/P))$
 of weights of $T_v(G/P)$ is $v(\Phi(\fg/\fp)) = v(\Phi^- \setminus \Phi^-_{\fl})$.
 Because $v \in W^P$, $v(\Phi^-_{\fl}) \subset \Phi^-$.  Therefore
 $v(\Phi^-) \cap \Phi^+ \subset v(\Phi^- \setminus \Phi^-_{\fl})$, which implies
 the assertion.
\end{Rem}

\subsection{Cominuscule flag varieties} \label{ss.cominuscule}
Let $P \supset B$ be a standard parabolic subgroup of $G$;
let $\fl$ be a Levi subalgebra of $\fp$ containing $\ft$, and let $\fu$ denote the nilradical
of $\fp$, so we have a Levi decomposition
$\fp = \fl + \fu$.   Let $\fu^-$ denote the nilradical of the opposite parabolic subalgebra
to $\fp$.  If $P$ is maximal, then $P$ corresponds to some simple root $\alpha$, in the sense
that the simple roots of $\fl$ are the simple roots of $\fg$ other than $\alpha$.

\begin{Lem} \label{l.rootl}
Let $\ga_1, \ldots, \ga_r$ denote the simple roots for $\fg$ and let $P$ denote the
maximal parabolic subgroup corresponding to $\ga_i$.  If $\ga = \sum_k n_k \ga_k$ is
a root of $\fg$, and $n_i = 0$, then $\ga$ is a root of $\fl$.
\end{Lem}

\begin{proof}
By replacing $\ga$ by $-\ga$ if necessary we may assume each $n_k \ge 0$ (i.e.
$\ga$ is a positive root).
The proof is by induction on $\sum n_k$.  If $\sum n_k = 1$ then $\ga$ is simple
and, as noted above, $\ga$ is a root of $\fl$.  Suppose now that the statement of
the lemma holds for all roots $\beta = \sum m_k \ga_k$ with $\sum m_k < \sum n_k$.
Since $(\ga, \ga) = \sum n_k (\ga, \ga_k) > 0$, there exists some $j \neq i$ with
$(\ga, \ga_j)>0$.  Then $\beta = s_{\ga_j}(\ga) = \sum m_k \ga_k$ is a root, necessarily positive
(since the only root which changes sign under $s_{\ga_j}$ is $\ga_j \neq \beta$).
Moreover, $\sum m_k =  \sum n_k - 2\frac{(\ga,\ga_j)}{(\ga_j,\ga_j)} < \sum n_k$.
Our inductive hypothesis implies that $\beta$ is a root of $\fl$; since the set of roots of
$\fl$ is preserved by $s_{\ga_j}$, we conclude that $s_{\ga_j} \beta = \alpha$ is a root
of $\fl$.
\end{proof}

If $G$ is simple, a maximal parabolic subgroup $P$ is called cominuscule if the corresponding
simple root $\alpha$ appears with coefficient equal to $1$
when the highest root of $G$ is written as a sum
of simple roots.  The corresponding generalized flag variety $G/P$ is
also called cominuscule.  Cominuscule flag varieties have the following
useful property
(which we learned from Michel Brion).

\begin{Prop} \label{p.cominuscule}
Let $G/P$ be a cominuscule generalized flag variety.  For any $v \in W^P$, there exists an element
$\xi \in \ft$ (depending on $v$) such that for any weight $\alpha$ of $T$ on
$T_{vP}(G/P)$, we have $\alpha(\xi) = -1$.
\end{Prop}

\begin{proof}
Let $\ga_1, \ldots, \ga_r$ denote the simple roots of $\fg$; these form a basis
of $\ft^*$, and we denote the dual basis of $\ft$ by $\xi_1, \ldots, \xi_r$.
Assume that $P$ corresponds to the simple root $\ga_i$.
First suppose that $v= e$ is the identity.  In this case, we can take
$\xi = \xi_i$.  The reason is that
as $T$-representations,
$$
T_{eP} (G/P) \cong \fg/\fp \cong \oplus \C_{-\alpha},
$$
where the sum is over the positive roots of $\fg$ which are not roots of $\fl$.
If $\ga = \sum n_k \ga_k$ is such a root, then $n_i > 0$ by Lemma \ref{l.rootl};
$n_i  \le 1$ since $\ga_i$ occurs with coefficient $1$ in the highest root of $\fg$,
so $n_i = 1$, and then $-\ga(\xi_i) = - n_i = -1$ as asserted.  For general
$v \in W^P$, we can take $\xi = v \xi_i$, since the set of weights of $T_{vP} (G/P)$
is $v$ applied to the set of weights of $T_{eP} (G/P)$.
\end{proof}

\subsection{Hilbert series and Hilbert polynomials in cominuscule flag varieties} \label{ss.Hilbertcominuscule}
We can now describe the Hilbert series and Hilbert polynomial of a Schubert variety
in a cominuscule generalized flag variety at a $T$-fixed point.  We will write
$H(X^w, v)(t)$ for the Hilbert series $H(X^w, vP)(t)$
and $h(X^w,v)(n)$ for the Hilbert polynomial $h(X^w,vP)(n)$.

\begin{Thm} \label{t.cominusculeformula}
Let $G/P$ be a cominuscule generalized flag variety and $v, w \in W^P$.
Fix a reduced expression $\vs = (s_{1}, \ldots s_{l})$ for $v$.  Let $d = \dim G/P$.
The Hilbert series $H(X_w,v)$ is given by
$$
H(X^w,v)(t) = \sum_{\vt \in T_{w, \vs}} \frac{(-1)^{e(\vt)}}{(1-t)^{d - l(\vt)}}.
$$
The Hilbert function is equal to the Hilbert polynomial $h(X^w,v)(n)$ for
all $n$, and is given by the formula
$$
h(X^w,v)(n) = \sum_{\vt \in T_{w, \vs} }(-1)^{e(\vt)}
\begin{pmatrix}
n + d - l(\vt) - 1 \\
d - l(\vt) -1.
\end{pmatrix}
$$
\end{Thm}

\begin{proof}
By Proposition \ref{p.cominuscule}, there exists $\xi$ in $\ft$ so that $\ga(\xi) = -1$ for each weight $\ga$ of $T_v(G/P)$.
By  Remark \ref{r.tangent}, each $r(i)$ is a weight of $T_v(G/P)$.  Hence $\ev_{\xi}(e^{-r(i)}) = t$.  Therefore, Proposition \ref{p.fixedhilbert} and Theorem \ref{t.pullback} imply
\begin{equation*}
\begin{split}
    H(X^w,v)(t)=\frac{\ev_\xi i_v^* [\co_{X^w}]}{(1-t)^d}
        &= \frac{\ev_\xi \sum_{\vt \in T_{w,\vs}} (-1)^{e(\vt)} (1 - e^{-r(i_1)}) (1 - e^{-r(i_2)}) \cdots (1 - e^{-r(i_m)})}{(1-t)^d}\\
        &= \frac{\sum_{\vt \in T_{w,\vs}} (-1)^{e(\vt)} (1-t)^{m}}{(1-t)^d}
        = \sum_{\vt \in T_{w,\vs}} \frac{(-1)^{e(\vt)} }{(1-t)^{d-l(\vt)}}.
\end{split}
\end{equation*}
Recalling the identity
$\displaystyle
    \frac{1}{(1-t)^k}=\sum_{n=0}^{\infty} {n+k-1\choose k-1}t^n,
$
we obtain
\begin{equation*}
\begin{split}
    H(X^w,v)(t)&=\sum_{\vt \in T_{w,\vs}}(-1)^{e(\vt)}\sum_{n=0}^{\infty}
    {n+d-l(\vt)-1\choose d-l(\vt)-1}t^n\\
    &=\sum_{n=0}^{\infty}\left(\sum_{\vt \in T_{w,\vs}}(-1)^{e(\vt)}
    {n+d-l(\vt)-1\choose d-l(\vt)-1}\right)t^n.
\end{split}
\end{equation*}
Thus
\[
h(X^w,v)(n)=\sum_{\vt \in T_{w,\vs}}(-1)^{e(\vt)}
    {n+d-l(\vt)-1\choose d-l(\vt)-1}.
\]
\end{proof}

We may alternatively index the summations for the Hilbert series and Hilbert polynomial by the nonnegative integers $\N$.
\begin{Cor}\label{c.cominusculemult}
Let $G/P$ be a cominuscule generalized flag variety and $v, w \in W^P$, $v\geq w$. Fix a reduced expression $\vs = (s_{1}, \ldots s_{l})$ for $v$. Let $d_w=\dim X^w=\dim G/P-l(w)$. For $k\in\N$, define $m_k=\#\{\vt\in T_{w,\vs}\mid e(\vt)=k\}$.
Then 
\begin{alignat*}{3}
    &H(X^w,v)(t)& &= &\ \ &\sum_{k\in\N}\frac{(-1)^k m_k}{(1-t)^{d_w-k}}\\
    &h(X^w,v)(n)& &=& &\sum_{k\in\N} (-1)^k m_k {n+d_w-k-1\choose d_w-k-1}\\
    &\mult(X^w,v)& &=& &m_0
\end{alignat*}
\begin{proof} Let $d=\dim G/P$.
Note that $e(\vt)=l(\vt)-l(w)$, so $d-l(\vt)=(d-l(w))-e(\vt)=d_w-e(\vt)$. Thus
\begin{equation*}
    H(X^w,v)(t) = \sum_{\vt \in T_{w, \vs}} \frac{(-1)^{e(\vt)}}{(1-t)^{d - l(\vt)}}
    =\sum_{\vt\in T_{w,\vs}}\frac{(-1)^{e(\vt)}}{(1-t)^{d_w-e(\vt)}}
    =\sum_{k\in\N}\frac{(-1)^k m_k}{(1-t)^{d_w-k}}.
\end{equation*}
The formula for the Hilbert polynomial follows similarly. The highest degree term of $h(X^w,v)(n)$ is $m_0 n^{d_w-1}/(d_w-1)!$, implying that $\mult(X^w,v)=m_0$.
\end{proof}
\end{Cor}

The formula for multiplicity in the above corollary can be restated as follows.
Let $G/P$ be a cominuscule generalized flag variety and $v, w \in W^P$.
Fix a reduced expression $(s_{1}, \ldots s_{l})$  for $v$. Recall that $v\geq w$ in the Bruhat order if and only if there exists a subsequence $(i_1,\ldots,i_m)$ of $({1}, \ldots, {l})$ such that $(s_{i_1},\ldots, s_{i_m})$ is a reduced expression for $w$; in this case, $\mult(X^w,v)$ is equal to the number of such subsequences.

\section{The $0$-Hecke algebra and fully commutative elements} \label{s.nilHecke}
This section contains some results about fully commutative elements and the $0$-Hecke
algebra which we need to connect the pullback formula of Theorem \ref{t.pullback}
with the combinatorics of Young diagrams.

Given any
$q$-tuple $\vs = (s_1, \ldots, s_q)$ of elements of $S$,  let $H_{\vs} = H_{s_1} H_{s_2} \cdots H_{s_q}$.
If $q = 0$ we define $H_{\vs} = H_1$.
Let $(\vs, s)$ denote the $q+1$-tuple $(s_1, \ldots, s_q,s)$, and given
a $q'$-tuple $\vs = (s'_1, \ldots, s'_{q'})$, let $(\vs, \vs') = (s_1, \ldots, s_q, s'_1, \ldots, s'_{q'})$.
The length of a $q$-tuple $\vs$ is $l(\vs) = q$.  If
\begin{equation} \label{e.reduced}
w = s_1 s_2 \cdots s_q
\end{equation}
 and $l(w) = l(\vs)$
then we will say $\vs$ is a reduced expression for $w$.  We will also use the
term reduced expression to refer to
the equation \eqref{e.reduced}.

We begin with some preliminary results.

\begin{Lem} \label{l.minlength}
If $H_{\vs} = H_w$, then $l(\vs) \ge l(w)$. If $l(\vs) = l(w)$, then $\vs$ is a reduced expression for $w$.
\end{Lem}

\begin{proof}
We proceed by induction on $l(\vs)$.  If $l(\vs) = 0$ or $l(\vs) = 1$ then the lemma
is trivial.  Suppose the lemma is true for tuples of length $q$, and
$\vs = (\vt, s_{q+1})$, where $\vt = (s_1, \ldots, s_q)$.  Then $H_{\vt} = H_u$ for
$u \in W$ with $l(\vt) \ge l(u)$.  Then $w$ equals either $u$ or $u s_{q+1}$, so
$$
l(\vs) = l(\vt)+1 \ge l(u)+1 \ge l(w).
$$

Now assume that $l(\vs) = l(w)$, so $l(w)=q+1$. Then since $l(u)\leq l(\vt)=q$ and  $l(w)$ equals either $l(u)$ or $l(u)+1$, we must have that $l(u)=q$ and $l(w)=l(u)+1$. Thus $w$ must equal $u s_{q+1}$. By the inductive hypothesis, $u=s_1\cdots s_q$. Therefore $w$ equals $s_1\cdots s_{q+1}$.
\end{proof}

The right (resp.~left) weak order on $W$ is the transitive closure of the relation
$u <_R us$ (resp.~$u<_L su$) for $u \in W$, $s \in S$ with $l(u) < l(us)$ (resp.~$l(u)< l(su)$).
Given a reduced
expression $w = s_1 \cdots s_q$, for any $k<q$, we have $s_1 \ldots s_k <_R w$,
and for any $k>1$, we have $s_k s_{k+1} \cdots s_q <_L w$.
We can extend these results to the Hecke algebra; we only state the version using
$<_R$.

\begin{Prop} \label{p.weakorder}
Suppose that $\vs = (s_1, \ldots, s_q)$ and $H_{\vs} = H_w$.  Suppose that
$k \le q$ and $\vt = (s_1, \ldots, s_k)$ is a reduced expression for $u = s_1 \cdots s_k$.
Then $u \le_R w$.
\end{Prop}

\begin{proof}
Let
$\vr = (s_{k+1}, \ldots, s_q)$, so $\vs = (\vt, \vr)$.  It suffices to show that there is a subsequence
$\vq$ of $\vr$ such that $(\vt, \vq)$ is a reduced expression for $w$.  We proceed by induction
on $l(\vs)$.   Lemma \ref{l.minlength} implies that $l(\vs) \ge l(w)$.   If $l(\vs)= l(w)$ we are done.
Otherwise, there is some $j \ge k$ such that
$$
u < u s_{k+1}  < \cdots u s_{k+1} \cdots s_j > u s_{k+1} \cdots s_j s_{j+1}.
$$
Let $\vr'$ denote the sequence $\vr$ with $s_{j+1}$ deleted, and
let $\vs' = (\vt, \vr')$.   Since
$H_{s_1} \ldots H_{s_j} = H_{s_1} \ldots H_{s_j} H_{s_{j+1}}$,
we have $H_{\vs'} = H_{\vs}$.  Our inductive hypothesis to $\vs'$ implies that there is
a subsequence $\vq$ of $\vr'$ such that $(\vt, \vr')$ is a reduced expression for $w$.
Since $\vq$ is also a subsequence of $\vr$, the result follows.
\end{proof}

Given two elements $s,t \in S$, let $m(s,t)$ denote the order of $st$ in $W$.   Given any
$q$-tuple $\vs = (s_1, \ldots, s_q)$ of elements of $S$, let $\vs(s,t)$ denote the sub-tuple of $\vs$ formed by the occurrences of $s$ and
$t$.  For example, if $\vs = (s,t,u,s,u,s,v,s,t)$ then $\vs(s,t) = (s,t,s,s,s,t)$.

Given $w \in W$ of length $q$, let $V_w$ denote the set of all $q$-tuples $\vs$ which are
reduced expressions
for $w$.  Form a graph $G_w$ with vertex set $V_w$, such that $\vs = (s_1, \ldots, s_q), \vt = (t_1, \ldots, t_q) \in V_w$ are joined by an edge if there are elements
$s,t \in S$ and a sequence of indices $i, i+1, i+2, \ldots, i+ m(s,t)$ such that $(s_{i+1}, s_{i+2}, \ldots, s_{i+m(s,t)}) = (s,t, \ldots )$
and $(t_{i+1}, t_{i+2}, \ldots, t_{i+m(s,t)}) = (t,s, \ldots )$ (we will say that this edge corresponds to the braid relation between $s$ and $t$).

An element $w \in W$ is called fully commutative if any reduced
expression for $w$ can be obtained from any other by using only the relations $st = ts$ where
$s$ and $t$ are commuting elements of $S$.  Suppose that $w$ is fully commutative and $m(s,t) \ge 3$
(that is, $s$ and $t$ do not commute).
Then as observed by Stembridge \cite{Ste:96},
there is no edge in the graph corresponding to the braid relation between
$s$ and $t$.   Stembridge also observed that this implies that if $\vs$ and $\vt$ are joined by an edge, then $\vs(s,t) = \vt(s,t)$, so since the graph is connected, $\vs(s,t) = \vt(s,t)$ for any two elements $\vs$ and $\vt$ of $V_w$.  Write
$w(s,t) = \vs(s,t)$ where $\vs$ is any element of $V_w$.

Observe that $w(s,t)$ can have repeated elements.  For example, in type $B_3$, if $w=s_2 s_3 s_2 s_1$ (which is fully
commutative), then $w(s_1,s_2) = (s_2, s_2, s_1)$.

\begin{Lem}\label{l.reflectionsametimes}
 Let $w$ be fully commutative, and let $\vs$, $\vt\in V_w$. Then any $s\in S$ occurs the same number of times in $\vs$ as in $\vt$.
\end{Lem}
\begin{proof}
If two elements of $V_w$ are connected by an edge, then they differ only by the interchange of two elements of $S$. Therefore $s$ must occur the same number of times in both elements.  Since the graph $G_w$ is connected, the result follows.
\end{proof}

\begin{Prop} \label{p.statistic2}
Suppose $\vs = (s_1, \ldots, s_q)$ satisfies $H_{\vs} = H_w$, where $w$ is fully commutative,
and suppose that $q > l = l(w)$.   Then there exist $i<j$ such that $s_i=s_j$ and $s_i$
commutes with $s_k$ for every $k$, $i<k<j$.
\end{Prop}

\begin{proof}
Since $q>l$, there is some index
$j$ such that
\begin{equation} \label{e.statistic}
s_1 < s_1 s_2 < \cdots < s_1 s_2 \ldots s_{j-1}>s_1 s_2 \ldots s_{j-1} s_j.
\end{equation}
Let $s = s_j$ and $u =  s_1 s_2 \cdots s_{j-1}$; then
$\vs' = (s_1, \ldots, s_{j-1})$ is a reduced expression for $u$.  By Proposition \ref{p.weakorder},
$u \le_R w$, so $u$ is fully commutative (see \cite{Ste:96}).  Because $us<s$, there
is a reduced expression for $u$ which ends in $s$.  Therefore any reduced expression for $u$ must
have at least one term equal to $s$.
In particular this holds for the reduced expression $\vs'$.
Let $i$ be the largest integer with $1 \le i \le j-1$ satisfying
$s_i=s$.  It suffices
to show that $s_k$ commutes with $s$ for all $k$ with $i<k<j$.  Suppose this fails; then
$t = s_k$ does not commute with $s$ for some $k$ with $i<k<j$.
We have chosen $i$ so that $s$ is not an element of the set $\{s_{k+1} , s_{k+2}, \ldots,s_{ j-1} \}$.
Therefore $u(s,t) = \vs'(s,t)$ ends in $t$.  On the other hand, $u(s,t) = s$ since $u$ has a reduced expression
ending in $s$.  This is a contradiction.  We conclude that $s_k$ commutes with $s$ for all $k$ with $i<k<j$, as
desired.
\end{proof}

\section{Applications to the Grassmannian} \label{s.Grassmannian}

Let $G$ be $\SL_n(\C)$, $P_d$ the maximal parabolic subgroup of $G$ corresponding to the simple root $\ga_d$, $B^-$ the Borel subgroup of lower triangular matrices in $G$, and $T$ the group of diagonal matrices in $G$.  The Weyl group $W=N_G(T)/T$ is isomorphic to $S_n$, the permutation group on $n$ elements. The Weyl group $W_{P_d}$ of $P_d$ is isomorphic to $S_d\times S_{n-d}$, and the set $W^{P_d}$ of minimal length coset representatives of $W/W_{P_d}$ consists of the permutations $w=(w_1,\ldots,w_n)$ such that $w_1<\cdots<w_d$ and $w_{d+1}<\cdots<w_n$.
The coset space $G/P_d$ is identified with $\Grass(d,n)$, the Grassmannian variety of $d$-dimensional complex subspaces of $\C^n$. It is an irreducible projective variety of complex dimension $d(n-d)$. The cosets $wP_d$, $w\in W^{P_d}$, are precisely the $T$-fixed points of $G/P_d$. By abuse of notation, we often denote $wP_d$ by $w$. The Schubert variety $X^w$ is by definition $\bar{B^-wP_d}\subseteq G/P_d$. It is an irreducible projective variety of codimension $l(w)$. For $v,w\in W^{P_d}$, $v\in X^w$ if and only if $v\geq w$ in the Bruhat order.

In this section, we give formulas for $i_v^*[\co_{X^w}]$ (Theorems \ref{t.ktheory_eyd} and \ref{t.ktheory_svt}), as well as the Hilbert series, Hilbert polynomial, and multiplicity of $X^w$ at $v$ (Section \ref{ss.hilbseries_An}). These are reformulations of Theorem \ref{t.pullback} and Corollary \ref{c.cominusculemult}, expressed in terms of indexing sets which are specific to the combinatorics of the the symmetric group and the Grassmannian. Namely, our indexing sets are {\it excited Young diagrams} and {\it set-valued tableaux}.
The term excited Young diagram is due to Ikeda and Naruse \cite{IkNa:09}; in \cite{Kre:05} this is called a subset of a Young diagram.
In this paper we have modified the definition of excited Young diagram for our applications
to $K$-theory;
the earlier definition corresponds to our {\it reduced} excited Young diagram.
Reduced excited Young diagrams were discovered independently by Kreiman \cite{Kre:05} and Ikeda and Naruse \cite{IkNa:09}.   A related version of excited Young diagram introduced in \cite{IkNa:11} is discussed in Section \ref{ss.previous}.

Our formula for $i_v^*[\co_{X^w}]$, expressed in terms of set-valued tableaux, was obtained in \cite{Kre:05}. The derivation of the formula there relies on an equivariant Gr\"{o}bner degeneration of $X^w$ in a neighborhood of $v$ to a union of coordinate subspaces; this degeneration is due to \cite{KoRa:03}, \cite{Kre:03}, \cite{KrLa:04}, and \cite{Kre:08}. In \cite{Kre:05}, computing $i_v^*[\co_{X^w}]$ involves cataloging the weights of the intersections of these coordinate subspaces. Each coordinate subspace is expressed as $V(m_T)$, where $m_T$ is a monomial indexed by the entries of a Young tableau $T$. The intersection $V(m_{T_1})\cap\cdots \cap V(m_{T_k})$ is then equal to $V(m_{T_1},\ldots,m_{T_k})=V(m_{T_1\cup\cdots\cup T_k})$,
where $T_1\cup\cdots\cup T_k$ is the set-valued tableau whose entry in each box is equal to the union of the entries of $T_1,\ldots,T_k$ in the same box. In this way set-valued tableaux arise naturally from this approach.

Our approach is different. In both this section and the next one, our methods are modeled on those of Ikeda and Naruse \cite{IkNa:09}. We generalize their arguments from reduced excited Young diagrams to excited Young diagrams, and correspondingly from nil-Coxeter algebras to 0-Hecke algebras. In several places, we use their results directly.
Whereas set-valued tableaux are more suitable for the methods used in \cite{Kre:05}, excited Young diagrams are more suitable for the methods used here.

\subsection{Permutations, partitions, and Young diagrams}\label{ss.partitions_An}
A \textbf{partition} is a sequence of integers $\gl=(\gl_1,\ldots,\gl_d)$
such that $\gl_1\geq\cdots\geq\gl_d\geq 0$. Let $\cp_{d,n-d}$ denote the set of partitions such that $\gl_1\leq n-d$. A \textbf{Young diagram} is a set of boxes arranged in a left justified array, such that the row lengths weakly decrease from top to bottom. To any partition $\gl$ we associate the Young diagram $D_\gl$ whose $i$-th row contains $\gl_i$ boxes.

\begin{example} Let $d=5$, $n=11$, and $\gl=(4,4,2,1)\in\cp_{d,n-d}$. The Young diagram $D_\gl$ fits inside a $d\times (n-d)$ rectangle.
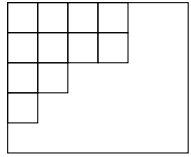
\begin{center}
\begin{tikzpicture}[scale=.4]
\draw (.5,.5) -- ++(6,0) -- ++(0,5) -- ++(-6,0) -- cycle;
\foreach \pos in
      { {(1,5)},{(2,5)},{(3,5)},{(4,5)},
        {(1,4)},{(2,4)},{(3,4)},{(4,4)},
        {(1,3)}, {(2,3)},
        {(1,2)}
      }
          {\draw \pos +(-.5,-.5) rectangle ++(.5,.5);}
\end{tikzpicture}
\end{center}
\end{example}

The map $W^{P_d}\to \cp_{d,n}$ given by $v\mapsto \gl_v$, where
\begin{equation}\label{e.lambdav}
(\gl_v)_i=v_{d+1-i}-(d+1-i), i=1,\ldots,d,
\end{equation}
is a bijection.
Thus $W^{P_d}$, $\cp_{d,n}$, and the set of Young diagrams which fit inside a $d\times (n-d)$ rectangle can all be identified; each parametrizes the $T$-fixed points of $G/P_d$. We now record several properties of these sets and relationships among them.

An \textbf{inversion of a permutation} $v=\{v_1,\ldots,v_n\}\in S_n$ is a pair $(i,j)$ for which $i<j$ and $v_i>v_j$. The number of inversions of $v$ equals $l(v)$.

Assume that $v\in W^{P_d}$. If $(i,j)$
is an inversion, then $i\in \{1,\ldots,d\}$ and $j\in\{d+1,\ldots,n\}$. For $i\in\{1,\ldots,d\}$ and $j\in\{d+1,\ldots,n\}$, define $l_i(v)=\#\{k\in \{d+1,\ldots,n\}\mid v_i>v_k\}$ and $l^j(v)=\#\{k\in \{1,\ldots,d\}\mid v_k>v_j\}$.
Then
\begin{equation}\label{e.lenghtcomps}
    l_i(v)=v_i-i, \qquad
    l^j(v)=j-v_j,
\end{equation}
and
\begin{equation}\label{e.lengthfromcomps}
    l(v)=\sum\limits_{i=1}^d l_i(v)=\sum\limits_{j=d+1}^n l^j(v).
\end{equation}
Equations \eqref{e.lambdav} and \eqref{e.lenghtcomps} imply
\begin{equation}\label{e.partitionfromcomps}
    (\gl_v)_i=l_{d+1-i}(v).
\end{equation}
The length $|\gl|$ of a partition $\gl$ is defined to be $\gl_1+\cdots+\gl_d$.
\begin{Lem}\label{l.lengthofperm}
For $v\in W^{P_d}$, $|\gl_v|=l(v)$.
\end{Lem}
\begin{proof}
By \eqref{e.partitionfromcomps} and  \eqref{e.lengthfromcomps}, $|\gl_v|=\sum_{i=1}^d (\gl_v)_i=\sum_{i=1}^d l_{d+1-i}(v)=\sum_{i=1}^d l_i(v)=l(v)$.
\end{proof}

For $\gl\in \cp_{d,n}$, define the partition $\gl^t=(\gl^t_1,\ldots,\gl^t_{n-d})\in P_{n-d,n  }$ by $\gl^t_j=\#\{i\in \{1,\ldots,d\}\mid \gl_i\geq j\}$, $j=1,\ldots,n-d$. Then $\gl^t_j$ is the number of boxes in the $j$-th column of $D_\gl$. We call $\gl^t$ the \textbf{transpose} of $\gl$.
\begin{Lem}\label{l.trans_partition}
For $v\in W^{P_d}$, $(\gl_v)^t_j=(d+j)-v_{d+j}$, $j=1,\ldots,n-d$.
\end{Lem}
\begin{proof}
$(\gl_v)^t_j
    =\#\{i\in \{1,\ldots,d\}\mid (\gl_v)_i\geq j\}
    =\#\{i\in \{1,\ldots,d\}\mid l_{d+1-i}(v)\geq j\}
    =\#\{i\in \{1,\ldots,d\}\mid l_i(v)\geq j\}
    =\#\{i\in \{1,\ldots,d\}\mid v_i>v_{d+j}\}
    =l^{d+j}(v)
    =(d+j)-v_{d+j}.
$
\end{proof}

The following lemma will be needed in Section \ref{s.orthosymplectic}. We say that a partition $\gl$ is \textbf{symmetric} if $\gl^t=\gl$, and that a Young diagram is symmetric if the length of column $j$ equals the length of row $j$ for all $j$. Clearly $\gl$ is symmetric if and only if $D_\gl$ is symmetric. We say that a permutation $v=(v_1,\ldots,v_{2n})$ is \textbf{symmetric} if $v_{\bar{\imath}}=\bar{v_i}$, $i=1,\ldots,n$, where $\bar{x}=2n+1-x$.
\begin{Lem}\label{l.symmsymm}
The permutation $v=(v_1,\ldots,v_{2n})\in W^{P_{n}}$ is symmetric if and only if $\gl_{v}$ is symmetric.
\end{Lem}
\begin{proof} We have
{\allowdisplaybreaks
\begin{align*}
v_{\bar{\imath}}=\bar{v_i}, i=1,\ldots, n
    &\iff v_{\bar{\imath}}+v_i=2n+1,i=1,\ldots,n\\
    &\iff v_{n+j}+v_{n+1-j}=2n+1, j=1,\ldots, n\\
    &\iff (n+j)-v_{n+j}=v_{n+1-j}-(n+1-j), j=1,\ldots, n\\
    &\iff (\gl_v)^t_j=(\gl_v)_j, j=1,\ldots, n
\end{align*}
}
\end{proof}

\subsection{Restriction formula in terms of excited Young diagrams}
\label{ss.eyd}
  Our convention is to number the rows of a Young diagram from top to bottom and the columns from left to right. The box in row $i$ and column $j$ is denoted by $(i,j)$.

 Suppose that  $C$ is any subset of $D_\gl$, $(i,j)\in C$, and $(i+1,j),(i,j+1),(i+1,j+1)\in D_\gl\setminus C$. Define an \textbf{excitation of type 1} to be an operation which replaces $C$ by $C'=C\setminus (i,j)\cup (i+1,j+1)$, and denote such an operation by
 $C\exco C'$. Define an \textbf{excitation of type 2} to be an operation which replaces $C$ by $C''=C\cup (i+1,j+1)$, and denote such an operation by $C\exct C''$.

\begin{figure}[height=30em]
    \centering
    \begin{tikzpicture}[
    level distance = 4cm,sibling distance = 2cm, grow = east,
    emph/.style={edge from parent/.style={draw,double distance =
    1.5pt,>=latex,->}},
    norm/.style={edge from parent/.style={draw,>=latex,->}}
]
\node [rectangle]  {%
    \begin{tikzpicture}[scale=.4]
    \foreach \pos in
    {{(1,4)},{(3,4)},{(1,3)},{(1,2)},{(4,2)}}
        {\draw[fill=blue!30] \pos +(-.5,-.5) rectangle ++(.5,.5);}
     \foreach \pos in
      { {(1,4)},{(2,4)},{(3,4)},{(4,4)},{(5,4)},
        {(1,3)},{(2,3)},{(3,3)},{(4,3)},{(5,3)},
        {(1,2)}, {(2,2)},{(3,2)},{(4,2)},
        {(1,1)},{(2,1)},{(3,1)}
      }
          {\draw \pos +(-.5,-.5) rectangle ++(.5,.5);}
    \end{tikzpicture}
}
    child[norm] {node [rectangle]
    {
        \begin{tikzpicture}[scale=.4]
        \foreach \pos in
        {{(1,4)},{(4,3)},{(1,3)},{(1,2)},{(4,2)}}
            {\draw[fill=blue!30] \pos +(-.5,-.5) rectangle ++(.5,.5);}
         \foreach \pos in
          { {(1,4)},{(2,4)},{(3,4)},{(4,4)},{(5,4)},
            {(1,3)},{(2,3)},{(3,3)},{(4,3)},{(5,3)},
            {(1,2)}, {(2,2)},{(3,2)},{(4,2)},
            {(1,1)},{(2,1)},{(3,1)}
          }
              {\draw \pos +(-.5,-.5) rectangle ++(.5,.5);}
        \end{tikzpicture}
        }
    }
    child[emph] {node [rectangle]
    {
        \begin{tikzpicture}[scale=.4]
        \foreach \pos in
        {{(1,4)},{(3,4)},{(4,3)},{(1,3)},{(1,2)},{(4,2)}}
            {\draw[fill=blue!30] \pos +(-.5,-.5) rectangle ++(.5,.5);}
         \foreach \pos in
          { {(1,4)},{(2,4)},{(3,4)},{(4,4)},{(5,4)},
            {(1,3)},{(2,3)},{(3,3)},{(4,3)},{(5,3)},
            {(1,2)}, {(2,2)},{(3,2)},{(4,2)},
            {(1,1)},{(2,1)},{(3,1)}
          }
              {\draw \pos +(-.5,-.5) rectangle ++(.5,.5);}
        \end{tikzpicture}
        }
    }
    child[emph] {node [rectangle]
    {
        \begin{tikzpicture}[scale=.4]
        \foreach \pos in
        {{(1,4)},{(3,4)},{(1,3)},{(1,2)},{(2,1)},{(4,2)}}
            {\draw[fill=blue!30] \pos +(-.5,-.5) rectangle ++(.5,.5);}
         \foreach \pos in
          { {(1,4)},{(2,4)},{(3,4)},{(4,4)},{(5,4)},
            {(1,3)},{(2,3)},{(3,3)},{(4,3)},{(5,3)},
            {(1,2)}, {(2,2)},{(3,2)},{(4,2)},
            {(1,1)},{(2,1)},{(3,1)}
          }
              {\draw \pos +(-.5,-.5) rectangle ++(.5,.5);}
        \end{tikzpicture}
        }
    }
    child[norm] {node [rectangle]
    {
        \begin{tikzpicture}[scale=.4]
        \foreach \pos in
        {{(1,4)},{(3,4)},{(1,3)},{(2,1)},{(4,2)}}
            {\draw[fill=blue!30] \pos +(-.5,-.5) rectangle ++(.5,.5);}
         \foreach \pos in
          { {(1,4)},{(2,4)},{(3,4)},{(4,4)},{(5,4)},
            {(1,3)},{(2,3)},{(3,3)},{(4,3)},{(5,3)},
            {(1,2)}, {(2,2)},{(3,2)},{(4,2)},
            {(1,1)},{(2,1)},{(3,1)}
          }
              {\draw \pos +(-.5,-.5) rectangle ++(.5,.5);}
        \end{tikzpicture}
        }
    };
\end{tikzpicture}
    \caption[Figure]{\label{f.ladder} The Young diagram $D_\gl$, $\gl=(5,5,4,3)$, appears on the left. The five shaded boxes of $D_\gl$ form a subset $C$ of $D_\gl$. The top arrow is the type 1 excitation
    $C\exco C\setminus (3,1)\cup (4,2)$.  The next arrow is the type 2 excitation $C\exct C\cup (4,2)$.
    The two other excitations which can be applied to $C$ are also shown.}
\end{figure}
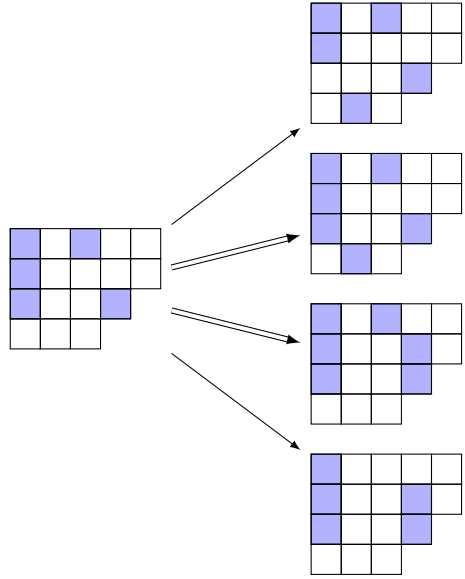

Let $\gl,\gm\in I_{d,n}$. If $\gl_i\leq \gm_i$, $i=1,\ldots,d$, then the map $(i,j)\mapsto (i,j)$ embeds $D_\gl$ as a subset of $D_\gm$. An \textbf{excited Young diagram} of $D_\gl$ in $D_\gm$ is defined to be a subset of $D_\gm$ which can be obtained by applying a sequence of excitations to the subset $D_\gl$. An excited Young diagram is said to be \textbf{reduced} if it can be obtained from $D_\gl$ by applying only type 1 excitations. Denote the set of excited Young diagrams of $D_\gl$ in $D_\gm$ by $\ce_\gl(\gm)$, and the set of reduced excited Young diagrams by $\ce^{\text{red}}_\gl(\gm)$ (see Figure \ref{f.eyd_all}).

\begin{figure}[height=30em]
    \centering
    \input{pic.eyd_all.tex}
    \caption{\label{f.eyd_all} The excited Young diagrams $\ce_\gl(\gm)$, for $\gm=(4,4,2)$, $\gl=(2,1)$. The excited Young diagrams enclosed by the dashed line comprise $\ce^{\text{red}}_\gl(\gm)$.}
\end{figure}

\begin{Thm}\label{t.ktheory_eyd} Let $w\leq v\in W^{P_d}$, and let $\gl=\gl_w$, $\gm=\gl_v$ be the corresponding partitions. Then
$\displaystyle
    i_v^*[\co_{X^w}]=(-1)^{l(w)}\sum_{C\in\ce_{\gl}(\gm)}\prod_{(i,j)\in C}\left(e^{\gre_{v_{d+1-i}}-\gre_{v_{d+j}}}-1\right)
$
\end{Thm}

\begin{example}\label{ex.restriction_A}
Let $n=7$, $d=3$. For $v=\{4,6,7,1,2,3,5\},w=\{1,3,5,2,4,6,7\}\in W^{P_d}$,  $\gm=\gl_v=(4,4,3)$ and $\gl=\gl_w=(2,1)$. Thus $l(w)=|\gl_w|=3$. The set of excited Young diagrams $\ce_{\gl}(\gm)$ appears in Figure \ref{f.eyd_all}.
 For any $C\in \ce_{\gl}(\gm)$ and $(i,j)\in C$,
a simple method of finding the
indices $v_{d+j}$ and $v_{d+1-i}$ of Theorem \ref{t.ktheory_eyd} is to
label the rows and columns of $C$ with the entries of $v$ as indicated below. Then $v_{d+1-i}$ and $v_{d+j}$ are the row $i$ and column $j$ labels respectively. For example, for
\begin{equation*}
C\ =\
\begin{tikzpicture}[scale=.6,every node/.style={scale=1},
    baseline={([yshift=-.5ex]current bounding box.center)}]
    \foreach \pos in
    {{(1,3)},{(2,3)},{(2,1)}}
        {\draw[fill=blue!30] \pos +(-.5,-.5) rectangle ++(.5,.5);}
     \foreach \pos in
      { {(1,3)},{(2,3)},{(3,3)},{(4,3)},
        {(1,2)},{(2,2)},{(3,2)},{(4,2)},
        {(1,1)}, {(2,1)},{(3,1)}
      }
          {\draw \pos +(-.5,-.5) rectangle ++(.5,.5);}
    \draw  (0,1) node{$4$};  \draw  (0,2) node{$6$};  \draw  (0,3) node{$7$};
    \draw  (1,4) node{$1$}; \draw  (2,4) node{$2$}; \draw  (3,4) node{$3$};
    \draw  (4,4) node{$5$};
\end{tikzpicture}\ \in\ce_{\gl_w}(\gl_v),
\end{equation*}

\noindent we have
$
\prod_{(i,j)\in C}\left(e^{\gre_{v_{d+1-i}}-\gre_{v_{d+j}}}-1\right)
=\left(e^{\gre_7-\gre_1}-1\right)\left(e^{\gre_7-\gre_2}-1\right)
\left(e^{\gre_4-\gre_2}-1\right).
$

By Theorem \ref{t.ktheory_eyd},
{\allowdisplaybreaks
\begin{align*}
    i_v^*[\co_{X^w}] &=
    -\left(e^{\gre_7-\gre_1}-1\right)\left(e^{\gre_7-\gre_2}-1\right)
    \left(e^{\gre_6-\gre_1}-1\right)
    -\left(e^{\gre_7-\gre_1}-1\right)\left(e^{\gre_7-\gre_2}-1\right)
        \left(e^{\gre_4-\gre_2}-1\right)\\
    &- \left(e^{\gre_7-\gre_1}-1\right)\left(e^{\gre_6-\gre_1}-1\right)
    \left(e^{\gre_6-\gre_3}-1\right)
    -\left(e^{\gre_7-\gre_1}-1\right)\left(e^{\gre_4-\gre_2}-1\right)
        \left(e^{\gre_6-\gre_3}-1\right)\\
    &-\left(e^{\gre_6-\gre_2}-1\right)\left(e^{\gre_4-\gre_2}-1\right)
        \left(e^{\gre_6-\gre_3}-1\right)\\
    &-\left(e^{\gre_7-\gre_1}-1\right)\left(e^{\gre_7-\gre_2}-1\right)
    \left(e^{\gre_6-\gre_3}-1\right)\left(e^{\gre_6-\gre_1}-1\right)\\
    &-\left(e^{\gre_7-\gre_1}-1\right)\left(e^{\gre_7-\gre_2}-1\right)
    \left(e^{\gre_6-\gre_1}-1\right)\left(e^{\gre_4-\gre_2}-1\right)\\
    &-\left(e^{\gre_7-\gre_1}-1\right)\left(e^{\gre_7-\gre_2}-1\right)
    \left(e^{\gre_4-\gre_2}-1\right)\left(e^{\gre_6-\gre_3}-1\right)\\
    &-\left(e^{\gre_7-\gre_1}-1\right)\left(e^{\gre_6-\gre_1}-1\right)
    \left(e^{\gre_4-\gre_2}-1\right)\left(e^{\gre_6-\gre_3}-1\right)\\
    &-\left(e^{\gre_7-\gre_1}-1\right)\left(e^{\gre_6-\gre_2}-1\right)
    \left(e^{\gre_4-\gre_2}-1\right)\left(e^{\gre_6-\gre_3}-1\right)\\
    &-\left(e^{\gre_7-\gre_1}-1\right)\left(e^{\gre_7-\gre_2}-1\right)
    \left(e^{\gre_6-\gre_1}-1\right)\left(e^{\gre_6-\gre_3}-1\right)
    \left(e^{\gre_4-\gre_2}-1\right).
\end{align*}
}
\end{example}

Theorem \ref{t.ktheory_eyd} is a reformulation of Theorem \ref{t.pullback}, in which the indexing set $T(w,\vs)$ and integer $r(c)$ of the latter theorem are replaced by expressions involving excited Young diagrams. These replacements are given in Proposition \ref{p.subsequences}.

Let $w,v\in W^{P_d}$, and let $\gl=\gl_w$ and $\gm=\gl_v$ be the corresponding partitions.
 Fill in each box $(i,j)$ of $D_{\gm}$ with the simple reflection $s_{d+j-i}$, thus obtaining a reflection-valued tableau denoted by $T_{\gm}$.
Then
$v=s_{i_1}s_{i_2}\cdots s_{i_l}$,
where $s_{i_1},s_{i_2},\ldots,s_{i_l}$ are the entries of $T_{\gm}$ read from right to left, beginning with the bottom row, then the next row up, etc. Since $l=|D_{\gm}|=|\gm|=l(v)$, this decomposition is reduced. To any subset $C\subseteq D_{\gm}$, form the subsequence $\vs_C=(s_{j_1},\ldots,s_{j_q})$ of $(s_{i_1},\ldots,s_{i_l})$ whose entries lie in the set $C$ of boxes of $T_{\gm}$. If $C$ and $D$ are different subsets of $D_\gm$, then we regard $\vs_C$ and $\vs_D$ as different subsequences of $(s_{i_1},\ldots,s_{i_l})$, even if they have the same entries.

\begin{example}\label{ex.reduceddecomp} Let $n=8$, $d=4$.
 For $v=\{3,5,6,8,1,2,4,7\}\in W^{P_d}$, $\gm=\gl_v=(4,3,3,2)$,
\begin{align*}
T_{\gm}&=
\begin{tikzpicture}[
    scale=.4,
    every node/.style={scale=.8},
    baseline={([yshift=-.5ex]current bounding box.center)}]
         \foreach \pos / \label in
          { {(1,4)}/{s_4},{(2,4)}/{s_5},{(3,4)}/{s_6},{(4,4)}/{s_7},
            {(1,3)}/{s_3},{(2,3)}/{s_4},{(3,3)}/{s_5},
            {(1,2)}/{s_2},{(2,2)}/{s_3},{(3,2)}/{s_4},
            {(1,1)}/{s_1},{(2,1)}/{s_2}
          }
              {
                \draw \pos +(-.5,-.5) rectangle ++(.5,.5);
                \draw \pos node{$\label$};
              }
\end{tikzpicture}
\intertext{and $s_2s_1s_4s_3s_2s_5s_4s_3s_7s_6s_5s_4$ is a reduced decomposition for $v$. 
For}
C&=
\begin{tikzpicture}[scale=.4,every node/.style={scale=.8},baseline={([yshift=-.5ex]current bounding box.center)}]
        \foreach \pos in
        {{(1,2)},{(1,3)},{(2,3)},{(3,2)},{(4,4)}}
            {\draw[fill=green!90] \pos +(-.5,-.5) rectangle ++(.5,.5);}
         \foreach \pos in
          { {(1,4)},{(2,4)},{(3,4)},{(4,4)},
            {(1,3)},{(2,3)},{(3,3)},
            {(1,2)},{(2,2)},{(3,2)},
            {(1,1)},{(2,1)}
          }
              {
                \draw \pos +(-.5,-.5) rectangle ++(.5,.5);
              }
\end{tikzpicture}
\subset D_{\gm},
\end{align*}
we have $\vs_C=(s_4,s_2,s_4,s_3,s_7)$.
\end{example}

\begin{Prop}\label{p.subsequences}
Denote the reduced  decomposition  $(s_{i_1},\cdots, s_{i_l})$ for $v$ obtained above by $\vs_v$. By definition,
$T(w,\vs_v)=\{\vs_C\mid C\subseteq D_\gm, H_{\vs_C}=H_w\}$. We have
\begin{itemize}
    \item[(i)] $T(w,\vs_v)=\{\vs_C\mid C\in\ce_{\gl}(\gm)\}$.

    \item[(ii)] Let $(i,j)$ be the box of $T_{\gm}$ containing $s_{i_c}$. Then
    $r(c)=\gre_{v_{d+j}}-\gre_{v_{d+1-i}}$.
\end{itemize}
\end{Prop}

The reduced decomposition $\vs_v$ can be deduced from a more general method of producing reduced decompositions of arbitrary permutations (see
\cite[Remark 2.1.9]{Man:01}). The decomposition here also appears in Ikeda and Naruse \cite{IkNa:09}.
Proposition \ref{p.subsequences}(ii) is due to Ikeda and Naruse \cite{IkNa:09}. A version of
 Proposition \ref{p.subsequences}(i) which involves the nil-Coxeter alegra and (what we call) reduced excited Young diagrams is also proved in \cite{IkNa:09}.

For (co)Grassmannian permutations $w$ and $v$, the notion of a {\it pipe dream for $w$ on $D(v)$}, introduced by Woo and Yong \cite{WoYo:12} (see also \cite{LiYo:12}), is closely related to that of an excited Young diagram of $D_\gl$ in $D_\gm$. In our language, a {\it pipe dream for $w$ on $D(v)$} is equal to a subset $C$ of $D_\gm$ such that $H_{\vs_C}=H_w$. By Proposition \ref{p.subsequences}(i), the set of all such pipe dreams is equal to $\ce_\gl(\gm)$. General pipe dreams have been studied by \cite{BeBi:93}, \cite{FoKi:96}, \cite{KnMi:04}, \cite{KnMi:05}.  The excited Young diagrams introduced in \cite{IkNa:11} are also related to those of this paper (see Section \ref{ss.previous}).

 \subsubsection{Proof of Proposition \ref{p.subsequences}(i)}\label{ss.proof_p.subsequences}

Let $C$ be a subset of $D_\gm$.
Proposition \ref{p.subsequences}(i) is equivalent to
\begin{equation}\label{e.subsequences}
    H_{\vs_C}=H_w\text{ if and only if }C\in\ce_\gl(\gm).
\end{equation}
The proof of the reverse implication is fairly straightforward (see Lemma \ref{l.containmentone}). The proof of the forward implication (Lemma \ref{l.containmenttwo}) is by induction on $e_1(C)$ and $e_2(C)$, defined below.
Lemma \ref{l.hecketoeyd} helps us to translate from excited Young diagrams to 0-Hecke algebras.

\begin{Def}\label{d.energy}
If $H_{\vs_C}=H_w$, then define
\begin{enumerate}
    \item $|C|=$ number of boxes of $C$

    \item $e_1(C)=(1/2)(\sum_{(i,j)\in C}(i+j)-\sum_{(i,j)\in D_\gl}(i+j))$

    \item $e_2(C)=|C|-|D_{\gl}|=|C|-|\gl|$
\end{enumerate}
We call $e_1(C)$ and $e_2(C)$ the \textbf{type 1 energy} and \textbf{type 2 energy} of $C$ respectively.
\end{Def}

\begin{Lem}\label{l.hecketoeyd} If $H_{\vs_C}=H_w$, then
\begin{enumerate}
    \item $|C|=l(\vs_C)$.

    \item $e_2(C)=e(\vs_C)$.

    \item $e_2(C)\geq 0$, and $e_2(C)=0$ if and only if $\vs_C$ is a reduced expression for $w$.
\end{enumerate}
\end{Lem}
\begin{proof}
(i) Clear from the definition of $\vs_C$.

\noindent (ii) $e_2(C)=|C|-|\gl|=l(\vs_C)-l(w)=e(\vs_C)$, where the last equality is the definition of $e(\vs_C)$. 

\noindent (iii) By (ii), this can be restated:
If $H_{\vs_C}=H_w$, then $l(\vs_C)-l(w)\geq 0$, and $l(\vs_C)-l(w)=0$ if and only if $\vs_C$ is a reduced expression for $w$. This follows immediately from Lemma \ref{l.minlength}.
\end{proof}

\begin{Rem}
If $C\in\ce_\gl(\gm)$, then one can give interpretations of $e_1(C)$ and $e_2(C)$ in terms of excitations; these interpretations are not required for the sequel. An excitation of type 1 applied to $C$  does not alter $|C|$, whereas an excitation of type 2 increases it by 1.  Thus $e_2(C)$ is the number of excitations of type 2 which must be applied in order to obtain $C$ from $D_{\gl}$.  In particular, $e_2(C)\geq 0$, and $e_2(C)=0$ if and only if $C$ is reduced. In general, the number of type 1 excitations which must be applied in order to obtain $C$ from $D_\gl$ is not defined, since it may be possible to obtain $C$ by two different sequences of excitations which have a different number of type 1 excitations. However, if $C$ is reduced, then this number is defined and is given by $e_1(C)$.
\end{Rem}

\textbf{Diagonal $k$} of $D_\gl$ is defined to be the set of boxes $(i,j)$ such that $j-i=k$.
We say that box $(i,j)$ \textbf{lies on diagonal $k$} if $j-i=k$.
Impose the following total order on the boxes of $D_\gm$: $(i,j)<(k,l)$ if $i>k$ or if $i=k$ and $j>l$.

\begin{Lem}\label{l.modifysubsequence}
Let $\ga'<\ga$ be two boxes of $D_\gm$ which lie on the same diagonal $k$. Suppose that $\ga\in C$, $\ga'\not\in C$, and that $C$ contains no box $\gb$, $\ga'<\gb<\ga$, which lies on diagonal $k-1$, $k$, or $k+1$. 
Then $H_{\vs_{C\,\setminus \ga\,\cup\,\ga'}}=H_{\vs_{C \,\cup\,\ga'}}=H_{\vs_C}$.
\end{Lem}
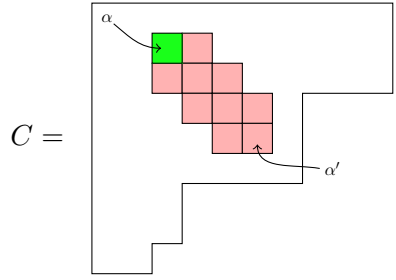
\begin{figure}[h!]
\begin{equation*}
    C=\ \
    \begin{tikzpicture}[xscale=.4,yscale=.4,
    baseline={([yshift=-.5ex]current bounding box.center)},
    every node/.style={scale=.6}]
        \foreach \pos in
        {{(4,8)},{(3,7)},{(4,7)},{(5,7)},{(4,6)},{(5,6)},(6,6),(5,5),(6,5)}
            {\draw[fill=red!30] \pos +(-.5,-.5) rectangle ++(.5,.5);}
        \draw[fill=green!90] (3,8) +(-.5,-.5) rectangle ++(.5,.5);
          \draw (0.5,0.5) -- ++(2,0) -- ++(0,1) -- ++(1,0) -- ++(0,2) -- ++(4,0) -- ++(0,3) -- ++(3,0) -- ++(0,3) -- ++(-10,0)
        -- cycle;
        \node[inner sep=1pt] (src_one) at (1,9) {$\ga$};
        \node (dst_one) at (3,8) {};
        \draw[->] (src_one) to [out=-40,in=180] (dst_one);
        \node[inner sep=3pt] (src_two) at (8.5,4) {$\ga'$};
        \node (dst_two) at (6,5) {};
        \draw[->] (src_two.west) to [out=170,in=-90] (dst_two);
      \end{tikzpicture}
\end{equation*}
\caption{\label{f.genexcitation}$C$ contains box $\ga$ but does not contain any of the red boxes.}
\end{figure}
\begin{proof}
Boxes $\ga'$ and $\ga$ of $T_\gm$ contain the same reflection, namely $s_{k+d}$; let $r=k+d$. To create $\vs_{C\,\setminus \ga\,\cup\,\ga'}$ from $\vs_C$, one moves the reflection $s_r$ of $\vs_C$ which lies in box $\ga$ of $C$ past all of the reflections of $\vs_C$ which lie in boxes $\gb$ of $C$, $\ga'<\gb<\ga$. Since all such reflections lie outside of diagonals $k-1$, $k$, and $k+1$, $s_r$ commutes with them. Thus $H_{\vs_{C\,\setminus \ga\,\cup\,\ga'}}=H_{\vs_C}$. One proves $H_{\vs_{C \,\cup\,\ga'}}=H_{\vs_C}$ similarly.
\end{proof}

\begin{Lem}\label{l.containmentone}
If $C\in\ce_{\gl}(\gm)$, then $H_{\vs_C}=H_w$.
\end{Lem}
\begin{proof}
 Since $C$ is obtained by applying a sequence of excitations to $D_{\gl}$, and $H_w=H_{\vs_{D_{\gl}}}$,  we need only prove that applying a single excitation to a subset $D$ of $D_\gm$ does not alter $H_{\vs_D}$. This is a special case of Lemma \ref{l.modifysubsequence}.
\end{proof}

\begin{Lem}\label{l.containmenttwo} If  $H_{\vs_C}=H_w$, then $C\in \ce_{\gl}(\gm)$.
\end{Lem}
\begin{proof}
We prove this lemma for the following three cases, which increase in generality: first when $e_2(C)=0$ and $e_1(C)=0$, then when $e_2(C)=0$ and $e_1(C)$ is arbitrary, and finally with no restrictions on $C$. Each serves as the base case of an inductive proof for the subsequent more general case.
\vspace{1em}

\noindent \textbf{Case 1}. $e_2(C)=0$ and $e_1(C)=0$.

Since $e_2(C)=0$, $\vs_C$ is a reduced expression for $w$ by Lemma \ref{l.hecketoeyd}(iii). Thus, by Lemma \ref{l.reflectionsametimes}, any reflection must occur the same number of times in $\vs_C$ as in $\vs_{D_{\gl}}$. Thus on each diagonal, $C$ must have the same number of boxes as $D_{\gl}$ does. Since $e_1(C)=0$, $C=D_{\gl}$.

\vspace{2em}

\noindent \textbf{Case 2}. $e_2(C)=0$ and $e_1(C)$ arbitrary.

The proof is by induction on $e_1(C)$. If $e_1(C)=0$, then we are done, by Case 1. Assume that $e_1(C)>0$. Then $C\neq D_{\gl}$. Since $e_2(C)=0$, $\vs_C$ is a reduced expression for $w$ and therefore $C$ must have the same number of boxes as $D_\gl$ does. Thus there must be some box which is contained in $D_{\gl}$ but not in $C$. Let $\ga$ be the maximal such, and let $k$ be the diagonal of $\ga$. As in Case 1, on each diagonal, $C$ must have the same number of boxes as $D_{\gl}$ does. This implies that $C$ must contain a box lying on diagonal $k$ and in a lower row than $\ga$; let $\ga'$ be the maximal such. By maximality of $\ga'$, $C$ contains no box $\gb$, $\ga'<\gb<\ga$, which lies on diagonal $k$. Since $\vs_{C}(s_k,s_{k-1})=\vs_{D_{\gl}}(s_k,s_{k-1})$ (see Section \ref{s.nilHecke}), $C$ contains no box $\gb$, $\ga'<\gb<\ga$, which lies on diagonal $k-1$. Since $\vs_{C}(s_k,s_{k+1})=\vs_{D_{\gl}}(s_k,s_{k+1})$, $C$ contains no box $\gb$, $\ga'<\gb<\ga$, which lies on diagonal $k+1$. 
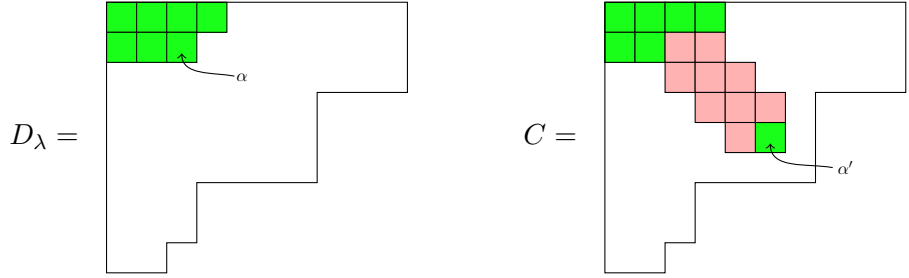
\begin{figure}[h!]
\begin{equation*}
    D_{\gl}=\ \
    \begin{tikzpicture}[xscale=.4,yscale=.4,
    baseline={([yshift=-.5ex]current bounding box.center)},
    every node/.style={scale=.6}]
        \foreach \pos in
        {{(1,9)},{(2,9)},{(3,9)},{(4,9)},
        (1,8),(2,8),(3,8)}
            {\draw[fill=green!90] \pos +(-.5,-.5) rectangle ++(.5,.5);}
          \draw (0.5,0.5) -- ++(2,0) -- ++(0,1) -- ++(1,0) -- ++(0,2) -- ++(4,0) -- ++(0,3) -- ++(3,0) -- ++(0,3) -- ++(-10,0)
        -- cycle;
        \node[inner sep=1pt] (src_one) at (5,7) {$\ga$};
        \node (dst_one) at (3,8) {};
        \draw[->] (src_one) to [out=170,in=-90] (dst_one);
    \end{tikzpicture}
\qquad\qquad
    C=\ \
    \begin{tikzpicture}[xscale=.4,yscale=.4,
    baseline={([yshift=-.5ex]current bounding box.center)},
    every node/.style={scale=.6}]
        \foreach \pos in
        {{(1,9)},{(2,9)},{(3,9)},{(4,9)},
        (1,8),(2,8),(6,5)}
            {\draw[fill=green!90] \pos +(-.5,-.5) rectangle ++(.5,.5);}
        \foreach \pos in
        {{(3,8)},{(4,8)},{(3,7)},{(4,7)},{(5,7)},{(4,6)},{(5,6)},(6,6),(5,5)}
            {\draw[fill=red!30] \pos +(-.5,-.5) rectangle ++(.5,.5);}
          \draw (0.5,0.5) -- ++(2,0) -- ++(0,1) -- ++(1,0) -- ++(0,2) -- ++(4,0) -- ++(0,3) -- ++(3,0) -- ++(0,3) -- ++(-10,0)
        -- cycle;
        \node[inner sep=3pt] (src_two) at (8.5,4) {$\ga'$};
        \node (dst_two) at (6,5) {};
        \draw[->] (src_two.west) to [out=170,in=-90] (dst_two);
      \end{tikzpicture}
\end{equation*}
\caption{\label{f.Dgl_and_C}$D_\gl$ contains the green boxes. $C$ contains the green boxes but not the red boxes.}
\end{figure}

Let $A=C\setminus \ga'\cup \ga$. By Lemma \ref{l.modifysubsequence}, $H_{\vs_A}=H_{\vs_C}=H_w$. Furthermore, $e_1(A)<e_1(C)$. By induction, $A\in\ce_{\gl}(\gm)$. Since $C$ is obtained by applying (type 1) excitations to $A$, $C\in\ce_{\gl}(\gm)$.

\vspace{2em}

\noindent \textbf{Case 3}. No restrictions on $C$.

The proof is by induction on $e_2(C)$. If $e_2(C)=0$, then we are done, by Case 2. Assume that $e_2(C)> 0$. Write $\vs_C=(s_{j_1},\ldots,s_{j_q})$. Since $e_2(C)>0$, $q>l(w)$ by Lemma \ref{l.hecketoeyd}(iii).  By assumption $H_{\vs_C}=H_w$. Proposition \ref{p.statistic2} implies that there exist $a<b$ such that $j_a=j_b$ and $s_{j_a}$ commutes with $s_{j_c}$ for every $c$, $a<c<b$. Let $\ga'$ and $\ga$ be the boxes of $C$ containing $s_{j_a}$ and $s_{j_b}$ respectively, and let $k$ be the diagonal of $\ga$ and $\ga'$. Then $C$ contains no box $\gb$, $\ga'<\gb<\ga$, which lies on diagonal $k-1$, $k$, or $k+1$. 
\begin{figure}[h!]
\begin{equation*}
    C=\ \
    \begin{tikzpicture}[xscale=.4,yscale=.4,
    baseline={([yshift=-.5ex]current bounding box.center)},
    every node/.style={scale=.6}]
        \foreach \pos in
        {(2,8),(5,5)}
            {\draw[fill=green!90] \pos +(-.5,-.5) rectangle ++(.5,.5);}
        \foreach \pos in
        {{(3,8)},{(2,7)},{(3,7)},{(4,7)},{(5,6)},(3,6),{(4,6)},(5,6),(4,5)}
            {\draw[fill=red!30] \pos +(-.5,-.5) rectangle ++(.5,.5);}
%
          \draw (0.5,0.5) -- ++(2,0) -- ++(0,1) -- ++(1,0) -- ++(0,2) -- ++(4,0) -- ++(0,3) -- ++(3,0) -- ++(0,3) -- ++(-10,0)
        -- cycle;
        \node[inner sep=1pt] (src_one) at (-.5,9) {$\ga$};
        \node (dst_one) at (2,8) {};
        \draw[->] (src_one) to [out=-40,in=180] (dst_one);
        \node[inner sep=3pt] (src_two) at (8.5,4) {$\ga'$};
        \node (dst_two) at (5,5) {};
        \draw[->] (src_two.west) to [out=170,in=-90] (dst_two);
      \end{tikzpicture}
\end{equation*}
\caption{\label{f.Case3} $C$ contains boxes $\ga$ and $\ga'$ but contains none of the red boxes.}
\end{figure}
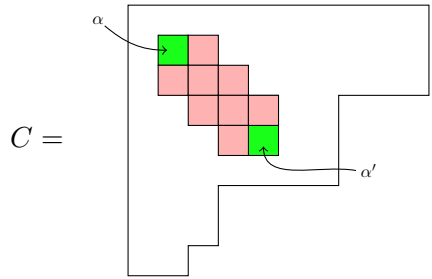
Let $A=C\setminus \ga'$. By Lemma \ref{l.modifysubsequence}, $H_{\vs_A}=H_{\vs_C}=H_w$. Furthermore, $e_2(A)<e_2(C)$. By induction, $A\in\ce_{\gl}(\gm)$. Since $C$ can be obtained by applying (type 2) excitations to $A$, $C\in\ce_{\gl}(\gm)$.
\end{proof}


\subsubsection{Proof of Proposition \ref{p.subsequences}(ii)}

By \eqref{e.lambdav} and Lemma \ref{l.trans_partition},
there are $v_{d+1-i}-(d+1-i)$ boxes in row $i$ and $(d+j)-v_{d+j}$ boxes in column $j$ of $D_{\gl_v}$. Recall that the entry of box $(l,m)$ of $T_{\gm}$ is $s_{d+m-l}$. Thus the entry of the rightmost box of row $i$ of $T_{\gm}$ is $s_{d+(v_{d+1-i}-(d+1-i))-i}=s_{v_{d+1-i}-1}$, and the entry of the  bottom box of column $j$ is $s_{d+j-((d+j)-v_{d+j})}=s_{v_{d+j}}$. Figure \ref{f.r(ic)_proof} shows some of the entries of $T_{\gm}$: $s_x$ and $s_y$ are the reflections in the rightmost box of row $i$ and lowest box of column $j$ respectively. Thus $x=v_{d+1-i}-1$, $y=v_{d+j}$. Also, $p=d+j-i$.
\begin{figure}[h!]
\begin{equation*}
T_{\gm}=
\,
    \begin{tikzpicture}[xscale=.6,yscale=.6,
    baseline={([yshift=-.5ex]current bounding box.center)},
    every node/.style={scale=.7}]
          \draw[gray!7,fill=gray!7] (0.5,0.5) -- ++(9,0) -- ++(0,2) -- ++(1,0) -- ++(0,2) -- ++(2,0) -- ++(0,2) -- ++(-9,0) -- ++(0,-1) -- ++(-3,0) -- cycle;
          \draw[gray!30,fill=gray!30] (2.5,0.5) -- ++(0,6) -- ++(10,0) -- ++(0,-1) -- ++(-9,0) -- ++(0,-5) -- cycle;
         \foreach \pos / \label in
          { {(3,6)}/{},{(4,6)}/{s_{p+1}},{(5,6)}/{s_{p+2}},{(6,6)}/{s_{p+3}},{(12,6)}/{s_{x}},
            {(2,5)}/{s_{p-2}},{(3,5)}/{s_{p-1}},{(4,5)}/{s_{p}},{(5,5)}/{s_{p+1}},{(6,5)}/{s_{p+2}},{(12,5)}/{s_{x-1}},
            {(2,4)}/{s_{p-3}},{(3,4)}/{s_{p-2}},{(4,4)}/{s_{p-1}},{(5,4)}/{s_{p}},{(6,4)}/{s_{p+1}},
            {(2,1)}/{s_{y-1}},{(3,1)}/{s_{y}},{(4,1)}/{s_{y+1}},{(5,1)}/{s_{y+2}}
          }
              {
                \draw \pos +(-.5,-.5) rectangle ++(.5,.5);
                \draw \pos node{$\label$};
              }
         \draw  (3,9) node{$j$};
         \draw  (0,6) node{$i$};
         \draw  (9,6) node{$\cdots$};
         \draw  (3,2.5) node{$\vdots$};
          \draw (0.5,0.5) -- ++(9,0) -- ++(0,2) -- ++(1,0) -- ++(0,2) -- ++(2,0) -- ++(0,4) -- ++(-12,0) -- cycle;
      \end{tikzpicture}
\end{equation*}
\caption{\label{f.r(ic)_proof} Some reflections involved in computing $r(c)$}
\end{figure}
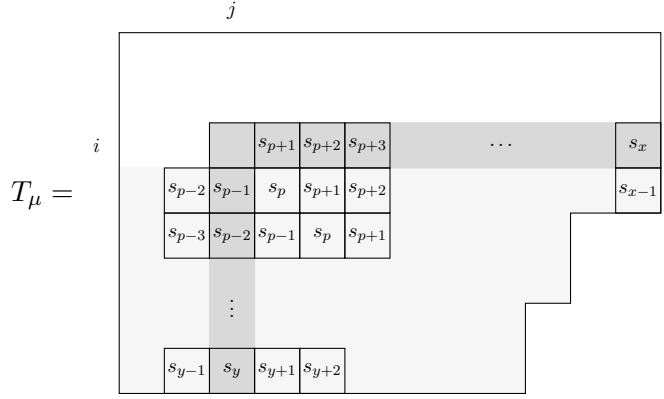

In the expression $r(c)=s_{i_1}s_{i_2}\cdots s_{i_{c-1}}(\ga_{i_c})= s_{i_1}s_{i_2}\cdots s_{p+1}(\gre_p-\gre_{p+1})$, the reflections  $s_{i_j}$ which lie outside darkly shaded boxes in Figure \ref{f.r(ic)_proof} can be removed.
Thus
{\allowdisplaybreaks
\begin{align*}
    r(c)
    &= s_y \cdots s_{p-2}s_{p-1}s_x\cdots s_{p+2} s_{p+1} (\gre_p-\gre_{p+1})\\
    &= s_y\cdots s_{p-2}s_{p-1}s_x\cdots s_{p+2}(\gre_p-\gre_{p+2})\\
    &= s_y\cdots s_{p-2}s_{p-1}s_x\cdots (\gre_p-\gre_{p+3})\\
    &= s_y\cdots s_{p-2}s_{p-1}(\gre_p-\gre_{x+1})\\
    &= s_y\cdots s_{p-2}(\gre_{p-1}-\gre_{x+1})\\
    &= s_y\cdots (\gre_{p-2}-\gre_{x+1})\\
    &= \gre_{y}-\gre_{x+1}
    =\gre_{v_{d+j}}-\gre_{v_{d+1-i}}.
\end{align*}
}

\subsection{Restriction formula in terms of set-valued tableaux}\label{ss.setvalued_tableaux}
Let $\gl$ be a partition. A \textbf{set-valued filling} of $D_\gl$ is a function $T$ which assigns to each box $(i,j)$ of $D_{\gl}$ a nonempty subset $T(i,j)$ of $\{1,\ldots,d\}$.  We call $\gl$ the \textbf{shape} of $T$. We call $(i,j)$ a box of $T$, and refer to an element of $T(i,j)$ as an entry of box $(i,j)$ of $T$, or just an entry of $T$. A set-valued filling $T$ in which each entry of box $(i,j)$ of  $T$ is less than or equal to each entry of box $(i,j+1)$ and strictly less than each entry of box $(i+1,j)$ is said to be \textbf{semistandard}. A  \textbf{set-valued Young tableau}, or just \textbf{set-valued tableau}, is defined to be a semistandard set-valued filling of $D_\gl$. A \textbf{Young tableau}, or just \textbf{tableau}, is a set-valued tableau in which each box contains a single entry.

Let $\gm$ be a partition. We say that a set-valued tableau is \textbf{restricted by $\gm$} if, for any box $(i,j)\in T$ and any entry $x$ of $(i,j)$,
\begin{equation}\label{e.restrictedbyv}
x+j-i\leq \gm(x).
\end{equation}
Denote by $\ct_{\gl}(\gm)$ (resp. $\ct^{\text{red}}_{\gl}(\gm)$) the set of set-valued tableaux (resp. tableaux) of shape $\gl$ which are restricted by $\gm$ (see Figure \ref{f.syt_all}).

\begin{Thm}\label{t.ktheory_svt} Let $w\leq v\in W^{P_d}$, and let $\gl=\gl_w$, $\gm=\gl_v$ be the corresponding partitions. Then
 $$
    i_v^*[\co_{X^w}]=(-1)^{l(w)}\sum_{T\in\ct_{\gl}(\gm)}\prod_{(i,j)\in T}\prod_{x\in T(i,j)}\left(e^{\gre_{v_{d+1-x}}-\gre_{v_{d+x+j-i}}}-1\right)
$$
\end{Thm}
Theorem \ref{t.ktheory_svt}, which  appeared in \cite{Kre:05}, is essentially the same statement as Theorem \ref{t.ktheory_eyd}, except that the indexing set $\ce_{\gl}(\gm)$ has been replaced by $\ct_{\gl}(\gm)$. Equation \eqref{e.f} below defines a map $f$ between these indexing sets, and Proposition \ref{p.fbijective} establishes that $f$ is a bijection. A similar and related bijection appears in \cite{KMY:09}.
The map $f$ restricts to a bijection between $\ce^{\text{red}}_{\gl}(\gm)$ and $\ct^{\text{red}}_{\gl}(\gm)$, which was given in 
\cite{Kre:05} and \cite{WoYo:12}, and is closely related to a bijection in \cite{Kog:00}. A bijection between $\ce^{\text{red}}_{\gl}(\gm)$ and the nonintersecting lattice paths of \cite{Kra:01}, \cite{Kra:05}, \cite{KoRa:03}, \cite{Kre:03}, \cite{KrLa:04}, \cite{Kre:08} was given by \cite{Kre:05}.

Set-valued tableaux of shape $\gl$ restricted by $\gm$ were introduced in \cite{Kre:05}. They also appeared in \cite{WoYo:12}, where they were identified as special types of flagged set-valued tableaux. General flagged set-valued tableaux, which were introduced in \cite{KMY:09}, are set-valued tableaux whose entries in row $i$ are less than or equal to the $i$-th coordinate of a fixed vector $\fb$, which is called the flag. The specific flags utilized in \cite{KMY:09} are similar to ours. However, whereas their flags depend on one vexillary permutation, ours depend on two Grassmannian permutations, namely $w$ and $v$. We point out that \cite{WoYo:12} applies more generally to covexillary permutations.

One difference between set-valued tableaux and excited Young diagrams is that the former are defined locally, whereas the latter are not. One can determine whether a set-valued tableau $T$ of shape $\gl$ lies in $\ct_{\gl}(\gm)$ by checking whether its entries satisfy the inequalities  \eqref{e.restrictedbyv}. In particular, one need only look only at $T$. On the other hand, according to the definition, in order to determine whether a subset $C$ of $D_{\gl}$ lies in $\ce_{\gl}(\gm)$, one must search for a sequence of excitations which when applied to $D_{\gl}$ produces $C$. One can give a local criterion for membership in $\ce_{\gl}(\gm)$ based on Proposition \ref{p.subsequences}(i): $C$ lies in $\ce_{\gl}(\gm)$ precisely when the product of the reflections of $\vs_C$ equals $w$. Of course, checking this requires calculations in the 0-Hecke algebra.

\begin{example} Let $n=7$, $d=3$. For $v=\{4,6,7,1,2,3,5\}$ and $w=\{1,3,5,2,4,6,7\}$,  $\gm=\gl_v=(4,4,3)$ and $\gl=\gl_w=(2,1)$. The set-valued tableaux $\ct_{\gl}(\gm)$ appear in Figure \ref{f.syt_all}. The expression for $i_v^*[\co_{X^w}]$ computed using Theorem \ref{t.ktheory_svt} is the same as the expression computed using Theorem \ref{t.ktheory_eyd} in Example \ref{ex.restriction_A}.
\end{example}

\subsubsection{Proof of Theorem \ref{t.ktheory_svt}}
Define $f:\ct_{\gl}(\gm)\to\{\text{subsets of }D_{\gm}\}$ by
\begin{equation}\label{e.f}
\begin{split}
f(T)&= \{(x,x+j-i)\mid (i,j)\in T, x\in T(i,j)\}\\
    &=\{(x,k\backslash\mid (i,k\backslash \in T, x\in T(i,k\backslash\}
\end{split}
\end{equation}
for $T\in \ct_{\gl}(\gm)$, where $(i,k\backslash$ denotes the box in row $i$, diagonal $k$ of $T$. 
To see that every box of $f(T)$ lies in $D_{\gm}$, and thus $f$ is well defined, observe that for every box $(i,j)\in T$ and entry $x\in T(i,j)$, $1\leq x\leq d$ and $1\leq x+j-i\leq \gm(x)$.  Indeed,  $1\leq x\leq d$ by definition. Semistandardness of $T$ forces $i\leq x$; hence $1\leq x+j-i$. The final inequality, namely $x+j-i\leq \gm(x)$, is \eqref{e.restrictedbyv}. As stated above, we prove Theorem \ref{t.ktheory_svt} by showing that $f$ is in fact a bijection from $\ct_\gl(\gm)$ to $\ce_\gl(\gm)$ (Proposition \ref{p.fbijective}).

We emphasize that $T$ and $f(T)$ are associated to Young diagrams of different shapes: $T$ has shape $\gl$, and $f(T)$ is a subset of $D_\gm$, which has shape $\gm$. The subset $f(T)$ of $D_{\gm}$ contains, for each integer $x$ of box $(i,j)$ of $T$, box $(x, x+j-i)$ of $D_{\gm}$. We write $f(T|_x)=(x, x+j-i)=(x,j-i\backslash$. Observe that
$f(T|_x)$ lies in the same diagonal as $x$, namely $j-i$, but in row $x$ instead of $i$. This suggests a more qualitative description of $f$. Suppose that the entries of some box $\ga$ of $T$  are $x_1,\ldots,x_k$. Corresponding to these entries, $f(T)$ will have boxes in $D_{\gm}$ in the same diagonal as $\ga$, and rows $x_1,\ldots,x_k$; thus the entries of $T$ record the rows of the boxes of $f(T)$. The inequality \eqref{e.restrictedbyv} merely ensures that the boxes of $f(T)$ actually lie in $D_\gm$.

In Section \ref{ss.eyd} we defined excitations on subsets of Young diagrams. Here we need an analogous operation on set-valued tableaux. Let $T\in\ct_{\gl}(\gm)$, and suppose that $x\in T(i,j)$, $x\not\in T(i,j+1)$, $x+1\not\in T(i,j)$, $x+1\not\in T(i+1,j)$, $(x+1)+j-i\leq \gm(x+1)$.   An excitation of type 1 replaces entry $x$ with $x+1$ in box $(i,j)$ of $T$, and an excitation of type 2 adds entry $x+1$ to box $(i,j)$ of $T$. Both types of excitations preserve semistandardness and the property of being restricted by $\gm$.
Let $T^{\text{top}}\in \ct_{\gl}(\gm)$ be defined by  $T^{\text{top}}(i,j)=\{i\}$, i.e., each box in row $i$ contains the single entry $i$. Then $f(T^{\text{top}})=D_{\gl}$.

\begin{figure}[height=30em]
    \centering
    \begin{tikzpicture}[
    level distance = 2.8cm,sibling distance = 2cm, grow = east,
    type2/.style={edge from parent/.style={draw,double distance =
    1.5pt,>=latex,->}},
    type1/.style={edge from parent/.style={draw,>=latex,->}},
    typ1/.style={>=latex,->},
    typ2/.style={double distance=1.5pt,>=latex,->}
]
\newcommand{\syt}[3]
    {
        \begin{tikzpicture}[scale=.6, every node/.style={scale=.8}]
             \foreach \pos / \label in
              { {(1,2)}/{#1},{(2,2)}/{#2},{(1,1)}/{#3} }
                  {
                    \draw \pos +(-.5,-.5) rectangle ++(.5,.5);
                    \draw \pos node {\label};
                    }
        \end{tikzpicture}
      }

\node (e1) [rectangle] {\syt{1}{1}{2}}
     child {node (e2) [rectangle] {\syt{1}{2}{2}}
        child {node (e6) [rectangle] {\syt{1}{2}{3}}
            edge from parent[draw=none]
            child {node (e10) [rectangle] {\syt{2}{2}{3}}
                edge from parent[draw=none]
            }
        }
    }
    child {node (e3) [rectangle] {\syt{1}{1}{3}}
        child {node (e7) [rectangle] {\syt{1}{2}{2,3}}
            edge from parent[draw=none]
            child {node (e11)[rectangle] {\syt{1,2}{2}{3}}
                edge from parent[draw=none]
            }
        }
    }
    child {node (e4) [rectangle] {\syt{1}{1}{2,3}}
        child {node (e8) [rectangle] {\syt{1}{1,2}{3}}
            edge from parent[draw=none]
        }
    }
    child {node (e5) [rectangle] {\syt{1}{1,2}{2}}
        child {node (e9) [rectangle] {\syt{1}{1,2}{2,3}}
            edge from parent[draw=none]
        }
    };
    \draw[rounded corners = 5pt, dashed]
        (e1.north west) -- (e1.north east)
        -- (e3.north west) -- (e3.north east)
        -- (e6.north west) -- (e10.north east)
        -- (e10.south east) -- (e2.south west)
        -- (e1.south west) -- cycle;
    \draw[typ2] (e5.0) -- (e9.180);
    \draw[typ1] (e5.340) -- (e8.160);
    \draw[typ2] (e4.20) -- (e9.200);
    \draw[typ1] (e4.340) -- (e7.160);
    \draw[typ2] (e3.20) -- (e8.200);
    \draw[typ1] (e3.340) -- (e6.160);
    \draw[typ2] (e2.20) -- (e7.200);
    \draw[typ1] (e2.0) -- (e6.180);
    \draw[typ2] (e6.20) -- (e11.200);
    \draw[typ1] (e6.0) -- (e10.180);
    \draw[typ2] (e1) -- (e4);
    \draw[typ2] (e1) -- (e5);
    \draw[typ1] (e1) -- (e2);
    \draw[typ1] (e1) -- (e3);
\end{tikzpicture}
    \caption{\label{f.syt_all} The set-valued tableaux $\ct_\gl(\gm)$, for $\gm=(4,4,2)$, $\gl=(2,1)$. The set-valued tableaux enclosed by the dashed line are Young tableaux; they comprise $\ct^{\text{red}}_\gl(\gm)$. The entries of each set-valued tableau record the row numbers of the boxes in the corresponding excited Young diagram of Figure \ref{f.eyd_all}.}
\end{figure}
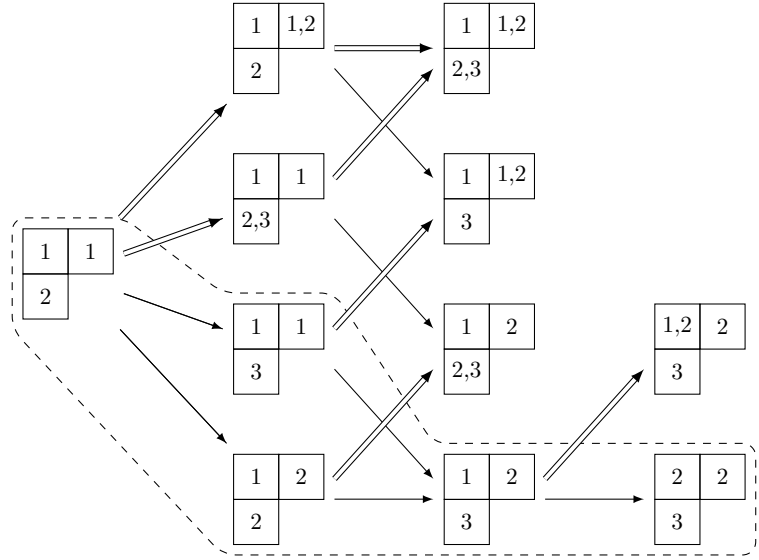

\begin{Lem}
Any element $T\in \ct_{\gl}(\gm)$ can be obtained by applying a sequence of excitations to  $T^{\text{top}}$.
\end{Lem}
\begin{proof}
For any set-valued tableau $T$ of shape $\gl$, define $s(T)$ to be the sum of the entries of $T$. Then
$s(T)\geq s(T^{\text{top}})$, and $s(T)=s(T^{\text{top}})$ if and only if $T=T^{\text{top}}$.
We proceed by induction on $s(T)$. Assume that $s(T)>s(T^{\text{top}})$, and thus $T\neq T^{\text{top}}$.

\begin{description}
    \item[\it Case 1] At least one box of $T$ contains more than one entry. Choose any such box. Define $T'$ to be the set-valued tableau obtained by removing the second smallest entry from this box of $T$.

    \item[\it Case 2] Every box of $T$ contains exactly one entry. Let $(i,j)$ be the largest box for which the entries of $T$ and $T^{\text{top}}$ do not agree (where we use the order on boxes introduced in Section \ref{ss.proof_p.subsequences}). Let $T'$ be the tableau obtained from $T$ by subtracting one from its entry in box $(i,j)$.
\end{description}
In both cases, $T'\in \ct_{\gl}(\gm)$ and $s(T')<s(T)$. By the induction hypothesis, $T'$ can be obtained by applying a sequence of excitations to $T^{\text{top}}$. Furthermore,
$T$ can be obtained by applying a sequence of excitations to $T'$. The result follows.
\end{proof}

Excitations commute with $f$, as made precise by the following lemma, which follows from the definitions.
\begin{Lem}\label{l.fcommutesexcitation}
Let $T\in\ct_{\gl}(\gm)$, and let $C=f(T)$. Let $x$, $x<d$, be an entry of box $(i,j)$ of $T$, and let $(a,b)=f(T|_x)\in C$. Then
$x+1\not\in T(i+1,j)$, $x\not\in T(i,j+1)$, $x+1\not\in T(i,j)$, $(x+1)+j-i\leq \gm(x+1)$ if and only $(a+1,b)\not\in C, (a,b+1)\not\in C,(a+1,b+1)\not\in C$, $(a+1,b+1)\in D_{\gm}$ respectively. If all of these conditions are satisfied, let $\gn$ be an excitation of $T$ modifying $x$, and let $\gm$ be the excitation of $C$ of the same type as $\gn$ modifying $(a,b)$. Then $f(\gn(T))=\gm(f(T))$. We say that $\gm$ corresponds to $\gn$ under $f$.
\end{Lem}

\begin{Lem}
The image of $f$ lies in $\ce_{\gl}(\gm)$.
\end{Lem}
\begin{proof}
 Let $T\in \ct_{\gl}(\gm)$. Then $T=\gn_k\cdots\gn_1(T^{\text{top}})$ for some excitations $\gn_1,\ldots,\gn_k$. By Lemma \ref{l.fcommutesexcitation}, $f(T)=\gm_k\cdots\gm_1(D_{\gl})\in \ce_{\gl}(\gm)$, where $\gm_1,\ldots\gm_k$ correspond to $\gn_1,\ldots\gn_k$ under $f$.
\end{proof}

\begin{Lem}\label{l.fsurjective}
The image of $f$ equals $\ce_{\gl}(\gm)$.
\end{Lem}
\begin{proof}
 Let $C\in \ce_{\gl}(\gm)$. Then $C=\gm_k\cdots\gm_1(D_{\gl})$ for some excitations $\gm_1,\ldots,\gm_k$. We prove that $C\in f(\ct_{\gm}(\gl))$ by induction on $k$. For the base case, we use the fact that $\gl=f(T^{\text{top}})$. Let $C'=\gm_{k-1}\cdots\gm_1(D_{\gl})$. The excitation $\gm_k$ modifies some box $(a,b)$ of $C'$. Thus  $(a+1,b)\not\in C, (a,b+1)\not\in C,(a+1,b+1)\not\in C$, $(a+1,b+1)\in D_{\gm}$. By the induction hypothesis, $C'=f(T')$ for some $T'\in\ct_{\gl}(\gm)$.  Let $x\in T'(i,j)$ be such that $f(T|_x)=(a,b)$, and let $\gn$ be the excitation of the same type as $\gm_k$ modifying $x$. By Lemma \ref{l.fcommutesexcitation}, $C=\gm_k C'=\gm_k f(T')=f(\gn T')$.
 \end{proof}

\begin{Lem}\label{l.finjective}
The map $f$ is injective.
\end{Lem}
\begin{proof}
Let $C\in\ce_{\gl}(\gm)$. By Lemma \ref{l.fsurjective}, there exists $T\in\ct_{\gl}(\gm)$ such that $f(T)=C$. We give a constructive proof of the uniqueness of $T$ by filling in the boxes of $T$ one diagonal at a time, beginning with the largest diagonal. As we shall see, there is only one way to fill in the boxes so that $T$ is semistandard and $f(T)=C$.

The largest diagonal $r$ of $D_{\gl}$ contains a single box, namely $(1,r\backslash$. For each box $(x,r\backslash$ of $C$, place an $x$ in box $(1,r\backslash$ of $T$. Now assume that we have filled in each box in diagonals $q+1,\ldots,r$ of $T$ with a nonempty set of positive integers. Let $(x,q\backslash\in C$. In order to satisfy $f(T)=C$, we must place an $x$ in some box of diagonal $q$ of $T$. In order for $T$ to be semistandard, we must place this $x$ in the unique box $(i,q\backslash$ of $T$ such that $x$ is strictly greater than all entries of box $(i-1,q+1\backslash$ and weakly less than all entries of box $(i,q+1\backslash$. Surjectivity of $f$ guarantees the existence of such a box $(i,q\backslash$.  Surjectivity of $f$ also ensures that if this procedure is carried out for every box of diagonal $q$ of $C$, every box of diagonal $q$ of $T$ will have at least one number placed inside of it.
\end{proof}

From Lemmas \ref{l.fsurjective} and \ref{l.finjective} we have

\begin{Prop}\label{p.fbijective}
The map $f$ is a bijection from $\ct_{\gl}(\gm)$ to $\ce_{\gl}(\gm)$.
\end{Prop}

\subsection{Hilbert series and Hilbert polynomials of points on Schubert varieties}\label{ss.hilbseries_An}

In type $A_n$, all of the maximal parabolic subgroups $P_d$ are cominuscule (cf. \cite[9.0.14]{BiLa:00}). Thus Corollary \ref{c.cominusculemult} may be used to compute the Hilbert series, Hilbert polynomial, and multiplicity at $v$ of a Schubert variety $X^w$ in the Grassmannian $\Grass(d,n)$. Let $\gl$ and $\gm$ be the partitions corresponding to $w$ and $v$ respectively. In the present setting, the constant $m_k$ of Corollary \ref{c.cominusculemult} is equal to the number of excited Young diagrams $C\in\ce_{\gl}(\gm)$ such that the number of boxes of $C$ is $k+|\gl|$ (due to Proposition \ref{p.subsequences}(i), Definition \ref{d.energy}(iii), and Lemma \ref{l.hecketoeyd}(ii)). In terms of set-valued tableaux, $m_k$ is equal to the number of $T\in\ct_{\gl}(\gm)$ with $k+|\gl|$ entries.

\begin{example}
Let $n=7$, $d=3$, $w=\{1,3,5,2,4,6,7\}$, and $v=\{4,6,7,1,2,3,5\}$, as in Example \ref{ex.restriction_A}. Then $\gl=\gl_w=(2,1)$, $\gm=\gl_v=(4,4,3)$, $l(w)=|\gl|=3$, and $d_w=d(n-d)-l(w)=9$. The set of excited Young diagrams $\ce_{\gl}(\gm)$ appears in Figure \ref{f.eyd_all}, and the set of set-valued tableaux $\ct_{\gl}(\gm)$ appears in Figure \ref{f.syt_all}. From either of these figures, one reads off $m_0=5$, $m_1=5$, and $m_2=1$.  Hence
{\allowdisplaybreaks
\begin{align*}
    H(X^w,v)(t)&=\frac{5}{(1-t)^9}-\frac{5}{(1-t)^8}+\frac{1}{(1-t)^7},\\
    h(X^w,v)(i)&=5{i +8\choose 8}-5{i+7\choose 7}+{i+6\choose 6},\\
    \mult(X^w,v)&=5.
\end{align*}
}
\end{example}

Several other multiplicity formulas have appeared in the literature: inductive \cite{LaWe:90}, determinantal \cite{RoZe:01}, \cite{WoYo:12}, \cite{LiYo:12}, and enumerative \cite{Kra:01}, \cite{KoRa:03} \cite{Kre:03}, \cite{KrLa:04}, \cite{Kra:05}, \cite{Kre:08}. The inductive formula of Lakshmibai and Weyman, which holds more generally for minuscule $G/P$, was used by Rosenthal and Zelevinsky to prove the determinantal formula, which in turn was used by Krattenthaler to prove the enumerative formula, which counts nonintersecting lattice paths.
The formula given above appeared earlier in \cite{IkNa:09} and \cite{WoYo:12}, and it can also be deduced from \cite{Kre:05} together with the multiplicity formulas of \cite{KoRa:03}, \cite{Kre:03}, \cite{KrLa:04}.

Formulas for the Hilbert series and Hilbert polynomial of $X^w$ at $v$ were obtained by \cite{Kra:05}. They were derived using an expression for the Hilbert function of $X^w$ at $v$ given in \cite{KoRa:03}, \cite{Kre:03}, \cite{KrLa:04}, and \cite{Kre:08}. The formulas given here appeared earlier in \cite{LiYo:12}.  We remark that the formulas of \cite{LiYo:12} apply more generally to covexillary permutations.

\section{Applications to the orthogonal and Lagrangian Grassmannians} \label{s.orthosymplectic}

In types $B_n$, $C_n$, and $D_n$, let $G$ be $\SO_{2n+1}(\C)$, $\Sp_{2n}(\C)$, and $\SO_{2n}(\C)$ respectively.  Each of these groups $G$ is defined to be the subgroup of a general linear group preserving a specified nondegenerate symmetric or skew symmetric inner product (see Appendix \ref{s.appendix_roots}). Let $P_n$ be the maximal parabolic subgroup of $G$ corresponding to simple root $\ga_n$.
\begin{itemize}
    \item[$B_n$:]  The coset space $G/P_n$ is identified with the \textbf{odd orthogonal Grassmannian} $OG(n,2n+1)$, which parametrizes the maximal ($n$ dimensional) isotropic subspaces of $\C^{2n+1}$. It has the structure of an irreducible  projective variety of dimension $n(n+1)/2$.

    \item[$C_n$:] The coset space $G/P_n$ is identified with the \textbf{Lagrangian Grassmannian} $LG(n,2n)$, which parametrizes the maximal ($n$ dimensional) isotropic subspaces of $\C^{2n}$. It has the structure of an irreducible  projective variety of dimension $n(n+1)/2$.

    \item[$D_n$:] The coset space $G/P_n$ is identified with the \textbf{even orthogonal Grassmannian} $OG(n,2n)$, which parametrizes one of the two components of the maximal ($n$-dimensional) isotropic subspaces of $\C^{2n}$.
        The even orthogonal Grassmannian has the structure of an irreducible projective variety of dimension $n(n-1)/2$.
\end{itemize}

 Let $B^-$ be the Borel subgroup of lower triangular matrices in $G$, and $T$ the subgroup of diagonal matrices in $G$. The Weyl group $W=N_G(T)/T$ embeds into
$\{w=(w_1,\ldots,w_{2n})\in S_{2n}\mid w_{\bar{\imath}}=\bar{w_i}\}$, where, for $x\in\{1,\ldots,2n\}$, $\bar{x}=2n+1-x$.
In particular, $w\in W$ is uniquely determined by $w_1,\ldots,w_n$. The Weyl group $W_{P_n}$ is isomorphic to $S_n$, and the set of \textbf{minimal length coset representatives} of $W/W_{P_n}$ is given by
\begin{equation}\label{e.WP}
    W^{{P_n}} = \left\{(w_1,\ldots,w_{2n})\in W\mid w_1<\cdots<w_n \right\}.
\end{equation}
The cosets $w{P_n}$, $w\in W^{P_n}$, are precisely the $T$-fixed points of $G/{P_n}$. By abuse of notation, we sometimes denote $w{P_n}$ by just $w$. The Schubert variety $X^w$ is by definition $\bar{B^-wP_n}\subseteq G/{P_n}$. It is an irreducible projective variety of codimension $l(w)$. For $v,w\in W^{P_n}$, $v\in X^w$ if and only if $v\geq w$ in the Bruhat order.

In this section, we give formulas for $i_v^*[\co_{X^w}]$ (Theorems \ref{t.ktheory_eyd_bcd} and \ref{t.ktheory_syt_bcd}), as well as the Hilbert series, Hilbert polynomial, and multiplicity of $X^w$ at $v$ (Section \ref{ss.hilbseries_BCD}). These are reformulations of Theorem \ref{t.pullback} and Corollary \ref{c.cominusculemult}, expressed in terms of {\it excited shifted Young diagrams} (\cite{Kre:06},
\cite{IkNa:09}) and {\it set-valued shifted tableaux}.
The term excited shifted Young diagram is due to Ikeda and Naruse \cite{IkNa:09}; in \cite{Kre:06}, in which only the case $G=\Sp_{2n}(\C)$ is studied, this is called a subset
of a Young diagram.  Also, in this paper we have modified the definition of excited shifted Young diagram for our applications
to $K$-theory;
the earlier definitions correspond to our {\it reduced} excited shifted Young diagrams. Reduced excited Young diagrams were discovered independently by Kreiman \cite{Kre:06} and Ikeda and Naruse \cite{IkNa:09}.   A related version of excited shifted  Young diagram was introduced in \cite{IkNa:11} (see Section \ref{ss.previous}).

 In type $C_n$, our formula for $i_v^*[\co_{X^w}]$, expressed in terms of set-valued shifted tableaux, was obtained earlier in \cite{Kre:06}. The derivation of the formula there relies on an equivariant Gr\"{o}bner degeneration of $X^w$ in a neighborhood of $v$; this degeneration is due to \cite{GhRa:06}. Our approach is different. As in Section \ref{s.Grassmannian}, our methods in this section are modeled on those of Ikeda and Naruse \cite{IkNa:09}, and in several places, we use their results directly.
Formulas for $i_v^*[\co_{X^w}]$ which were obtained in \cite{IkNa:11} are discussed in Section \ref{ss.previous}.

 \subsection{Strict partitions and shifted Young diagrams}\label{ss.strict_partitions_bcd}
A partition $\gl=(\gl_1,\ldots,\gl_n)\in \cp_{n,n}$ is said to be \textbf{strict} if $\gl_i=\gl_{i+1}$ implies $\gl_{i}=\gl_{i+1}=0$. Let $\cs\cp_{n}$ denote the set of such partitions. A \textbf{shifted Young diagram} is an array of boxes arranged such that the row lengths strictly decrease from top to bottom and the leftmost box of row $i$ lies in column $i$. To a strict partition $\gl$ we associate the shifted Young diagram $D'_\gl$ whose $i$-th row contains $\gl_i$ boxes. Then $D'_{\gl}$ fits in the upper triangular boxes of an $n\times n$ square.

\begin{example} Let $n=6$, and $\gl=(5,4,2,1)\in\cs\cp_{n}$.
The shifted Young diagram $D'_\gl$ fits in the upper triangular boxes of  an $n\times n$ square.
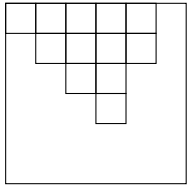
\begin{center}
\begin{tikzpicture}[scale=.4]
\draw (.5,.5) -- ++(6,0) -- ++(0,6) -- ++(-6,0) -- cycle;
\foreach \pos in
      { {(1,6)},{(2,6)},{(3,6)},{(4,6)},{(5,6)},
        {(2,5)},{(3,5)},{(4,5)},{(5,5)},
        {(3,4)},{(4,4)},
        {(4,3)}
      }
          {\draw \pos +(-.5,-.5) rectangle ++(.5,.5);}
\end{tikzpicture}
\end{center}
\end{example}

Define a map $W^{P_n}\to \cs\cp_{n}$, $w\mapsto \gl'_w$, as follows. Given $w\in W^{P_n}$, form the Young diagram $D_{\gl_w}$. By Lemma \ref{l.symmsymm}, $D_{\gl_w}$ is symmetric. In types $B_n$ and $C_n$, remove all boxes $(i,j)$ of $D_{\gl_w}$ such that $i> j$; in type $D_n$, remove all boxes $(i,j)$ of $D_{\gl_w}$ such that $i\geq j$. The resulting shifted Young diagram corresponds to $\gl'_w$. More explicitly,
in types $B_n$ and $C_n$,
\begin{equation}\label{e.glprimewbc}
    (\gl'_w)_i=\max\{(\gl_w)_i-(i-1),0\}, i=1,\ldots,n,
\end{equation}
and in type $D_n$,
\begin{equation}\label{e.glprimewd}
    (\gl'_w)_i=\max\{(\gl_w)_i-i,0\}, i=1,\ldots,n-1.
\end{equation}
As in type $A_n$, $l(w)=|\gl'_w|$.

\begin{example}\label{ex.strict_partitions} Let $n=6$, and $w=(1,4,\bar{6},\bar{5},\bar{3},\bar{2},2,3,5,6,\bar{4},\bar{1})$\\ $=(1,4,7,8,10,11,2,3,5,6,9,12)\in W^{P_n}$.
Then $\gl_w=(5,5,4,4,2)\in \cp_{n,n}$.
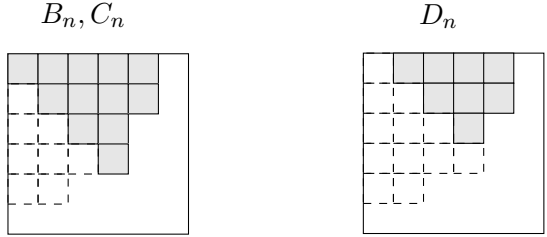
\begin{equation*}
\begin{tikzpicture}[scale=.4,
baseline={([yshift=-.5ex]current bounding box.center)}]
\draw (.5,.5) -- ++(6,0) -- ++(0,6) -- ++(-6,0) -- cycle;
\foreach \pos in
      { {(1,6)},{(2,6)},{(3,6)},{(4,6)},{(5,6)},
        {(2,5)},{(3,5)},{(4,5)},{(5,5)},
        {(3,4)},{(4,4)},
        {(4,3)}
      }
          {\draw[fill=gray!20] \pos +(-.5,-.5) rectangle ++(.5,.5);}
\foreach \pos in
      { {(1,5)},
        {(1,4)}, {(2,4)},
        {(1,3)}, {(2,3)},{(3,3)},
        {(1,2)}, {(2,2)}
      }
          {\draw[dashed] \pos +(-.5,-.5) rectangle ++(.5,.5);}
\draw (3,7) node[anchor=south] {$B_n, C_n$};
\end{tikzpicture}
\qquad\qquad\qquad
\begin{tikzpicture}[scale=.4,
baseline={([yshift=-.5ex]current bounding box.center)}]
\draw (.5,.5) -- ++(6,0) -- ++(0,6) -- ++(-6,0) -- cycle;
\foreach \pos in
      {{(2,6)},{(3,6)},{(4,6)},{(5,6)},
        {(3,5)},{(4,5)},{(5,5)},
        {(4,4)}
      }
          {\draw[fill=gray!20] \pos +(-.5,-.5) rectangle ++(.5,.5);}
\foreach \pos in
      { {(1,6)},
        {(1,5)},{(2,5)},
        {(1,4)}, {(2,4)},{(3,4)},
        {(1,3)}, {(2,3)}, {(3,3)}, {(4,3)},
        {(1,2)}, {(2,2)}
      }
          {\draw[dashed] \pos +(-.5,-.5) rectangle ++(.5,.5);}
\draw (3,7) node[anchor=south] {$D_n$};
\end{tikzpicture}
\end{equation*}

\noindent In types $B_n$ and $C_n$, $\gl'_w=(5,4,2,1)$, and in type $D_n$, $\gl'_w=(4,3,1)$. These are the strict partitions obtained by removing the unshaded boxes from $D_{\gl_w}$ in the diagrams above.
\end{example}

In types $B_n$ and $C_n$, the map $W^{P_n}\to \cs\cp_{n}$ is a bijection, and in type $D_n$, the map $W^{P_n}\to \cs\cp_{n-1}$ is a bijection. The assertion in type $D_n$  follows from the fact that both $W^{P_n}\to Q$ and $Q\to \cs\cp_{n-1}$ are bijections, where $Q$ is the set of symmetric partitions with at most $n$ rows and an even number of boxes of the form $(i,i)$ (i.e., lying on diagonal 0). The map $W^{P_n}\to Q$, $v\mapsto \gl_v$, is well defined and bijective by Lemma \ref{l.symmsymm} and \eqref{e.Weyl_Dn}. The map $Q\to \cs\cp_{n-1}$, $\gl_v\mapsto \gl'_v$, is bijective by straightforward combinatorial arguments.
In types $B_n$ and $C_n$, the proof is similar but simpler.

Thus $W^{P_n}$ in type $D_n$ and $W^{P_{n-1}}$ in type $B_{n-1}$ are both in bijection with $\cs\cp_{n-1}$. We identify the element of $W^{P_{n}}$ in type $D_{n}$ with the element of $W^{P_{n-1}}$ in type $B_{n-1}$ which corresponds to the same strict partition of $\cs\cp_{n-1}$. This in turn gives an identification between the fixed points of $\OG(n,2n)$ and of $\OG(n-1,2n-1)$.

\begin{example}\label{ex.fixedpts_bijection_BD}
    In type $D_n$, with $n=6$, let $w=(1,4,\bar{6},\bar{5},\bar{3},\bar{2},2,3,5,6,\bar{4},\bar{1})$ as in Example \ref{ex.strict_partitions}. Then $\gl_w=(5,5,4,4,2)\in \cp_{n,n}$ and $\gl'_w=(4,3,1)\in\cs\cp_{n-1}$. Let $\gl=\gl'_w$. In type $B_{n-1}$, $\gl\in\cs\cp_{n-1}$ corresponds to $(4,4,3,2)\in \cp_{n-1,n-1}$, which corresponds to $u=(1,4,\bar{5},\bar{3},\bar{2},2,3,5,\bar{4},\bar{1})\in W^{P_{n-1}}$.
        
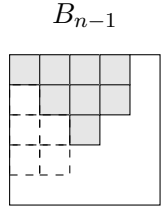
\begin{center}
        \begin{tikzpicture}[scale=.4,
        baseline={([yshift=-.5ex]current bounding box.center)}]
        \draw (.5,.5) -- ++(5,0) -- ++(0,5) -- ++(-5,0) -- cycle;
        \foreach \pos in
              { {(1,5)},{(2,5)},{(3,5)},{(4,5)},
                {(2,4)},{(3,4)},{(4,4)},
                {(3,3)}
              }
                  {\draw[fill=gray!20] \pos +(-.5,-.5) rectangle ++(.5,.5);}
        \foreach \pos in
              { {(1,4)},
                {(1,3)}, {(2,3)},
                {(1,2)}, {(2,2)}
              }
                  {\draw[dashed] \pos +(-.5,-.5) rectangle ++(.5,.5);}
        \draw (3,6) node[anchor=south] {$B_{n-1}$};
        \end{tikzpicture}
        \end{center}
    Thus we identify the Weyl group element $w$ in type $D_n$ with $u$ in type $B_{n-1}$, and we identify the point $w$ in $\OG(n,2n)$ with $u$ in $\OG(n-1,2n-1)$.
\end{example}

    Another identification between $W^{P_n}$ in type $D_n$ and $W^{P_{n-1}}$ in type $B_{n-1}$ appears in \cite[1.3]{RaUp:10}. Given $w=(w_1,\ldots,w_{2n})\in W^{P_n}$ in type $D_n$, merely remove $n$ and $\bar{n}$ from the entries $w_1,\ldots,w_{2n}$ in order to produce element $u$ of $B_{n-1}$. (In Example \ref{ex.fixedpts_bijection_BD}, $u$ can be obtained from $w$ by removing $6$ and $\bar{6}$ from the entries $w_1,\ldots,w_{12}$.)
    In the other direction, suppose we are given $u=(u_1,\ldots,u_{2n-2})\in W^{P_{n-1}}$ in type $B_{n-1}$. Add $n$ and $\bar{n}$ to the entries of $u$ in such a way that the new permutation is symmetric, and the first $n$ of its entries have an even number of barred entries and are increasing.

For the remainder of this section, we prove that this identification is the same as ours. Let $\gl=(\gl_1,\ldots,\gl_{n-1})\in\cs\cp_{n-1}$. Let $\gl=\gl'_w$ for $w=(w_1,\ldots,w_{2n})\in W^{P_n}$ in type $D_n$, and let $\gl=\gl'_u$ for $u=(u_1,\ldots,u_{2n-2})\in W^{P_{n-1}}$ in type $B_{n-1}$. We wish to show that $u$ is obtained by removing $n$ and $\bar{n}$ from the entries of $w$. This is equivalent to the assertion $\{w_1,\ldots,w_n\}\setminus \{n,\bar{n}\}=\{u_1,\ldots,u_{n-1}\}$.

\begin{Lem}\label{l.partsbd} For $1\leq i\leq n-1$,
\begin{align*}
    & B_{n-1}:\quad u_{n-i}\geq \bar{n-1}\text{ if and only if }\gl_i>0,\text{ in which case }u_{n-i}=\bar{n-\gl_i} \\
    & D_n:\quad w_{n+1-i}\geq \bar{n-1}\text{ if and only if }\gl_i>0,\text{ in which case }w_{n+1-i}=\bar{n-\gl_i}
\end{align*}
\end{Lem}
\noindent Note that in type $B_{n-1}$, $\bar{x}=2n-1-x$, whereas in type $D_{n}$, $\bar{x}=2n+1-x$.
\begin{proof}
We prove this for type $D_n$. The proof for type $B_{n-1}$ is similar.
By \eqref{e.lambdav}, for $i=1,\ldots,n$, $(\gl_w)_i=w_{n+1-i}-(n+1-i)$. Thus \begin{equation}\label{e.lem_technical}
    (\gl_w)_i-i=w_{n+1-i}-(n+1)=w_{n+1-i}-\bar{n}.
 \end{equation}
 Thus $w_{n+1-i}\geq \bar{n-1}$ if and only if $w_{n+1-i}>\bar{n}$ if and only if $(\gl_w)_i>i$ if and only if $\gl_i=(\gl'_w)_i>0$, by \eqref{e.glprimewd}. In this case $(\gl_w)_i-i=(\gl'_w)_i=\gl_i$, so \eqref{e.lem_technical} implies $w_{n+1-i}=\gl_i+\bar{n}=\bar{n-\gl_i}$.
\end{proof}

Recall for $x=1,\ldots,n$, exactly one of $x$ or $\bar{x}$ occurs among $w_1,\ldots,w_n$, and for $y=1,\ldots,n-1$, exactly one of $y$ or $\bar{y}$ occurs among $u_1,\ldots,u_{n-1}$. In particular exactly one of $n$ or $\bar{n}$ occurs among $w_1,\ldots,w_n$. If one excludes this entry, then Lemma \ref{l.partsbd} implies that the barred entries of $w_1,\ldots,w_n$ and of $u_1,\ldots,u_{n-1}$ are the same, and hence the unbarred entries must be the same as well. This completes the proof.

\subsection{The orthogonal Grassmannians}\label{ss.orthog_grassmannians} In this section we show that the restriction formula in type $B_n$ can be obtained from the restriction formula in type $D_{n+1}$. 

Let $S$ and $T$ be the set of diagonal matrices in $\SO_{2n+2}(\C)$ and $\SO_{2n+1}(\C)$ respectively. The $(n+1)$-st diagonal entry of any element of $T$ must be 1, and the map $\diag(t_1,\ldots,t_{2n+1})\mapsto \diag(t_1,\ldots,t_{n},1,1,t_{n+2},\ldots,t_{2n+1})$ embeds $T$ as a closed subgroup of $S$. Let $Y=\OG(n+1,2n+2)$ and $X=\OG(n,2n+1)$.   There is a $T$-equivariant isomorphism $\pi:Y\to X$ (cf. \cite[1.3]{RaUp:10}).
Hence
    $\pi^*:K_T(Y)\to K_T(X)$ is an isomorphism.
The sets $\{T$-fixed points of $Y\}$, $\{T$-fixed points of $X\}$, and $\{S$-fixed points of $Y\}$ can be identified. The identification of the first two sets follows from the $T$-equivariance of $\pi$, and the identification of the last two sets is described in Section \ref{ss.strict_partitions_bcd}. In fact, these two identifications are the same (cf. \cite[1.3]{RaUp:10}).

If $S$ acts on any scheme $M$, then there is a natural restriction homomorphism\\ $\res:K_S(M)\to K_T(M)$ taking the class of an
$S$-equivariant sheaf $\cf$ to the class of the same sheaf, but viewed as equivariant with respect to the $T$ action. If $w$ is any $S$-fixed point (equivalently, $T$-fixed point) of $Y$, then the square and triangle in the diagram
\begin{center}
\begin{tikzpicture}[scale=.6]
[description/.style={fill=white,inner sep=2pt}] \matrix (m) [matrix
  of math nodes, row sep=3.5em, column sep=3.5em, text
  height=1.5ex, text depth=0.25ex]
      { K_S(Y) & K_T(Y) & K_T(X)\\
      K_S(w)& K_T(w) & \\};
      \path[>=stealth,->]
          (m-1-1) edge node[auto] {$\res$} (m-1-2)
          (m-1-3) edge node[auto,swap] {$\pi^*$} node[auto] {$\sim$} (m-1-2)
          (m-2-1) edge node[auto,swap] {$\res$} (m-2-2)
          (m-1-1) edge node[auto,swap] {$i_w^*$} (m-2-1)
          (m-1-2) edge node[auto,swap] {$i_w^*$} (m-2-2)
          (m-1-3) edge node[auto] {$i_w^*$} (m-2-2);
\end{tikzpicture}
\end{center}
commute. Let $\fs^*$ and $\ft^*$ be the duals of the Lie algebras of $S$ and $T$ respectively, with bases as given in Appendix
\ref{s.appendix_roots}. Define $f:\fs^*\to \ft^*$ by $\gre_i\mapsto \gre_i$, $i=1,\ldots,n$, $\gre_{n+1}\mapsto 0$.  Then $\res:K_S(w)\to K_T(w)$ is the homomorphism defined by $e^{\gm}\mapsto e^{f(\gm)}$.

The isomorphism $\pi$ identifies Schubert varieties of $Y$ with Schubert varieties of $X$ (cf. \cite[1.3]{RaUp:10}).
We shall denote a Schubert variety in $Y$ as $Y^w$ and $\pi(Y^w)$ as $X^w$.
\begin{Prop}\label{p.BfromD_restriction} Let $w$ denote both a $T$-fixed point of $X$ and the corresponding $S$-fixed point of $Y$. Then $i_w^*[X^w]=\res(i_w^*[Y^w])$.
\end{Prop}
\begin{proof} We have $\res[Y_w]=[Y_w]$, where the first $[Y_w]$ lies in $K_S(Y)$ and the second in $K_T(Y)$. Further, $\pi^*[X^w]=[\pi^{-1}(X^w)]=[Y^w]$. Hence
 $
    \res (i_w^*[Y^w])=(i_w^*\circ (\pi^*)^{-1}\circ \res)[Y^w]=i_w^*[X^w].
$
\end{proof}

\subsection{Restriction formula in terms of excited shifted Young diagrams}
\label{ss.eyd_bcd}
Suppose that  $\gl$ is a strict partition and $S$ is a subset of $D'_\gl$. As in type $A_n$, an excitation moves or adds a box to $S$. Let $(i,j)\in S$. If $i<j$, then there are two possible excitations based on $(i,j)$, and these are the same as in type $A_n$:
if $(i+1,j),(i,j+1),(i+1,j+1)\in D_\gl\setminus S$, then  excitation of type 1 replaces $S$ by $S\setminus (i,j)\cup (i+1,j+1)$, and an excitation of type 2 replaces $S$ by $S\cup (i+1,j+1)$. If $i=j$, then the excitations are described as follows:
\begin{itemize}
\item[$B_n$, $C_n$:] If $(i,i+1), (i+1,i+1)\in D'_\gl\setminus S$,
    then an excitation of type 1 replaces $S$ by $S\setminus (i,i)\cup (i+1,i+1)$, and an excitation of type 2 replaces $S$ by $S\cup (i+1,i+1)$.

\item[$D_{n}$:] If $(i,i+1),(i+1,i+1), (i+1,i+2), (i+2,i+2)\in D'_\gl\setminus S$, then an excitation of type 1 replaces $S$ by $S\setminus (i,i)\cup (i+2,i+2)$, and an excitation of type 2 replaces $S$ by $S\cup (i+2,i+2)$.
\end{itemize}

 Let $\gl,\gm\in \cs\cp_{n}$. If $\gl_i\leq \gm_i$, $i=1,\ldots,n$, the map $(i,j)\mapsto (i,j)$ embeds $D'_{\gl}$ as a subset of $D'_{\gm}$. An \textbf{excited shifted Young diagram} of $D'_\gl$ in $D'_\gm$ is a subset of $D'_\gm$ which can be obtained by applying a sequence of excitations to the subset $D'_\gl$. An excited shifted Young diagram is said to be \textbf{reduced} if it can be obtained by applying only type 1 excitations. 
 Denote the set of excited shifted Young diagrams of $D'_\gl$ in $D'_\gm$ by $\ce'_\gl(\gm)$, and the set of reduced excited shifted Young diagrams by $\ce'^{\text{red}}_\gl(\gm)$. 

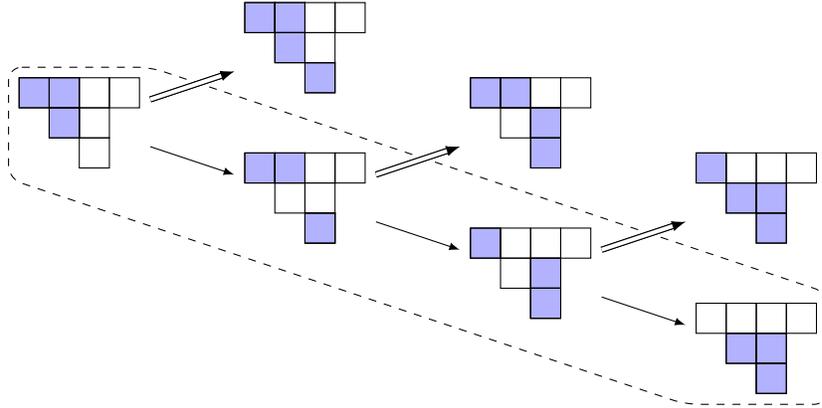
\begin{figure}[height=30em]
    \centering
    \begin{tikzpicture}[
    level distance = 3cm,sibling distance = 2cm, grow = east,
    type2/.style={edge from parent/.style={draw,double distance =
    1.5pt,>=latex,->}},
    type1/.style={edge from parent/.style={draw,>=latex,->}},
    typ1/.style={>=latex,->},
    typ2/.style={double distance=1.5pt,>=latex,->}
]
\newcommand{\eyd}[4]
    {
        \begin{tikzpicture}[scale=.4]
            \foreach \pos in
            {{#1},{#2},{#3},{#4}}
                {\draw[fill=blue!30] \pos +(-.5,-.5) rectangle ++(.5,.5);}
             \foreach \pos in
              { {(1,3)},{(2,3)},{(3,3)},{(4,3)},
                {(2,2)},{(3,2)},
                {(3,1)}
              }
                  {\draw \pos +(-.5,-.5) rectangle ++(.5,.5);}
        \end{tikzpicture}
      }

\node (e1) [rectangle] {\eyd{(1,3)}{(2,3)}{(2,2)}{(2,2)}}
    child[type1] {node (e2) [rectangle] {\eyd{(1,3)}{(2,3)}{(3,1)}{(3,1)}}
        child[type1] {node (e4) [rectangle] {\eyd{(1,3)}{(3,2)}{(3,1)}{(3,1)}}
            child[type1] {node (e6) [rectangle] {\eyd{(2,2)}{(3,2)}{(3,1)}{(3,1)}}}
            child[type2] {node (e7) [rectangle] {\eyd{(1,3)}{(2,2)}{(3,2)}{(3,1)}}}
        }
        child[type2] {node (e5) [rectangle] {\eyd{(1,3)}{(2,3)}{(3,1)}{(3,2)}}}
    }
    child[type2] {node (e3) [rectangle] {\eyd{(1,3)}{(2,3)}{(2,2)}{(3,1)}}};
    \draw[rounded corners = 5pt, dashed]
        (e1.north west) -- (e1.north east)
        -- (e6.north east) -- (e6.south east)
        -- (e6.south west) -- (e1.south west)
        -- cycle;
        \draw[typ2] (e1) -- (e3);
        \draw[typ2] (e2) -- (e5);
        \draw[typ2] (e4) -- (e7);
\end{tikzpicture}
    \caption{\label{f.eyd_allbc} The excited shifted Young diagrams $\ce'_\gl(\gm)$ in types $B_n$ and $C_n$, for $n=4$, $\gm=(4,2,1)$, $\gl=(2,1)$. The dashed line encloses $\ce'^{\text{red}}_\gl(\gm)$.}
\end{figure}

\begin{figure}[height=30em]
    \centering
    \input{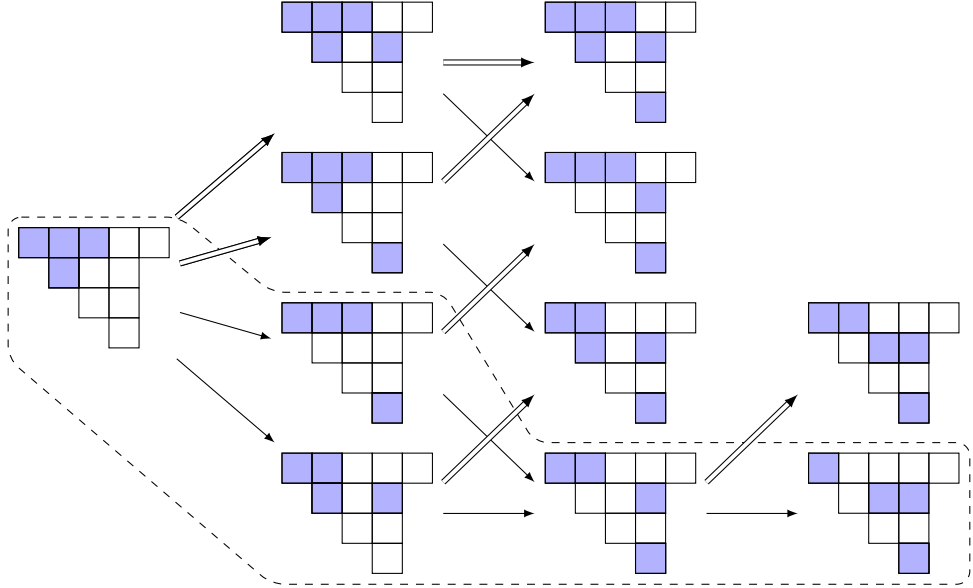}
    \caption{\label{f.eyd_alld} The excited shifted Young diagrams $\ce'_\gl(\gm)$ in type $D_n$, for $n=6$, $\gm=(5,3,2,1)$, $\gl=(3,1)$. The dashed line encloses $\ce'^{\text{red}}_\gl(\gm)$.}
\end{figure}

\begin{Thm}\label{t.ktheory_eyd_bcd}
Let $w\leq v\in W^{P_n}$, and let $\gl$, $\gm$ be the corresponding strict partitions. 
{\allowdisplaybreaks
\begin{align*}
    & B_n:\quad i_v^*[\co_{X^w}]=(-1)^{l(w)}\sum_{C\in\ce'_{\gl}(\gm)}\prod_{(i,j)\in C}\left(e^{-2^{-\delta_{ij}}(\gre_{v_{n+i}}+\gre_{v_{n+j}})}-1\right)\\
    & C_n:\quad i_v^*[\co_{X^w}]=(-1)^{l(w)}\sum_{C\in\ce'_{\gl}(\gm)}\prod_{(i,j)\in C}\left(e^{-(\gre_{v_{n+i}}+\gre_{v_{n+j}})}-1\right)\\
    & D_{n}:\quad i_v^*[\co_{X^w}]=(-1)^{l(w)}\sum_{C\in\ce'_{\gl}(\gm)}\prod_{(i,j)\in C}\left(e^{-(\gre_{v_{n+i}}+\gre_{v_{n+j+1}})}-1\right)
\end{align*}
}
where for $1\leq m\leq n$, $\gre_{\bar{m}}$ is defined to equal $-\gre_m$.
\end{Thm}

\begin{Rem}
We now have two methods for computing $i_v^*[\co_{X^w}]$ in type $B_n$. We can either employ the above formula for type $B_n$, or we can invoke Proposition \ref{p.BfromD_restriction}: first use the formula for type $D_{n+1}$ and then set every $\gre_{n+1}$ to 0. In general, the two methods produce different expressions for $i_v^*[\co_{X^w}]$. From the combinatorics alone it is not clear that these two expressions result in the same value. In general, the latter method is computationally simpler because it involves fewer excited Young diagrams. We illustrate both methods in the examples.
\end{Rem}

\begin{example}\label{ex.restriction_BC} 
Let $n=4$.
Let $v=\{2,\bar{4},\bar{3},\bar{1},1,3,4,\bar{2}\}$, $w=\{1,2,\bar{4},\bar{3},3,4,\bar{2},\bar{1}\}$ in type $C_n$ or $B_n$.
Then $\gl_v=(4,3,3,1)$, $\gm=\gl'_v=(4,2,1)$, $\gl_w=(2,2)$, and $\gl=\gl'_w=(2,1)$. Thus $l(w)=|\gl'_w|=3$. The set of excited shifted Young diagrams $\ce'_{\gl}(\gm)$ appears in Figure \ref{f.eyd_allbc}.
For any $C\in\ce'_{\gl}(\gm)$ and $(i,j)\in C$, a simple method for finding the indices $v_{n+i}$ and $v_{n+j}$ appearing in Theorem \ref{t.ktheory_eyd_bcd} is to label the rows of $C$, from top to bottom, as well as the columns, from left to right, with the entries $v_{n+1},v_{n+2},\ldots$ of $v$ (where $\bar{x}$ is replaced by $-x$). Then $v_{n+i}$ and $v_{n+j}$ are the row $i$ and column $j$ labels respectively. For example, for
\begin{equation*}
C=\
\begin{tikzpicture}[scale=.6,every node/.style={scale=1},
    baseline={([yshift=-.5ex]current bounding box.center)}]
    \foreach \pos in
    {{(1,3)},{(2,3)},{(3,1)}}
        {\draw[fill=blue!30] \pos +(-.5,-.5) rectangle ++(.5,.5);}
     \foreach \pos in
      { {(1,3)},{(2,3)},{(3,3)},{(4,3)},
        {(2,2)},{(3,2)},
        {(3,1)}
      }
          {\draw \pos +(-.5,-.5) rectangle ++(.5,.5);}
    \draw  (0,1) node{$4$};  \draw  (0,2) node{$3$};  \draw  (0,3) node{$1$};
    \draw  (1,4) node{$1$}; \draw  (2,4) node{$3$}; \draw  (3,4) node{$4$};
    \draw  (4,4) node{$-2$};
\end{tikzpicture}
\in\ce'_{\gl}(\gm),
\end{equation*}
in type $B_n$,
\begin{align*}
&\prod_{(i,j)\in C}\Big(e^{-2^{-\gd_{ij}}(\gre_{v_{n+i}}+\gre_{v_{n+j}})}-1\Big)
=\Big(e^{-\gre_1}-1\Big)\Big(e^{-(\gre_1+\gre_3)}-1\Big)
\Big(e^{-\gre_4}-1\Big),
\intertext{and in type $C_n$,}
&\prod_{(i,j)\in C}\Big(e^{-(\gre_{v_{n+i}}+\gre_{v_{n+j}})}-1\Big)
=\Big(e^{-2\gre_1}-1\Big)\Big(e^{-(\gre_1+\gre_3)}-1\Big)
\Big(e^{-2\gre_4}-1\Big).
\end{align*}
Invoking Theorem \ref{t.ktheory_eyd_bcd}, we obtain, in type $B_n$,
{\allowdisplaybreaks
\begin{align*}
    i_v^*[\co_{X^w}]=&
    -\Big(e^{-\gre_1}-1\Big)\Big(e^{-(\gre_1+\gre_3)}-1\Big)
    \Big(e^{-\gre_3}-1\Big)
    -\Big(e^{-\gre_1}-1\Big)\Big(e^{-(\gre_1+\gre_3)}-1\Big)
    \Big(e^{-\gre_4}-1\Big)\\
    &-\Big(e^{-\gre_1}-1\Big)\Big(e^{-(\gre_3+\gre_4)}-1\Big)
    \Big(e^{-\gre_4}-1\Big)
    -\Big(e^{-\gre_3}-1\Big)\Big(e^{-(\gre_3+\gre_4)}-1\Big)
    \Big(e^{-\gre_4}-1\Big)\\
    &-\Big(e^{-\gre_1}-1\Big)\Big(e^{-(\gre_1+\gre_3)}-1\Big)
    \Big(e^{-\gre_3}-1\Big)\Big(e^{-\gre_4}-1\Big)\\
    &-\Big(e^{-\gre_1}-1\Big)\Big(e^{-(\gre_1+\gre_3)}-1\Big)
    \Big(e^{-(\gre_3+\gre_4)}-1\Big)\Big(e^{-\gre_4}-1\Big)\\
    &-\Big(e^{-\gre_1}-1\Big)\Big(e^{-\gre_3}-1\Big)
    \Big(e^{-(\gre_3+\gre_4)}-1\Big)\Big(e^{-\gre_4}-1\Big),
\intertext{and in type $C_n$,}
    i_v^*[\co_{X^w}]=&
    -\Big(e^{-2\gre_1}-1\Big)\Big(e^{-(\gre_1+\gre_3)}-1\Big)
    \Big(e^{-2\gre_3}-1\Big)
    -\Big(e^{-2\gre_1}-1\Big)\Big(e^{-(\gre_1+\gre_3)}-1\Big)
    \Big(e^{-2\gre_4}-1\Big)\\
    &-\Big(e^{-2\gre_1}-1\Big)\Big(e^{-(\gre_3+\gre_4)}-1\Big)
    \Big(e^{-2\gre_4}-1\Big)
    -\Big(e^{-2\gre_3}-1\Big)\Big(e^{-(\gre_3+\gre_4)}-1\Big)
    \Big(e^{-2\gre_4}-1\Big)\\
    &-\Big(e^{-2\gre_1}-1\Big)\Big(e^{-(\gre_1+\gre_3)}-1\Big)
    \Big(e^{-2\gre_3}-1\Big)\Big(e^{-2\gre_4}-1\Big)\\
    &-\Big(e^{-2\gre_1}-1\Big)\Big(e^{-(\gre_1+\gre_3)}-1\Big)
    \Big(e^{-(\gre_3+\gre_4)}-1\Big)\Big(e^{-2\gre_4}-1\Big)\\
    &-\Big(e^{-2\gre_1}-1\Big)\Big(e^{-2\gre_3}-1\Big)
    \Big(e^{-(\gre_3+\gre_4)}-1\Big)\Big(e^{-2\gre_4}-1\Big).
    \end{align*}
}
\end{example}

\begin{example}\label{ex.restriction_D} 
Let $n=6$. In type $D_n$, let $v=\{2,6,\bar{5},\bar{4},\bar{3},\bar{1},1,3,4,5,\bar{6},\bar{2}\}$ and $w=\{1,2,4,6,\bar{5},\bar{3},3,5,\bar{6},\bar{4},\bar{2},\bar{1}\}$. Then
$\gl_v=(6,5,5,5,4,1)$, $\gm=\gl'_v=(5,3,2,1)$, $\gl_w=(4,3,2,1)$, and $\gl=\gl'_w=(3,1)$. Thus $l(w)=|\gl'_w|=4$. The set of excited shifted Young diagrams $\ce'_{\gl}(\gm)$ appears in Figure \ref{f.eyd_alld}. For any $C\in\ce'_{\gl}(\gm)$ and $(i,j)\in C$, a simple method of finding the indices $v_{n+i}$ and $v_{n+j+1}$ of Theorem \ref{t.ktheory_eyd_bcd} is to label the rows of $C$, from top to bottom, with the numbers $v_{n+1},v_{n+2},\ldots$, and the columns, from left to right, with the numbers $v_{n+2},v_{n+3},\ldots$. For both row and column labels, $\bar{x}$ is replaced by $-x$. Then $v_{n+i}$ and $v_{n+j+1}$ are the row $i$ and column $j$ labels respectively. For example, for
\begin{equation*}
C=\
\begin{tikzpicture}[scale=.6,every node/.style={scale=1},
    baseline={([yshift=-.5ex]current bounding box.center)}]
     \foreach \pos in
        {{(1,4)},{(2,4)},{(4,3)},{(4,1)}}
            {\draw[fill=blue!30] \pos +(-.5,-.5) rectangle ++(.5,.5);}
     \foreach \pos in
          { {(1,4)},{(2,4)},{(3,4)},{(4,4)},{(5,4)},
            {(2,3)},{(3,3)},{(4,3)},
            {(3,2)},{(4,2)},
            {(4,1)}
          }
     {\draw \pos +(-.5,-.5) rectangle ++(.5,.5);}
    \draw  (0,1) node{$5$};  \draw  (0,2) node{$4$};  \draw  (0,3) node{$3$}; \draw  (0,4) node{$1$};
    \draw  (1,5) node{$3$}; \draw  (2,5) node{$4$}; \draw  (3,5) node{$5$};
    \draw  (4,5) node{$-6$}; \draw  (5,5) node{$-2$};
\end{tikzpicture}
\in\ce'_{\gl}(\gm),
\end{equation*}
we have
\[
\prod_{(i,j)\in C}\left(e^{-(\gre_{v_{n+i}}+\gre_{v_{n+j+1}})}-1\right)
=\left(e^{-(\gre_1+\gre_3)}-1\right)\left(e^{-(\gre_1+\gre_4)}-1\right)
\left(e^{-(\gre_3-\gre_6)}-1\right)\left(e^{-(\gre_5-\gre_6)}-1\right).
\]
By Theorem \ref{t.ktheory_eyd_bcd},
{\allowdisplaybreaks
\begin{align*}
    i_v^*[\co_{X^w}]&
    =
    \left(e^{-(\gre_1+\gre_3)}-1\right)\left(e^{-(\gre_1+\gre_4)}-1\right)
    \left(e^{-(\gre_1+\gre_5)}-1\right)\left(e^{-(\gre_3+\gre_4)}-1\right)\\
    &+\left(e^{-(\gre_1+\gre_3)}-1\right)\left(e^{-(\gre_1+\gre_4)}-1\right)
    \left(e^{-(\gre_1+\gre_5)}-1\right)\left(e^{-(\gre_5-\gre_6)}-1\right)\\
    &+\left(e^{-(\gre_1+\gre_3)}-1\right)\left(e^{-(\gre_1+\gre_4)}-1\right)
    \left(e^{-(\gre_3-\gre_6)}-1\right)\left(e^{-(\gre_3+\gre_4)}-1\right)\\
    &+\left(e^{-(\gre_1+\gre_3)}-1\right)\left(e^{-(\gre_1+\gre_4)}-1\right)
    \left(e^{-(\gre_3-\gre_6)}-1\right)\left(e^{-(\gre_5-\gre_6)}-1\right)\\
    &+\left(e^{-(\gre_1+\gre_3)}-1\right)\left(e^{-(\gre_3+\gre_5)}-1\right)
    \left(e^{-(\gre_3-\gre_6)}-1\right)\left(e^{-(\gre_5-\gre_6)}-1\right)\\
    &+\left(e^{-(\gre_1+\gre_3)}-1\right)\left(e^{-(\gre_1+\gre_4)}-1\right)
    \left(e^{-(\gre_1+\gre_5)}-1\right)\left(e^{-(\gre_3-\gre_6)}-1\right)
    \left(e^{-(\gre_3+\gre_4)}-1\right)\\
    &+\left(e^{-(\gre_1+\gre_3)}-1\right)\left(e^{-(\gre_1+\gre_4)}-1\right)
    \left(e^{-(\gre_1+\gre_5)}-1\right)  \left(e^{-(\gre_3+\gre_4)}-1\right)\left(e^{-(\gre_5-\gre_6)}-1\right)\\
    &+\left(e^{-(\gre_1+\gre_3)}-1\right)\left(e^{-(\gre_1+\gre_4)}-1\right)
    \left(e^{-(\gre_1+\gre_5)}-1\right)\left(e^{-(\gre_3-\gre_6)}-1\right)
    \left(e^{-(\gre_5-\gre_6)}-1\right)\\
    &+\left(e^{-(\gre_1+\gre_3)}-1\right)\left(e^{-(\gre_1+\gre_4)}-1\right)
    \left(e^{-(\gre_3-\gre_6)}-1\right)  \left(e^{-(\gre_3+\gre_4)}-1\right)\left(e^{-(\gre_5-\gre_6)}-1\right)\\
    &+\left(e^{-(\gre_1+\gre_3)}-1\right)\left(e^{-(\gre_1+\gre_4)}-1\right)
    \left(e^{-(\gre_3+\gre_5)}-1\right)
    \left(e^{-(\gre_3-\gre_6)}-1\right)\left(e^{-(\gre_5-\gre_6)}-1\right)\\
    &+\left(e^{-(\gre_1+\gre_3)}-1\right)\left(e^{-(\gre_1+\gre_4)}-1\right)
    \left(e^{-(\gre_1+\gre_5)}-1\right)\left(e^{-(\gre_3-\gre_6)}-1\right)
    \left(e^{-(\gre_3+\gre_4)}-1\right)\left(e^{-(\gre_5-\gre_6)}-1\right)
\end{align*}
}
\end{example}

\begin{example}
    Let $n=5$. In type $B_n$, let $v=\{2,\bar{5},\bar{4},\bar{3},\bar{1},1,3,4,5,\bar{2}\}$ and $w=\{1,2,4,\bar{5},\bar{3},3,5,\bar{4},\bar{2},\bar{1}\}$. Then $v$ and $w$ are identified with the corresponding elements of $W^{P_{n+1}}$ in Example \ref{ex.restriction_D} (see Section \ref{ss.strict_partitions_bcd}). By Proposition \ref{p.BfromD_restriction}, $i_v^*[\co_{X^w}]$ can be obtained from the same expression in Example \ref{ex.restriction_D} by replacing each $\gre_6$ by 0.
\end{example}

Theorem \ref{t.ktheory_eyd_bcd} is a reformulation of Theorem \ref{t.pullback}, in which the indexing set $T(w,\vs)$ and integer $r(c)$ of the latter theorem are expressed in terms of excited shifted Young diagrams.  These replacements are described explicitly in Proposition \ref{p.subsequences_bcd}.

Let $w,v\in W^{P_n}$, and let $\gl=\gl'_w$ and $\gm=\gl'_v$ be the corresponding strict partitions.
Form a reflection-valued shifted tableau $T'_{\gm}$ as follows:
\begin{description}
    \item[$B_n$, $C_n$] Fill each box $(i,j)$ of $D'_{\gm}$ with the reflection $s_{n+i-j}$.

    \item[$D_{n}$] Fill each box $(i,i)$ of $D'_{\gm}$ with $s_{n}$ if $i$ is odd or $s_{n-1}$ if $i$ is even; fill each box $(i,j)$, $i<j$, with $s_{n+i-(j+1)}$.

\end{description}
Then $v=s_{i_1}\cdots s_{i_l}$, where $s_{i_1},\ldots, s_{i_l}$ are the entries of $T'_{\gm}$ read from right to left, beginning with the bottom row, then the next row up, etc. This decomposition is reduced. To any subset $C$ of $D'_{\gm}$, form the subsequence $\vs_C=(s_{j_1},\ldots,s_{j_q})$ of $(s_{i_1},\ldots,s_{i_l})$ whose entries lie in the set $C$ of boxes of $T'_{\gm}$. If $C$ and $D$ are different subsets of $D'_\gm$, then we regard $\vs_C$ and $\vs_D$ as different subsequences of $(s_{i_1},\ldots,s_{i_l})$, even if they have the same entries.

 \begin{example}\label{ex.reduceddecompbc}
Let $n=6$. For $v=\{3,6,\bar{5},\bar{4},\bar{2},\bar{1},1,2,4,5,\bar{6},\bar{3}\}\in W^{P_n}$ in $C_n$ and\\ $v=\{3,6,\bar{5},\bar{4},\bar{2},\bar{1},7,1,2,4,5,\bar{6},\bar{3}\}\in W^{P_n}$ in $B_n$, $\gl_v=(6,6,5,5,4,2)$, $\gm=\gl'_v=(6,5,3,2)$,
 \begin{equation*}
T'_{\gm}=
\begin{tikzpicture}[
    scale=.5,
    every node/.style={scale=.8},
    baseline={([yshift=-.5ex]current bounding box.center)}]
         \foreach \pos / \label in
          { {(1,4)}/{s_6},{(2,4)}/{s_5},{(3,4)}/{s_4},{(4,4)}/{s_3},{(5,4)}/{s_2},{(6,4)}/{s_1},
            {(2,3)}/{s_6},{(3,3)}/{s_5},{(4,3)}/{s_4},{(5,3)}/{s_3},{(6,3)}/{s_2},
            {(3,2)}/{s_6},{(4,2)}/{s_5},{(5,2)}/{s_4},
            {(4,1)}/{s_6},{(5,1)}/{s_5}
          }
              {
                \draw \pos +(-.5,-.5) rectangle ++(.5,.5);
                \draw \pos node{$\label$};
              }
\end{tikzpicture}
\end{equation*}
and $s_5s_6s_4s_5s_6s_2s_3s_4s_5s_6s_1s_2s_3s_4s_5s_6$ is a reduced decomposition for $v$.
\end{example}

 \begin{example}\label{ex.reduceddecompd}
 Let $n=7$. For $v=\{3,6,7,\bar{5},\bar{4},\bar{2},\bar{1},1,2,4,5,\bar{7},\bar{6},\bar{3}\}\in W^{P_n}$ in $D_n$,\\ $\gl_v=(7,7,6,6,4,4,2)$ and $\gm=\gl'_v=(6,5,3,2)$,
 \begin{equation*}
T'_{\gm}=
\begin{tikzpicture}[
    scale=.5,
    every node/.style={scale=.8},
    baseline={([yshift=-.5ex]current bounding box.center)}]
         \foreach \pos / \label in
          { {(1,4)}/{s_7},{(2,4)}/{s_5},{(3,4)}/{s_4},{(4,4)}/{s_3},{(5,4)}/{s_2},{(6,4)}/{s_1},
            {(2,3)}/{s_6},{(3,3)}/{s_5},{(4,3)}/{s_4},{(5,3)}/{s_3},{(6,3)}/{s_2},
            {(3,2)}/{s_7},{(4,2)}/{s_5},{(5,2)}/{s_4},
            {(4,1)}/{s_6},{(5,1)}/{s_5}
          }
              {
                \draw \pos +(-.5,-.5) rectangle ++(.5,.5);
                \draw \pos node{$\label$};
              }
\end{tikzpicture}
\end{equation*}
and $s_5s_6s_4s_5s_7s_2s_3s_4s_5s_6s_1s_2s_3s_4s_5s_7$ is a reduced decomposition for $v$.
\end{example}

\begin{Prop}\label{p.subsequences_bcd}
Denote the reduced  decomposition  $(s_{i_1},\cdots, s_{i_l})$ for $v$ obtained above by $\vs_v$. By definition,
$T(w,\vs_v)=\{\vs_C\mid C\subseteq D'_\gm, H_{\vs_C}=H_w\}$. We have
\begin{itemize}
    \item[(i)] $T(w,\vs_v)=\{\vs_C\mid C\in\ce'_{\gl}(\gm)\}$.

    \item[(ii)] Let $(i,j)$ be the box of $T'_{\gm}$ containing $s_{i_c}$. Define $\gre_{\bar{m}}=-\gre_m$ for $1\leq m\leq n$. Then
            {\allowdisplaybreaks
            \begin{equation*}
            \begin{split}
                & B_n:\quad r(c)=2^{-\gd_{ij}}(\gre_{v_{n+i+1}} +\gre_{v_{n+j+1}})\\
                & C_n:\quad r(c)=\gre_{v_{n+i}} +\gre_{v_{n+j}}\\
                & D_n:\quad r(c)=\gre_{v_{n+i}} +\gre_{v_{n+j+1}}
            \end{split}
            \end{equation*}
            }
\end{itemize}
\end{Prop}

The reduced decomposition $\vs_v$ is due to Ikeda and Naruse \cite{IkNa:09}. Proposition \ref{p.subsequences_bcd}(ii) is as well, although our expressions for the constants $r(c)$ are different than theirs. A version of
 Proposition \ref{p.subsequences_bcd}(i) which involves the nil-Coxeter algebra and (what we call) reduced excited Young diagrams is also proved in \cite{IkNa:09}.

Proposition \ref{p.subsequences_bcd} is the counterpart for types $B_n$, $C_n$, and $D_n$ of Proposition \ref{p.subsequences}. The proof of part (i) carries over with very minor modifications. We omit the details. We prove part (ii) below.

 \begin{Def}\label{d.energy_bcd}
Let $C$ be a subset of $D_\gm$ such that $H_{\vs_C}=H_w$.
Define
\begin{enumerate}
    \item $|C|=$ number of boxes of $C$

    \item
        $B_n$, $C_n$: $e_1(C)=(1/2)(\sum_{(i,j)\in C}(i+j)-\sum_{(i,j)\in D_\gl}(i+j))$\\
        $D_n$: $e_1(C)=(1/2)(\sum_{(i,j)\in C, i<j}(i+j)+\sum_{(i,i)\in C}i-\sum_{(i,j)\in D_\gl, i<j}(i+j)-\sum_{(i,i)\in D_\gl}i)$


    \item $e_2(C)=|C|-|D'_{\gl}|=|C|-|\gl|$
\end{enumerate}
We call $e_1(C)$ and $e_2(C)$ the \textbf{type 1 energy} and \textbf{type 2 energy} of $C$ respectively.
\end{Def}


\subsubsection{Proof of Proposition \ref{p.subsequences_bcd}(ii)}

\textbf{1. Type $C_n$}.
Recall that entry of box $(l,m)$ of $T'_{\gl_v}$ is $s_{n+l-m}$.
By \eqref{e.lambdav},
the rightmost box of row $i$ lies in column $v_{n+1-i}-(n+1-i)$; by Lemma \ref{l.trans_partition}, the lowest box of column $j$, assuming that this box lies to the right of the `descending staircase', lies in row $n+j-v_{n+j}$. The entries of these two boxes in $T'_{\gm}$ are $s_a$ and $s_b$ respectively, where
\begin{align}
    a&=n+i-(v_{n+1-i}-(n+1-i))=2n+1-v_{n+1-i}=\bar{v_{n+1-i}}=v_{n+i}.\label{e.rightboxc}\\
    b&= n+((n+j)-v_{n+j})-j=(2n+1)-v_{n+j}-1=\bar{v_{n+j}}-1.\label{e.bottomboxc}
\end{align}
We consider three cases: $i<j$ and
$(j,j)\in D'_{\gm}$, $i<j$ and $(j,j)\not\in D'_{\gm}$, and $i=j$.

\vspace{1em}

\noindent {\em Case 1. $i<j$ and $(j,j)\in D'_{\gm}$}. Let $s_x, s_y$ denote the reflections which lie in the rightmost box of rows $i$ and $j$ respectively.
Figure \ref{f.Case1_bc} shows some of the entries of $T'_{\gm}$.
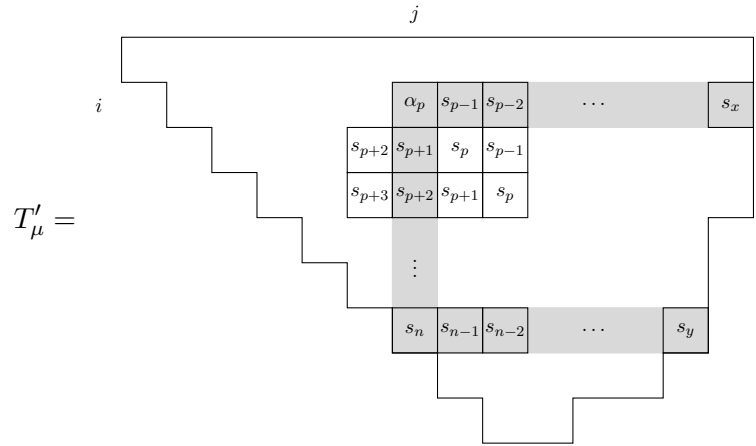
\begin{figure}[htbp]
\begin{equation*}
T'_{\gm}=\,
    \begin{tikzpicture}[xscale=.6,yscale=.6,
    baseline={([yshift=-.5ex]current bounding box.center)},
    every node/.style={scale=.7}]
         \draw  (0,8) node{$i$};
         \draw  (7,10) node{$j$};
          \draw[gray!30,fill=gray!30] (14.5,8.5) -- ++(-8,0) -- ++(0,-6) -- ++(7,0) -- ++(0,1) -- ++(-6,0) --  ++(0,4) -- ++(7,0) -- cycle;
         \foreach \pos / \label in
          { {(7,8)}/{\ga_p},{(8,8)}/{s_{p-1}},{(9,8)}/{s_{p-2}},{(14,8)}/{s_{x}},
            {(6,7)}/{s_{p+2}},{(7,7)}/{s_{p+1}},{(8,7)}/{s_{p}},{(9,7)}/{s_{p-1}},
            {(6,6)}/{s_{p+3}},{(7,6)}/{s_{p+2}},{(8,6)}/{s_{p+1}},{(9,6)}/{s_{p}},
            {(7,3)}/{s_n},{(8,3)}/{s_{n-1}},{(9,3)}/{s_{n-2}},{(13,3)}/{s_y}
          }
              {
                \draw \pos +(-.5,-.5) rectangle ++(.5,.5);
                \draw \pos node{$\label$};
              }

        \draw (0.5,9.5)
            -- ++(0,-1) -- ++(1,0)
            -- ++(0,-1) -- ++(1,0)
            -- ++(0,-1) -- ++(1,0)
            -- ++(0,-1) -- ++(1,0)
            -- ++(0,-1) -- ++(1,0)
            -- ++(0,-1) -- ++(1,0)
            -- ++(0,-1) -- ++(1,0)
            -- ++(0,-1) -- ++(1,0)
            -- ++(0,-1) -- ++(1,0)
            -- ++(1,0) -- ++(0,1)
            -- ++(2,0) -- ++(0,1)
            -- ++(1,0) -- ++(0,3)
            -- ++(1,0) -- ++ (0,4)
            -- cycle;
        \draw (7,4.5) node {$\vdots$};
        \draw (11,3) node {$\cdots$};
        \draw (11,8) node {$\cdots$};

    \end{tikzpicture}
\end{equation*}
\caption{\label{f.Case1_bc} Type $C_n$, Case 1}
\end{figure}
In the expression $r(c)=s_{i_1}s_{i_2}\cdots s_{i_{c-1}}(\ga_{i_c})=s_{i_1}s_{i_2}\cdots s_{p-2}s_{p-1}(\gre_p-\gre_{p+1})$, the reflections $s_{i_j}$ which lie outside of the shaded boxes can be removed. We have
{\allowdisplaybreaks
\begin{align*}
    r(c)
    &= s_y\cdots s_{n-2}s_{n-1}s_n\cdots s_{p+2} s_{p+1} s_x\cdots s_{p-2} s_{p-1} (\gre_p-\gre_{p+1})\\
    &= s_y\cdots s_{n-2}s_{n-1}s_n\cdots s_{p+2} s_{p+1} s_x\cdots s_{p-2} (\gre_{p-1}-\gre_{p+1})\\
    &= s_y\cdots s_{n-2}s_{n-1}s_n\cdots s_{p+2} s_{p+1} s_x\cdots (\gre_{p-2}-\gre_{p+1})\\
    &= s_y\cdots s_{n-2}s_{n-1}s_n\cdots s_{p+2} s_{p+1} (\gre_{x}-\gre_{p+1})\\
    &= s_y\cdots s_{n-2}s_{n-1}s_n\cdots s_{p+2} (\gre_{x}-\gre_{p+2})\\
    &= s_y\cdots s_{n-2}s_{n-1}s_n\cdots (\gre_{x}-\gre_{p+3})\\
    &= s_y\cdots s_{n-2}s_{n-1}s_n(\gre_{x}-\gre_{n})\\
    &= s_y\cdots s_{n-2}s_{n-1}(\gre_{x}+\gre_{n})\\
    &= s_y\cdots s_{n-2}(\gre_{x}+\gre_{n-1})\\
    &= s_y\cdots (\gre_{x}+\gre_{n-2})
    = \gre_{x}+\gre_{y}=\gre_{v_{n+i}}+\gre_{v_{n+j}},
\end{align*}
}
where the last equality is due to \eqref{e.rightboxc}.

\noindent {\em Case 2.  $i<j$ and $(j,j)\not\in D'_{\gm}$}.
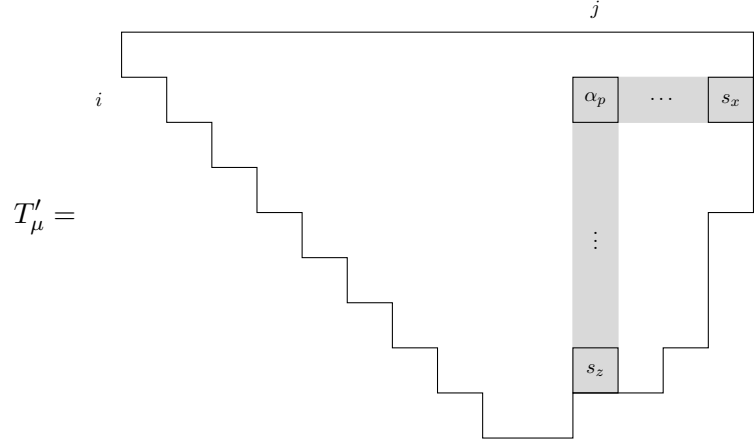
\begin{figure}[htbp]
\begin{equation*}
T'_{\gm}=\,
    \begin{tikzpicture}[xscale=.6,yscale=.6,
    baseline={([yshift=-.5ex]current bounding box.center)},
    every node/.style={scale=.7}]
         \draw  (0,8) node{$i$};
         \draw  (11,10) node{$j$};
          \draw[gray!30,fill=gray!30] (14.5,8.5) -- ++(-4,0) -- ++(0,-7) -- ++(1,0) -- ++(0,6) -- ++(3,0) --  cycle;
         \foreach \pos / \label in
          { {(14,8)}/{s_{x}}, {(11,2)}/{s_z}, {(11,8)}/{\ga_p}
          }
              {
                \draw \pos +(-.5,-.5) rectangle ++(.5,.5);
                \draw \pos node{$\label$};
              }

        \draw (0.5,9.5)
            -- ++(0,-1) -- ++(1,0)
            -- ++(0,-1) -- ++(1,0)
            -- ++(0,-1) -- ++(1,0)
            -- ++(0,-1) -- ++(1,0)
            -- ++(0,-1) -- ++(1,0)
            -- ++(0,-1) -- ++(1,0)
            -- ++(0,-1) -- ++(1,0)
            -- ++(0,-1) -- ++(1,0)
            -- ++(0,-1) -- ++(1,0)
            -- ++(1,0) -- ++(0,1)
            -- ++(2,0) -- ++(0,1)
            -- ++(1,0) -- ++(0,3)
            -- ++(1,0) -- ++ (0,4)
            -- cycle;
        \draw (11,5) node {$\vdots$};
        \draw (12.5,8) node {$\cdots$};

    \end{tikzpicture}
\end{equation*}
\caption{\label{f.Case2_bc}Type $C_n$, Case 2}
\end{figure}
As in the above calculation, reflections lying outside of the shaded region can be removed from the expression $r(c)=s_{i_1}s_{i_2}\cdots s_{i_{c-1}}(\ga_{i_c})$, where $\ga_{i_c}=\ga_{p}=\gre_p-\gre_{p+1}$. Let $s_x, s_z$ be the reflections which lie in the rightmost box of row $i$ and bottom box of column $j$ respectively. One checks that  $r(c)=\gre_{x}-\gre_{z+1}=\gre_{v_{n+i}}-\gre_{\bar{v_{n+j}}}=\gre_{v_{n+i}}+\gre_{v_{n+j}}$,
using \eqref{e.rightboxc} and \eqref{e.bottomboxc}.

\vspace{1em}

\noindent {\em Case 3: $i=j$}. In this case, $r(c)
    = \cdots s_{x}\cdots s_{n-2}s_{n-1}(2\gre_n)=2\gre_x=\gre_{v_{n+i}}+\gre_{v_{n+i}}$.

\vspace{1em}
\noindent\textbf{2. Type $B_n$}.
The Weyl group in type $B_n$ is the same as in type $C_n$, and the simple roots other than $\ga_n$ are identical as well. It follows that if $i<j$, then $r(c)$ is the same value computed in type $C_n$. If $i=j$, then $ r(c) = \cdots s_{x}\cdots s_{n-2}s_{n-1}(\gre_n)=\gre_x=2^{-1}(\gre_x+\gre_x)$.
 \vspace{1em}

\noindent\textbf{3. Type $D_n$}.

Recall that for each element $v\in W^{P_n}$, $D'_{\gm}=D'_{\gl'_v}$ is formed by removing all boxes $(i,j)$ of $D_{\gl_v}$ such that $i\geq j$. The column number of any box in $D'_{\gl'_v}$ is one less than the column number of the corresponding box of $D_{\gl_v}$. Thus in order to `place' a box in $D'_{\gl'_v}$ into its appropriate box in $D_{\gl_v}$, one must add one to its column index.

Recall also that for $l<m$ and for $l=m$, $l$ even, the entry of box $(l,m)$ of $T'_{\gl_v}$ is $s_{n+l-(m+1)}$.
By \eqref{e.lambdav},
the rightmost box of row $i$ lies in column $v_{n+1-i}-(n+1-i)$ of $D_{\gl_v}$, and thus in column $v_{n+1-i}-(n+1-i)-1$ of $D'_{\gl'_v}$; by Lemma \ref{l.trans_partition}, the lowest box of column $j$, assuming that this box lies to the right of the `descending staircase', lies in row $n+j+1-v_{n+j+1}$. The entries of these two boxes in $T'_\gm$ are $s_a$ and $s_b$ respectively, where
\begin{align}
    a&=n+i-(v_{n+1-i}-(n+1-i)-1+1)=2n+1-v_{n+1-i}=\bar{v_{n+1-i}}=v_{n+i}\label{e.rightboxd}\\
    b&=n+((n+j+1)-v_{n+j+1})-(j+1)=(2n+1)-v_{n+j+1}-1=\bar{v_{n+j+1}}-1\label{e.bottomboxd}
\end{align}
We consider four possibilities for $(i,j)\in D'_{\gl'_v}$:
$i<j$, $(j+1,j+1)\in D'_{\gl'_v}$, $j$ odd; $i<j$, $(j+1,j+1)\in D'_{\gl'_v}$, $j$ even; $i<j$, $(j+1,j+1)\not\in D'_{\gl'_v}$; and $i=j$.

\vspace{1em}

\noindent {\em Cases 1 and 2: $i<j$ and $(j+1,j+1)\in D'_{\gm}$}.
 Let $s_x, s_y$ denote the reflections which lie in the rightmost box of row $i$ and row $j+1$ respectively. Reflections $s_{i_k}$ not lying in the shaded boxes of Figure \ref{f.Case1_d} can be removed from the expression $r(c)=s_{i_1}s_{i_2}\cdots s_{i_{c-1}}(\ga_{i_c})$, where $\ga_{i_c}=\ga_{p}=\gre_p-\gre_{p+1}$.
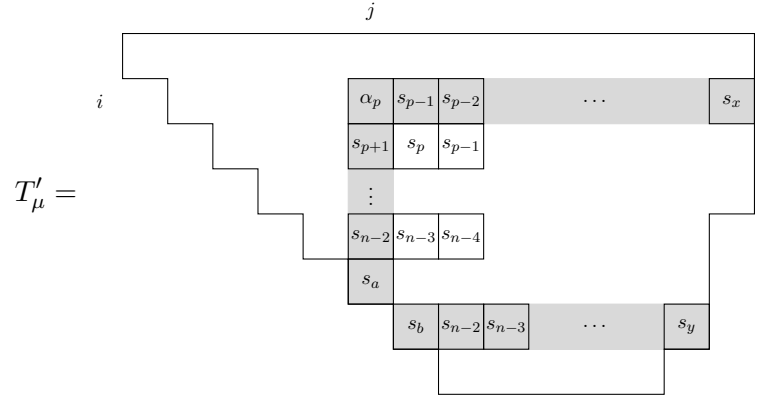
\begin{figure}[htbp]
\begin{equation*}
T'_{\gm}=\,
    \begin{tikzpicture}[xscale=.6,yscale=.6,
    baseline={([yshift=-.5ex]current bounding box.center)},
    every node/.style={scale=.7}]
         \draw  (0,7) node{$i$};
         \draw  (6,9) node{$j$};
          \draw[gray!30,fill=gray!30] (14.5,7.5) -- ++(-9,0) -- ++(0,-5) -- ++(1,0) -- ++(0,4) -- ++(8,0) -- cycle;
         \draw[gray!30,fill=gray!30] (6.5,2.5) -- ++(7,0) -- ++(0,-1) -- ++(-7,0) -- cycle;
         \foreach \pos / \label in
          { {(6,7)}/{\ga_p},{(7,7)}/{s_{p-1}},{(8,7)}/{s_{p-2}},{(14,7)}/{s_{x}},
            {(6,6)}/{s_{p+1}},{(7,6)}/{s_{p}},{(8,6)}/{s_{p-1}},
            {(6,4)}/{s_{n-2}},{(7,4)}/{s_{n-3}},{(8,4)}/{s_{n-4}},
            {(6,3)}/{s_{a}},
            {(7,2)}/{s_b},{(8,2)}/{s_{n-2}},{(9,2)}/{s_{n-3}},{(13,2)}/{s_y}
          }
              {
                \draw \pos +(-.5,-.5) rectangle ++(.5,.5);
                \draw \pos node{$\label$};
              }
        \draw (6,5) node{\vdots};
        \draw (11,7) node{$\cdots$};
        \draw (11,2) node{$\cdots$};

        \draw (0.5,8.5)
            -- ++(0,-1) -- ++(1,0)
            -- ++(0,-1) -- ++(1,0)
            -- ++(0,-1) -- ++(1,0)
            -- ++(0,-1) -- ++(1,0)
            -- ++(0,-1) -- ++(1,0)
            -- ++(0,-1) -- ++(1,0)
            -- ++(0,-1) -- ++(1,0)
            -- ++(0,-1) -- ++(1,0)
            -- ++(4,0) -- ++(0,1)
            -- ++(1,0) -- ++(0,3)
            -- ++(1,0) -- ++ (0,4)
            -- cycle;
    \end{tikzpicture}
\end{equation*}
\caption{\label{f.Case1_d}Type $D_n$, Cases 1 and 2}
\end{figure}
If $j$ is odd then $a=n$ and $b=n-1$; otherwise $a=n-1$ and $b=n$. In either case, $r(c)=\gre_{x}+\gre_{y}=\gre_{v_{n+i}}+\gre_{v_{n+j+1}}$, where the last equality is due to \eqref{e.rightboxd}.

\vspace{1em}

\noindent {\em Case 3: $i<j$ and $(j+1,j+1)\not\in D'_{\gm}$}. Let $s_x, s_z$ be the reflections which lie in the rightmost box of row $i$ and bottom box or column $j$ respectively. Reflections $s_{i_k}$ not in the shaded boxes of Figure \ref{f.Case3_d} can be removed from the expression $r(c)=s_{i_1}s_{i_2}\cdots s_{i_{c-1}}(\ga_{i_c})$, where $\ga_{i_c}=\ga_{p}=\gre_p-\gre_{p+1}$. One checks that $r(c)=\gre_x-\gre_{z+1}=\gre_{v_{n+i}}-\gre_{\bar{v_{n+j+1}}}=\gre_{v_{n+i}}+\gre_{v_{n+j+1}}$, where the last equality is due to \eqref{e.rightboxd} and \eqref{e.bottomboxd}.
\begin{figure}[htbp]
\begin{equation*}
T'_{\gm}=\,
    \begin{tikzpicture}[xscale=.6,yscale=.6,
    baseline={([yshift=-.5ex]current bounding box.center)},
    every node/.style={scale=.7}]
         \draw  (0,7) node{$i$};
         \draw  (9,9) node{$j$};
          \draw[gray!30,fill=gray!30] (14.5,7.5) -- ++(-6,0) -- ++(0,-7) -- ++(1,0) -- ++(0,6) -- ++(5,0) -- cycle;
         \foreach \pos / \label in
          { {(14,7)}/{s_{x}},{(9,7)}/{\ga_p},{(9,1)}/{s_z}
          }
              {
                \draw \pos +(-.5,-.5) rectangle ++(.5,.5);
                \draw \pos node{$\label$};
              }
        \draw (9,4) node{\vdots};
        \draw (12,7) node{$\cdots$};

        \draw (0.5,8.5)
            -- ++(0,-1) -- ++(1,0)
            -- ++(0,-1) -- ++(1,0)
            -- ++(0,-1) -- ++(1,0)
            -- ++(0,-1) -- ++(1,0)
            -- ++(0,-1) -- ++(1,0)
            -- ++(0,-1) -- ++(1,0)
            -- ++(0,-1) -- ++(1,0)
            -- ++(0,-1) -- ++(1,0)
            -- ++(4,0) -- ++(0,1)
            -- ++(1,0) -- ++(0,3)
            -- ++(1,0) -- ++ (0,4)
            -- cycle;


    \end{tikzpicture}
\end{equation*}
\caption{\label{f.Case3_d}Type $D_n$, Case 3}
\end{figure}
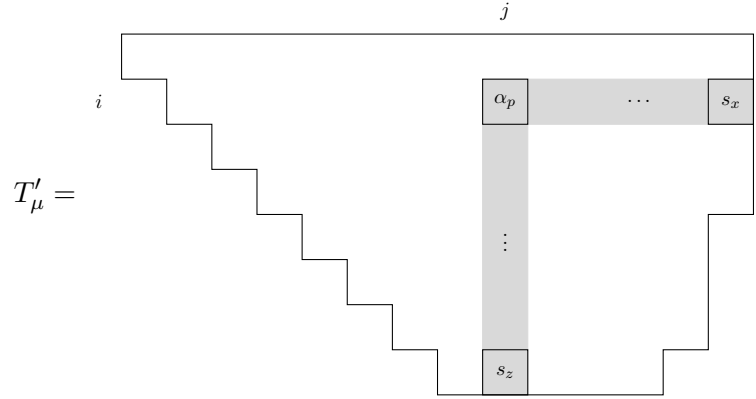

\vspace{1em}

\noindent {\em Case 4: $i=j$}. The analysis is similar to the other cases.

\subsection{Relationship with previously obtained restriction formulas} \label{ss.previous}

The formulas in this paper for $i_v^*[\co_{X^w}]$ for types $A_n$ and $C_n$, expressed in terms of set-valued tableaux, appeared earlier in \cite{Kre:05} and \cite{Kre:06}. A restriction formula in type $A_n$ can also be obtained by specializing the factorial Grothendieck polynomials of \cite{Mcn:06}. This restriction formula is generalized to types  $B_n$, $C_n$, and $D_n$ in \cite{IkNa:11}. The main difference in the restriction formulas of this paper is that they are positive, meaning that they result in $(-1)^{l(w)}$ times sums of monomials in $e^{-\ga}-1$, where $\ga$ is a positive root (see Theorem \ref{t.pullback} and Remark \ref{r.tangent}). In the two examples below, we compare the restriction formulas of this paper and of  \cite{IkNa:11}. 

The formulas of \cite{IkNa:11}, whose notation and definitions we adopt in this section, use the binary operators $\oplus$ and $\ominus$ (see \cite{FoKi:94}, \cite{FoKi:96}):
\begin{equation*}
a\oplus b=a+b+\gb ab\ \text{ and }\ a\ominus b=\frac{a-b}{1+\gb b},
\end{equation*}
where $\gb$ is a parameter. Note that $\ominus$ is the inverse of $\oplus$. Define $\oplus x= 0\oplus x=x$, $\ominus x=0\ominus x$, and set the parameter $\gb$ equal to $-1$. One checks that
$(1-e^x)\oplus (1-e^y)=1-e^{x+y}$, $(1-e^x)\ominus (1-e^y)=1-e^{x-y}$, and  $\ominus (1-e^{x})=1-e^{-x}$.

\begin{example}
We work in type $C_n$, $n=2$. Let $v=w=\{2,4,1,3\}=\{2,\bar{1},1,\bar{2}\}\in W^{P_n}$. Then $\gl_v=(2,1)$, $\gm=\gl'_v=(2,0)$, and $\gl=\gl'_w=(2,0)$.

We first compute $i_v^*[\co_{X^w}]$ using Theorem \ref{t.ktheory_eyd_bcd}. The set $\ce'_{\gl}(\gm)$ consists of the single element
\begin{center}
\begin{tikzpicture}[scale=.5,every node/.style={scale=1},
    baseline={([yshift=-.5ex]current bounding box.center)}]
    \foreach \pos in
    {{(1,1)},{(2,1)}}
        {\draw[fill=blue!30] \pos +(-.5,-.5) rectangle ++(.5,.5);}
     \foreach \pos in
      { {(1,1)},{(2,1)}
      }
          {\draw \pos +(-.5,-.5) rectangle ++(.5,.5);}
    \draw  (0,1) node{$1$};
    \draw  (1,2) node{$1$}; \draw  (2,2) node{$-2$};
\end{tikzpicture}
\end{center}
Thus
\begin{equation}\label{e.pull_smallex}
    i_v^*[\co_{X^w}]=(e^{-2\gre_1}-1)(e^{-(\gre_1-\gre_2)}-1).
\end{equation}

We next compute $i_v^*[\co_{X^w}]$ using the methods of \cite{IkNa:11}.
The elements of $\ce_n^I(\gl)$ \cite[9.2]{IkNa:11} are
\begin{center}
    \begin{tikzpicture}
\newcommand{\eydin}[3]
    {
        \begin{tikzpicture}[scale=.4]
            \foreach \pos in
            {{#1},{#2}}
                {\draw[fill=blue!30] \pos +(-.5,-.5) rectangle ++(.5,.5);}
            \foreach \pos in
            {{#3}}
                {\draw \pos +(-.7,0) node{$\times$};}
             \foreach \pos in
              { {(1,2)},{(2,2)},{(3,2)},
                {(2,1)}, {(3,1)}
              }
                  {\draw \pos +(-.5,-.5) rectangle ++(.5,.5);}
        \end{tikzpicture}
      }

\newcommand{\eyd}[2]
    {
        \begin{tikzpicture}[scale=.4]
            \foreach \pos in
            {{#1},{#2}}
                {\draw[fill=blue!30] \pos +(-.5,-.5) rectangle ++(.5,.5);}
             \foreach \pos in
              { {(1,2)},{(2,2)},{(3,2)},
                {(2,1)}, {(3,1)}
              }
                  {\draw \pos +(-.5,-.5) rectangle ++(.5,.5);}
        \end{tikzpicture}
      }

    \node[rectangle] (first) {\eyd{(1,2)}{(2,2)}};
    \node[rectangle] (second) [right=of first] {\eydin{(1,2)}{(3,1)}{(2,2)}};
    \node[rectangle] (third) [right=of second] {\eydin{(2,1)}{(3,1)}{(1,2)}};

\end{tikzpicture}

\end{center}
By \cite[(9.6)]{IkNa:11},
\begin{align*}
    \GC_\gl^{(n)}(x\,|\,b)= & (x_1\oplus x_1)(x_1\oplus x_2)+(x_1\oplus x_1)(x_2\oplus b_1)(1-x_1\oplus x_2)\\
    &+(x_2\oplus x_2)(x_2\oplus b_1)(1-x_1\oplus x_1).
\intertext{Substituting $b_\gm=(\ominus b_{\gm_1},\ldots,\ominus b_{\gm_r},0,0,0,\ldots)=(\ominus b_{2},0,0,\ldots)$ for $x$,}
    \GC_\gl^{(n)}(b_\gm\,|\,b)= & (\ominus b_{2}\ominus b_{2})(\ominus b_{2}\oplus 0)+(\ominus b_{2}\ominus b_{2})(0\oplus b_1)(1-\ominus b_{2}\oplus 0)\\
    &+(0\oplus 0)(0\oplus b_1)(1-\ominus b_{2}\ominus b_{2})\\
    = & (\ominus b_{2}\ominus b_{2})(\ominus b_{2})+(\ominus b_{2}\ominus b_{2})(b_1)(1-\ominus b_{2}).
\end{align*}
Substituting $1-e^{t_i}$ for $b_i$, $i=1,2$, and then replacing $t_1,t_2$ by $\gre_2,\gre_1$ respectively in order to account for the ordering of the roots of the Dynkin diagram in \cite[4.6]{IkNa:11}, we obtain
\begin{equation}\label{e.pull_smallex_in}
    i_v^*[\co_{X^w}] = (1-e^{-2\gre_1})(1-e^{-\gre_1})+(1-e^{-2\gre_1})(1-e^{\gre_2})(1-(1-e^{-\gre_1})).
\end{equation}
One checks that this agrees with \eqref{e.pull_smallex}.

\end{example}

\begin{example}\label{ex.in_d5}
Consider type $D_n$, $n=5$. Let $v=\{2,4,5,8,10,1,3,6,7,9\}=$\\ $\{2,4,5,\bar{3},\bar{1},1,3,\bar{5},\bar{4},\bar{2}\}$,  $w=\{1,2,5,7,8,3,4,6,9,10\}=\{1,2,5,\bar{4},\bar{3},3,4,\bar{5},\bar{2},\bar{1}\}\in W^{P_n}$. Then $\gl_v=(5,4,2,2,1)$, $\gm=\gl'_v=(4,2)$, $\gl_w=(3,3,2,0,0)$, $\gl=\gl'_w=(2,1)$.

We first compute $i_v^*[\co_{X^w}]$ using Theorem \ref{t.ktheory_eyd_bcd}. The set $\ce'_{\gl}(\gm)$ consists of the single element
\begin{center}
\begin{tikzpicture}[scale=.6,every node/.style={scale=1},
    baseline={([yshift=-.5ex]current bounding box.center)}]
    \foreach \pos in
    {{(1,2)},{(2,2)},{(2,1)}}
        {\draw[fill=blue!30] \pos +(-.5,-.5) rectangle ++(.5,.5);}
     \foreach \pos in
      { {(1,2)},{(2,2)},{(3,2)},{(4,2)},
        {(2,1)},{(3,1)}
      }
          {\draw \pos +(-.5,-.5) rectangle ++(.5,.5);}
    \draw  (0,1) node{$3$};  \draw  (0,2) node{$1$};
    \draw  (1,3) node{$3$}; \draw  (2,3) node{-$5$}; \draw  (3,3) node{-$4$}; \draw  (4,3) node{-$2$};
\end{tikzpicture}
\end{center}

Thus
\begin{equation}\label{e.pull_bigger}
    i_v^*[\co_{X^w}]=-(e^{-(\gre_1+\gre_3)}-1)(e^{-(\gre_1-\gre_5)}-1)
    (e^{-(\gre_3-\gre_5)}-1).
\end{equation}

We next compute $i_v^*[\co_{X^w}]$ using the methods of \cite{IkNa:11}. The elements of $\ce_n^I(\gl)$ \cite[9.2]{IkNa:11} are
\begin{center}
    \begin{tikzpicture}
\newcommand{\eydin}[5]
    {
        \begin{tikzpicture}[scale=.4]
            \foreach \pos in
            {{#1},{#2},{#3}}
                {\draw[fill=blue!30] \pos +(-.5,-.5) rectangle ++(.5,.5);}
            \foreach \pos in
            {{#4},{#5}}
                {\draw \pos +(-.7,0) node{$\times$};}
             \foreach \pos in
              { {(1,4)},{(2,4)},{(3,4)},{(4,4)},
                {(2,3)},{(3,3)},{(4,3)},
                {(3,2)},{(4,2)},
                {(4,1)}
              }
                  {\draw \pos +(-.5,-.5) rectangle ++(.5,.5);}
        \end{tikzpicture}
      }

\newcommand{\eyd}[3]
    {
        \begin{tikzpicture}[scale=.4]
            \foreach \pos in
            {{#1},{#2},{#3}}
                {\draw[fill=blue!30] \pos +(-.5,-.5) rectangle ++(.5,.5);}
             \foreach \pos in
              { {(1,4)},{(2,4)},{(3,4)},{(4,4)},
                {(2,3)},{(3,3)},{(4,3)},
                {(3,2)},{(4,2)},
                {(4,1)}
              }
                  {\draw \pos +(-.5,-.5) rectangle ++(.5,.5);}
        \end{tikzpicture}
      }

    \node[rectangle] (first) {\eyd{(1,4)}{(2,4)}{(2,3)}};
    \node[rectangle] (second) [right=of first] {\eydin{(1,4)}{(2,4)}{(4,1)}{(2,3)}{(2,3)}};
    \node[rectangle] (third) [right=of second] {\eydin{(1,4)}{(3,3)}{(4,1)}{(2,4)}{(2,4)}};
    \node[rectangle] (fourth) [right=of third] {\eydin{(1,4)}{(4,2)}{(4,1)}{(2,4)}{(3,3)}};
    \node[rectangle] (fifth) [right=of fourth] {\eydin{(3,2)}{(4,2)}{(4,1)}{(1,4)}{(1,4)}};

\end{tikzpicture}
\end{center}
By \cite[(9.6)]{IkNa:11},
\begin{align*}
    \GD_\gl^{(n)}(x\,|\,b)= & (x_1\oplus x_2)(x_1\oplus x_3)(x_2\oplus x_3)\\
    & + (x_1\oplus x_2)(x_1\oplus x_3)(x_4\oplus b_1)(1-x_2\oplus x_3)\\
    & + (x_1\oplus x_2)(x_2\oplus x_4)(x_4\oplus b_1)(1-x_1\oplus x_3)\\
    & + (x_1\oplus x_2)(x_3\oplus b_1)(x_4\oplus b_1)(1-x_1\oplus x_3)(1-x_2\oplus x_4)\\
    & + (x_3\oplus x_4)(x_3\oplus b_1)(x_4\oplus b_1)(1-x_1\oplus x_2).
\intertext{Substituting $b_\gm=(\ominus b_{\gm_1+1},\ldots,\ominus b_{\gm_r+1},0,0,0,\ldots)=(\ominus b_{5},\ominus b_{3},0,0,0,\ldots)$ for $x$,}
    \GD_\gl^{(n)}(b_\gm\,|\,b)= & (\ominus b_{5}\ominus b_3)(\ominus b_{5})(\ominus b_3)\\
    & + (\ominus b_{5}\ominus b_3)(\ominus b_{5})(b_1)(1-\ominus b_3)\\
    & + (\ominus b_{5}\ominus b_3)(\ominus b_3)(b_1)(1-\ominus b_{5})\\
    & + (\ominus b_{5}\ominus b_3)(b_1)(b_1)(1-\ominus b_{5})(1-\ominus b_3).
\end{align*}
Substituting $1-e^{t_i}$ for $b_i$, $i=1,\ldots,5$, and then replacing $t_1,t_2,t_3,t_4,t_5$ by $\gre_5,\gre_4,\gre_3,\gre_2,\gre_1$ respectively in order to account for the ordering of the roots of the Dynkin diagram in \cite[4.6]{IkNa:11}, we obtain
\begin{equation}\label{e.pull_bigger_in}
\begin{split}
    i_v^*[\co_{X^w}] = & (1-e^{-(\gre_1-\gre_3)})(1-e^{-\gre_1})(1-e^{-\gre_3})\\
    & + (1-e^{-(\gre_1-\gre_3)})(1-e^{-\gre_1})(1-e^{\gre_5})(1-(1-e^{-\gre_3}))\\
    & + (1-e^{-(\gre_1-\gre_3)})(1-e^{-\gre_3})(1-e^{\gre_5})(1-(1-e^{-\gre_1}))\\
    & + (1-e^{-(\gre_1-\gre_3)})(1-e^{\gre_5})(1-e^{\gre_5})(1-(1-e^{-\gre_1}))(1-(1-e^{-\gre_3})).
\end{split}
\end{equation}
One checks that this agrees with \eqref{e.pull_bigger}.

\end{example}

Although the excited Young diagrams introduced in \cite{IkNa:11} and in this paper are both related to the reduced excited Young diagrams of \cite{IkNa:09}, \cite{Kre:05}, \cite{Kre:06}, they  are different combinatorial objects.
While the excited Young diagrams of this paper reside in Young diagrams, those of  \cite{IkNa:11} reside in a grid which is unbounded on the right.   This differs from the reduced excited Young diagrams of \cite{IkNa:09}, \cite{Kre:05}, \cite{Kre:06}, as
well as the excited Young diagrams in this paper, since these all reside in Young diagrams.   Another difference between the excited Young diagrams of \cite{IkNa:11} and of here is that the former are produced by modifying reduced excited Young diagrams by adding `$\times$' symbols; an excited Young diagram of \cite{IkNa:11} with $k$ of these symbols encodes $2^k$ of what we would call excited Young diagrams. 

\subsection{Restriction formula in terms of set-valued shifted tableaux}\label{s.setvalued_shiftedtableaux}

Let $\gl$ be a strict partition. A \textbf{set-valued filling} of $D'_\gl$ is a function $T$ which assigns to each box $(i,j)$ of $D'_{\gl}$ a nonempty subset $T(i,j)$ of $\{1,\ldots,n\}$.  We call $\gl$ the \textbf{shape} of $T$. We often call $(i,j)$ a box of $T$, and refer to an element of $T(i,j)$ as an entry of box $(i,j)$ of $T$, or just an entry of $T$. A set-valued filling $T$ in which each entry of box $(i,j)$ of  $T$ is less than or equal to each entry of box $(i,j+1)$ and strictly less than each entry of box $(i+1,j)$ is said to be \textbf{semistandard}. A  \textbf{set-valued shifted Young tableau}, or just \textbf{set-valued shifted tableau}, is a semistandard set-valued filling of $D'_\gl$. A \textbf{shifted Young tableau} or just \textbf{shifted tableau} is a set-valued tableau in which each box contains a single entry.

Let $\gm$ be a strict partition. We say that a set-valued shifted tableau is \textbf{restricted by $\gm$} if, for any box $(i,j)\in T$ and any entry $x$ of $(i,j)$,
\begin{equation}\label{e.restrictedbyv_bcd}
j-i\leq \gm(x)-1.
\end{equation}
Denote by $\ct'_{\gl}(\gm)$ (resp. $\ct'^{\text{red}}_{\gl}(\gm)$) the set of set-valued shifted tableaux (resp. shifted tableaux) of shape $\gl$ which are restricted by $\gm$
(see Figures \ref{f.syt_allbc} and \ref{f.syt_alld}). The following theorem in type $C_n$ appeared in \cite{Kre:06}.
\begin{figure}[height=30em]
    \centering
    \begin{tikzpicture}[
    level distance = 2.8cm,sibling distance = 2cm, grow = east,
    type2/.style={edge from parent/.style={draw,double distance =
    1.5pt,>=latex,->}},
    type1/.style={edge from parent/.style={draw,>=latex,->}},
    typ1/.style={>=latex,->},
    typ2/.style={double distance=1.5pt,>=latex,->}
]
\newcommand{\syt}[3]
    {
        \begin{tikzpicture}[scale=.6, every node/.style={scale=.8}]
             \foreach \pos / \label in
              { {(1,2)}/{#1},{(2,2)}/{#2},{(2,1)}/{#3} }
                  {
                    \draw \pos +(-.5,-.5) rectangle ++(.5,.5);
                    \draw \pos node {\label};
                    }
        \end{tikzpicture}
      }

\node (e1) [rectangle] {\syt{1}{1}{2}}
    child[type1] {node (e2) [rectangle] {\syt{1}{1}{3}}
        child[type1] {node (e4) [rectangle] {\syt{1}{2}{3}}
            child[type1] {node (e6) [rectangle] {\syt{2}{2}{3}}}
            child[type2] {node (e7) [rectangle] {\syt{1,2}{2}{3}}}
        }
        child[type2] {node (e5) [rectangle] {\syt{1}{1,2}{3}}}
    }
    child[type2] {node (e3) [rectangle] {\syt{1}{1}{2,3}}};
    \draw[rounded corners = 5pt, dashed]
        (e1.north west) -- (e1.north east)
        -- (e6.north east) -- (e6.south east)
        -- (e6.south west) -- (e1.south west)
        -- cycle;
        \draw[typ2] (e1) -- (e3);
        \draw[typ2] (e2) -- (e5);
        \draw[typ2] (e4) -- (e7);
\end{tikzpicture}
    \caption{\label{f.syt_allbc} The set-valued Young tableaux $\ct'_\gl(\gm)$ in types $B_n$ and $C_n$, for $\gm=(4,2,1)$, $\gl=(2,1)$. The dashed line encloses $\ct'^{\text{red}}_\gl(\gm)$. The entries of each element of $\ct'_\gl(\gm)$ record the row numbers of the boxes in the corresponding element of $\ce'_\gl(\gm)$ appearing in Figure \ref{f.eyd_allbc}.}
\end{figure}
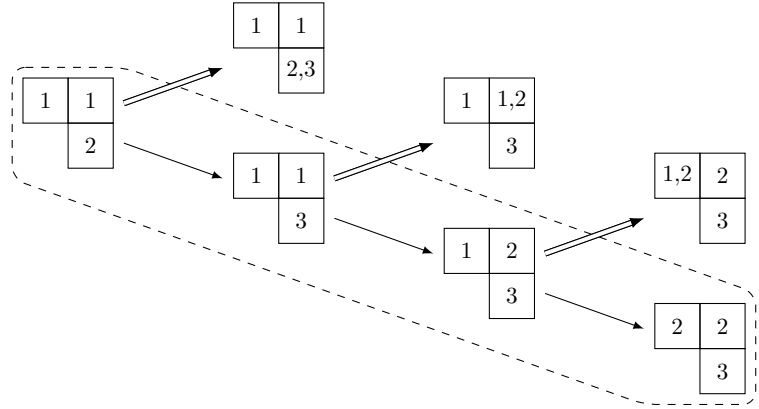

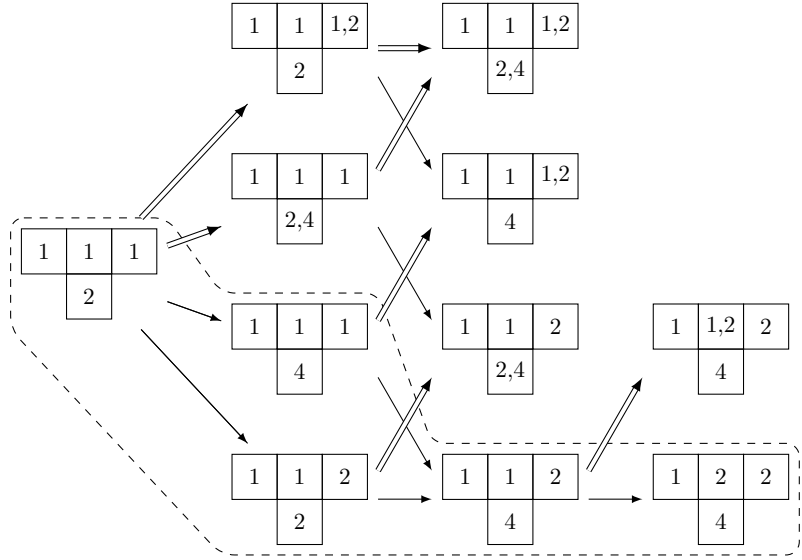
\begin{figure}[height=30em]
    \centering
    \begin{tikzpicture}[
    level distance = 2.8cm,sibling distance = 2cm, grow = east,
    type2/.style={edge from parent/.style={draw,double distance =
    1.5pt,>=latex,->}},
    type1/.style={edge from parent/.style={draw,>=latex,->}},
    typ1/.style={>=latex,->},
    typ2/.style={double distance=1.5pt,>=latex,->}
]
\newcommand{\syt}[4]
    {
        \begin{tikzpicture}[scale=.6, every node/.style={scale=.8}]
             \foreach \pos / \label in
              { {(1,2)}/{#1},{(2,2)}/{#2},{(3,2)}/{#3},{(2,1)}/{#4} }
                  {
                    \draw \pos +(-.5,-.5) rectangle ++(.5,.5);
                    \draw \pos node {\label};
                    }
        \end{tikzpicture}
      }

\node (e1) [rectangle] {\syt{1}{1}{1}{2}}
    child {node (e2) [rectangle] {\syt{1}{1}{2}{2}}
        child {node (e6) [rectangle] {\syt{1}{1}{2}{4}}
            edge from parent[draw=none]
            child {node (e10) [rectangle] {\syt{1}{2}{2}{4}}
                edge from parent[draw=none]
            }
        }
    }
    child {node (e3) [rectangle] {\syt{1}{1}{1}{4}}
        child {node (e7) [rectangle] {\syt{1}{1}{2}{2,4}}
            edge from parent[draw=none]
            child {node (e11) [rectangle] {\syt{1}{1,2}{2}{4}}
                edge from parent[draw=none]
            }
        }
    }
    child {node (e4) [rectangle] {\syt{1}{1}{1}{2,4}}
        child {node (e8) [rectangle] {\syt{1}{1}{1,2}{4}}
            edge from parent[draw=none]
        }
    }
    child {node (e5) [rectangle] {\syt{1}{1}{1,2}{2}}
        child {node (e9) [rectangle] {\syt{1}{1}{1,2}{2,4}}
            edge from parent[draw=none]
        }
    };
    \draw[rounded corners = 5pt, dashed]
        (e1.north west) -- (e1.north east)
        -- (e3.north west) -- (e3.north east)
        -- (e6.north west) -- (e10.north east)
        -- (e10.south east) -- (e2.south west)
        -- (e1.south west) -- cycle;
    \draw[typ2] (e5.0) -- (e9.180);
    \draw[typ1] (e5.340) -- (e8.160);
    \draw[typ2] (e4.20) -- (e9.200);
    \draw[typ1] (e4.340) -- (e7.160);
    \draw[typ2] (e3.20) -- (e8.200);
    \draw[typ1] (e3.340) -- (e6.160);
    \draw[typ2] (e2.20) -- (e7.200);
    \draw[typ1] (e2.0) -- (e6.180);
    \draw[typ2] (e6.20) -- (e11.200);
    \draw[typ1] (e6.0) -- (e10.180);
    \draw[typ2] (e1) -- (e4);
    \draw[typ2] (e1) -- (e5);
    \draw[typ1] (e1) -- (e2);
    \draw[typ1] (e1) -- (e3);
\end{tikzpicture}
    \caption{\label{f.syt_alld} The set-valued shifted tableaux $\ct'_\gl(\gm)$ in type $D_6$, where $\gm=(5,3,2,1)$, $\gl=(3,1)$. The dashed line encloses $\ct'^{\text{red}}_\gl(\gm)$. The entries of each element of $\ct'_\gl(\gm)$ record the row numbers of the boxes in the corresponding element of $\ce'_\gl(\gm)$ appearing in Figure \ref{f.eyd_alld}.}
\end{figure}

\begin{Thm}\label{t.ktheory_syt_bcd}
Let $w\leq v\in W^{P_n}$, and let $\gl=\gl'_w$ and $\gm=\gl'_v$ be the corresponding strict partitions. Then
{\allowdisplaybreaks
\begin{align*}
    & B_n:\quad i_v^*[\co_{X^w}]=(-1)^{l(w)}\sum_{T\in\ct'_{\gl}(\gm)}\prod_{(i,j)\in T}\prod_{x\in T(i,j)}\left(e^{-2^{-\delta_{ij}}(\gre_{v_{n+x+1}}+\gre_{v_{n+x+j-i+1}})}-1\right)\\
    & C_n:\quad i_v^*[\co_{X^w}]=(-1)^{l(w)}\sum_{T\in\ct'_{\gl}(\gm)}\prod_{(i,j)\in T}\prod_{x\in T(i,j)}\left(e^{-(\gre_{v_{n+x}}+\gre_{v_{n+x+j-i}})}-1\right)\\
    & D_{n}:\quad i_v^*[\co_{X^w}]=(-1)^{l(w)}\sum_{T\in\ct'_{\gl}(\gm)}\prod_{(i,j)\in T}\prod_{x\in T(i,j)}\left(e^{-(\gre_{v_{n+x}}+\gre_{v_{n+x+j-i+1}})}-1\right),
\end{align*}
}
where for $1\leq m\leq n$, $\gre_{\bar{m}}$ is defined to equal $-\gre_m$.
\end{Thm}

Theorem \ref{t.ktheory_syt_bcd} is essentially the same statement as Theorem \ref{t.ktheory_eyd_bcd}, except that the indexing set $\ce'_{\gl}(\gm)$ has been replaced by $\ct'_{\gl}(\gm)$.
This replacement is given by the map $f:\ct'_{\gl}(\gm)\to\{\text{subsets of }D'_{\gm}\}$,
\begin{equation}\label{e.f_bcd}
\begin{split}
f(T)= \{(x,x+j-i)\mid (i,j)\in\gl, x\in T(i,j)\}.
\end{split}
\end{equation}
\begin{Prop}\label{p.fbijective_bcd}
The map $f$ is a bijection from $\ct'_{\gl}(\gm)$ to $\ce'_{\gl}(\gm)$.
\end{Prop}
The proof of this proposition is so similar to that of Proposition \ref{p.fbijective} that we omit the details.

\subsection{Hilbert series and Hilbert polynomials of points on Schubert varieties}\label{ss.hilbseries_BCD}

In types $C_n$ and $D_n$, the parabolic subgroup ${P_n}$ is cominuscule (cf. \cite[9.0.14]{BiLa:00}). Thus Corollary \ref{c.cominusculemult} may be used to compute the Hilbert series, Hilbert polynomial, and multiplicity of $X^w$ at $v$. In the present setting, the constant $m_k$ of Corollary \ref{c.cominusculemult} is equal to the number of excited shifted Young diagrams $C\in\ce'_{\gl}(\gm)$ such that the number of boxes of $C$ is $k+|\gm|$. In terms of set-valued shifted Young tableaux, $m_k$ is equal to the number of $T\in\ct'_{\gl}(\gm)$ with $k+|\gm|$ entries.

\begin{example} In type $C_n$, $n=4$, let $w=\{1,2,\bar{4},\bar{3},3,4,\bar{2},\bar{1}\}$, $v=\{2,\bar{4},\bar{3},\bar{1},1,3,4,\bar{2}\}$, as in Example \ref{ex.restriction_BC}. Then $\gl=\gl'_w=(2,1)$, $\gm=\gl'_v=(4,2,1)$, $l(w)=|\gl'_w|=3$, and $d_w=n(n+1)/2-l(w)=7$. The set of excited shifted Young diagrams $\ce'_{\gl}(\gm)$ appears in Figure \ref{f.eyd_allbc}, and the set of set-valued shifted tableaux $\ct'_{\gl}(\gm)$ appears in Figure \ref{f.syt_allbc}. From either of these figures, one sees that $m_0=4$ and $m_1=3$. Hence
{\allowdisplaybreaks
\begin{align*}
    H(X^w,v)(t)&=\frac{4}{(1-t)^7}-\frac{3}{(1-t)^6},\\
    h(X^w,v)(i)&=4{i +6\choose 6}-3{i+5\choose 5},\\
    \mult(X^w,v)&=4.
\end{align*}
}

\end{example}

\begin{example}\label{ex.Hilbert_Dn} In type $D_n$, $n=6$, let $w=\{1,2,4,6,\bar{5},\bar{3},3,5,\bar{6},\bar{4},\bar{2},\bar{1}\}$ and $v=\{2,6,\bar{5},\bar{4},\bar{3},\bar{1},1,3,4,5,\bar{6},\bar{2}\}$, as in Example \ref{ex.restriction_D}. Then $\gl=\gl'_w=(3,1)$, $\gm=\gl'_v=(5,3,2,1)$, $l(w)=|\gl'_w|=4$, and $d_w=n(n-1)/2 - l(w)=11$. The set of excited shifted Young diagrams $\ce'_{\gl}(\gm)$ appears in Figure \ref{f.eyd_alld}, and the set of set-valued shifted tableaux $\ct'_{\gl}(\gm)$ appears in Figure \ref{f.syt_alld}. From either of these figures one reads off $m_0=5$, $m_1=5$, and $m_2=1$. Hence
\begin{align*}
    H(X^w,v)(t)&=\frac{5}{(1-t)^{11}}-\frac{5}{(1-t)^{10}}+\frac{1}{(1-t)^{9}},\\
    h(X^w,v)(i)&=5{i +10\choose 10}-5{i+9\choose 9}+{i+8\choose 8},\\
    \mult(X^w,v)&=5.
\end{align*}
\end{example}

In type $B_n$, ${P_n}$ is not cominuscule, so Corollary 1.14 may not be applied directly. However, since the isomorphism $\pi:\OG(n+1,2n+2)\to \OG(n,2n+1)$
identifies $T$-fixed points and Schubert varieties (see Section \ref{ss.orthog_grassmannians} or \cite[1.3]{RaUp:10}), properties of singularities of Schubert varieties in $\OG(n,2n+1)$ can be obtained from those of Schubert varieties in $\OG(n+1,2n+2)$.

\begin{example} In type $B_n$, $n=5$, let $w=(1,2,4,\bar{5},\bar{3},3,5,\bar{4},\bar{2},\bar{1})$ and\\ $v=(2,5,\bar{4},\bar{3},\bar{1},1,3,4,\bar{5},\bar{2})$.
Then $v$ and $w$ are identified with the corresponding elements of $W^{P_{n+1}}$ of Example \ref{ex.Hilbert_Dn} (see Section \ref{ss.strict_partitions_bcd}).
Therefore $H(X^w,v)(t)$, $h(X^w,v)(i)$, and $\mult(X^w,v)$ are the same as in that example.
\end{example}

Other multiplicity formulas appear in \cite{LaWe:90}, \cite{GhRa:06}, and \cite{RaUp:10}. The above formula for the multiplicity of $X^w$ at $v$, expressed in terms of reduced excited shifted Young diagrams, appeared earlier in \cite{IkNa:09}. In type $C_n$, the formula can be deduced from \cite{Kre:06} and \cite{GhRa:06}.  Formulas for the Hilbert function of $X^w$ at $v$ appear in \cite{GhRa:06}, \cite{RaUp:10}, and \cite{Upa:09}.


\appendix

\section{Root systems and Weyl groups in types $A_n$, $B_n$, $C_n$, and $D_n$}\label{s.appendix_roots}

We review some facts about the classical root systems and Weyl groups.

\noindent \textbf{Type $A_{n-1}$}.
The special linear group $\SL_{n}(\C)$ is equal to $G=\{g\in \GL_{n}(\C)\mid \det(g)=1\}$. The Lie algebra $\fg=\fs\fl_{n}(\C)=\{a\in \fg\fl_{n}(\C)\mid \tr(a)=0\}$. The set of diagonal matrices in $\fg$ forms a Cartan subalgebra
\begin{equation*}
    \fh=\{\diag(a_1,\ldots,a_{n})\mid a_1,\ldots,a_{n}\in\C, \sum a_i = 0\}.
\end{equation*}
For $1\leq i\leq n$, let $\gre_i\in\fh^*$ be the linear functional
 \begin{equation*}
    \gre_i(\diag(a_1,\ldots,a_{n}))=a_i.
 \end{equation*}
 Then $\{\gre_1,\ldots,\gre_{n}\}$ span $\fh^*$. The set of roots $\gP$ of $\fg$ relative to $\fh$ is $\{\gre_i- \gre_j, 1\leq i\neq j\leq n\}$.
 The set
 \begin{equation*}
    \gD=\{\ga_1=\gre_1-\gre_2,\ldots,\ga_{n-1}=\gre_{n-1}-\gre_{n}\}
\end{equation*}
forms a base of $\gP$, with respect to which the set of positive roots is
\begin{equation*}
    \gP^+=\{\gre_i-\gre_j, 1\leq i<j\leq n\}.
\end{equation*}
For $i\in\{1,\ldots,n-1\}$,  the reflection $s_i$ along $\ga_i$ is given by:
\begin{equation*}
    s_i:
    \begin{cases}
        &\gre_i\mapsto \gre_{i+1}\\
        &\gre_{i+1}\mapsto\gre_{i}
    \end{cases}.
\end{equation*}
The Weyl group $W=\langle s_1,\ldots,s_n\rangle\subseteq \GL(\fh^*)$ is the group of permutations of $\{\gre_1,\ldots,\gre_{n}\}$, which is isomorphic to $S_n$. Denote a permutation $w\in S_n$ by its one-line notation $w=(w_1,\ldots,w_n)$, where $w(1)=w_1$ , $w(2)=w_2$, etc.


 \noindent \textbf{Type $C_n$}. Define the inner product $\langle x,y\rangle=x^tJy$, $x,y\in\C^{2n}$, where $J$ is the antidiagonal $2n\times 2n$ matrix whose top  $n$ antidiagonal entries are 1's and whose bottom $n$ antidiagonal entries are -1's. The symplectic group $\Sp_{2n}(\C)$ is equal to $G=\{g\in \GL_{2n}(\C)\mid \langle gu,gv\rangle = \langle u,v\rangle, u,v\in\C^{2n}\}$. The Lie algebra $\fg=\fs\fp_{2n}(\C)=\{a\in \fg\fl_{2n}(\C)\mid \langle au,v\rangle +\langle u,av\rangle=0, u,v\in\C^{2n}\}$.
 The set of diagonal matrices in $\fg$ forms a Cartan subalgebra
 \begin{equation*}
    \fh=\{\diag(a_1,\ldots,a_n,-a_n,\ldots,-a_1)\mid a_1,\ldots,a_n\in\C\}.
 \end{equation*}
For $1\leq i\leq n$, let $\gre_i\in\fh^*$ be the linear functional
 \begin{equation*}
    \gre_i(\diag(a_1,\ldots,a_n,-a_n,\ldots,-a_1))=a_i.
 \end{equation*}
 Then $\{\gre_1,\ldots,\gre_n\}$ forms a basis for $\fh^*$. The set of roots $\gP$ of $\fg$ relative to $\fh$ is $\{\pm\gre_i\pm \gre_j, 1\leq i\neq j\leq n\}\cup\{\pm 2\gre_i,i=1,\ldots,n\}$. The set \begin{equation*}
    \gD=\{\ga_1=\gre_1-\gre_2,\ldots,\ga_{n-1}=\gre_{n-1}-\gre_n\}\cup \{\ga_n=2\gre_n\}
\end{equation*}
forms a base of $\gP$, with respect to which the set of positive roots is
\begin{equation*}
    \gP^+=\{\gre_i\pm\gre_j, 1\leq i<j\leq n\}\cup \{2\gre_i,i=1,\ldots,n\}.
\end{equation*}
 The reflection $s_i$ along $\ga_i$ is given by:
\begin{equation*}
\begin{split}
    &s_i:
    \begin{cases}
        &\gre_i\mapsto \gre_{i+1}\\
        &\gre_{i+1}\mapsto\gre_{i}\\
        &\gre_j\mapsto\gre_j, j\neq i,i+1
    \end{cases},
    \quad
    \text{ if }i\in\{1,\ldots,n-1\}
    \\
    &s_n:
    \begin{cases}
        &\gre_n\mapsto -\gre_{n}\\
        &\gre_j\mapsto\gre_j, j\neq n
    \end{cases}
\end{split}
\end{equation*}
The Weyl group $W=\langle s_1,\ldots,s_n\rangle\subseteq \GL(\fh^*)$ is the group of permutations and sign changes of $\{\gre_1,\ldots,\gre_n\}$. More precisely, $W\cong S_n\ltimes\Z_2^{n}$.

The map $W\to S_{2n}$ given by
    $s_i\mapsto (i,i+1)(\bar{i+1},\bar{i})$, $i\neq n$, $s_n\mapsto (i,\bar{i})$,
is a monomorphism, identifying $W$ with
\begin{equation}\label{e.Weyl_Cn}
    W\cong \{(w_1,\ldots,w_{2n})\in S_{2n}\mid w_{\bar{\imath}}=\bar{w_i}, 1\leq i\leq n\}.
\end{equation}


\noindent \textbf{Type $B_n$}.
Define the inner product $\langle x,y\rangle=x^tJy$, $x,y\in\C^{2n+1}$, where $J$ is the antidiagonal $(2n+1)\times (2n+1)$ matrix all of whose antidiagonal entries are 1's, except for the entry in row and column $n+1$, which is 2. The odd orthogonal group $\SO_{2n+1}(\C)$ is equal to $G=\{g\in \GL_{2n+1}(\C)\mid \langle gu,gv\rangle = \langle u,v\rangle, u,v\in\C^{2n+1}\}$.The Lie algebra $\fg=\fs\fo_{2n+1}(\C)=\{a\in \fg\fl_{2n+1}(\C)\mid \langle au,v\rangle +\langle u,av\rangle=0, u,v\in\C^{2n+1}\}$. The set of diagonal matrices in $\fg$ forms a Cartan subalgebra
\begin{equation*}
    \fh=\{\diag(a_1,\ldots,a_n,0,-a_n,\ldots,-a_1)\mid a_1,\ldots,a_n\in\C\}.
\end{equation*}
For $1\leq i\leq n$, let $\gre_i\in\fh^*$ be the linear functional
 \begin{equation*}
    \gre_i(\diag(a_1,\ldots,a_n,0,-a_n,\ldots,-a_1))=a_i.
 \end{equation*}
 Then $\{\gre_1,\ldots,\gre_n\}$ forms a basis for $\fh^*$. The set of roots $\gP$ of $\fg$ relative to $\fh$ is $\{\pm\gre_i\pm \gre_j, 1\leq i\neq j\leq n\}\cup\{\pm\gre_i,i=1,\ldots,n\}$. The set \begin{equation*}
    \gD=\{\ga_1=\gre_1-\gre_2,\ldots,\ga_{n-1}=\gre_{n-1}-\gre_n\}\cup \{\ga_n=\gre_n\}
\end{equation*}
forms a base of $\gP$, with respect to which the set of positive roots is
\begin{equation*}
    \gP^+=\{\gre_i\pm\gre_j, 1\leq i<j\leq n\}\cup \{\gre_i,i=1,\ldots,n\}.
\end{equation*}

Since the roots in types $B_n$ and $C_n$ agree up to scalar multiples, they have the same Weyl group. Thus \eqref{e.Weyl_Cn} also gives an identification of the Weyl group in type $B_n$.

\noindent \textbf{Type $D_n$}.
Define the inner product $\langle x,y\rangle=x^tJy$, $x,y\in\C^{2n}$, where $J$ is the antidiagonal $2n\times 2n$ matrix all of whose antidiagonal entries are 1's. The even orthogonal group $\SO_{2n}(\C)$ is equal to $G=\{g\in \GL_{2n}(\C)\mid \langle gu,gv\rangle = \langle u,v\rangle, u,v\in\C^{2n}\}$. The Lie algebra $\fg=\fs\fo_{2n}(\C)=\{a\in \fg\fl_{2n}(\C)\mid \langle au,v\rangle +\langle u,av\rangle=0, u,v\in\C^{2n}\}$, where $\langle x,y\rangle=x^tJy$, $x,y\in\C^{2n}$.
The set of diagonal matrices in $\fg$ forms a Cartan subalgebra
 \begin{equation*}
    \fh=\{\diag(a_1,\ldots,a_n,-a_n,\ldots,-a_1)\mid a_1,\ldots,a_n\in\C\}.
 \end{equation*}
For $1\leq i\leq n$, let $\gre_i\in\fh^*$ be the linear functional
 \begin{equation*}
    \gre_i(\diag(a_1,\ldots,a_n,-a_n,\ldots,-a_1))=a_i.
 \end{equation*}
Then $\{\gre_1,\ldots,\gre_n\}$ forms a basis for $\fh^*$. The set of roots $\gP$ of $\fg$ relative to $\fh$ is $\{\pm\gre_i\pm \gre_j, 1\leq i\neq j\leq n\}$. The set \begin{equation*}
    \gD=\{\ga_1=\gre_1-\gre_2,\ldots,\ga_{n-1}=\gre_{n-1}-\gre_n,\ga_n=\gre_{n-1}+\gre_n\}
\end{equation*}
forms a base of $\gP$, with respect to which the set of positive roots is
\begin{equation*}
    \gP^+=\{\gre_i\pm\gre_j, 1\leq i<j\leq n\}.
\end{equation*}
 The reflection $s_i$ along $\ga_i$ is given by:
\begin{equation*}
\begin{split}
    &s_i:
    \begin{cases}
        &\gre_i\mapsto \gre_{i+1}\\
        &\gre_{i+1}\mapsto\gre_{i}\\
        &\gre_j\mapsto\gre_j, j\neq i,i+1
    \end{cases},
    \quad
    \text{ if }i\in\{1,\ldots,n-1\}
    \\
    &s_n:
    \begin{cases}
        &\gre_{n-1}\mapsto -\gre_{n}\\
        &\gre_n\mapsto -\gre_{n-1}\\
        &\gre_j\mapsto\gre_j, j\neq n-1,n
    \end{cases}
\end{split}
\end{equation*}
The Weyl group $W=\langle s_1,\ldots,s_n\rangle\subseteq \GL(\fh^*)$ is the group of permutations and even number of sign changes of $\{\gre_1,\ldots,\gre_n\}$. More precisely, $W\cong S_n\ltimes\Z_2^{n-1}$.

The map $W\to S_{2n}$ given by
    $s_i\mapsto (i,i+1)(\bar{i+1},\bar{i})$, $i\neq n$, $s_n\mapsto (n,n+1)(n-1,n)(n+1,n+2)(n,n+1)$,
is a monomorphism, identifying $W$ with
\begin{equation}\label{e.Weyl_Dn}
    W\cong\left\{(w_1,\ldots,w_{2n})\in S_{2n}\mid
            w_{\bar{\imath}}=\bar{w_i}, 1\leq i\leq n,
             \#\{i<n\mid w_i>n\}\text{ is even}.
            \right\}
\end{equation}

\section{Restriction formulas and opposite Schubert varieties} \label{s.appendix_restriction-opp}
In this section we explain the relation between the restriction formulas for
Schubert varieties and for opposite Schubert varieties.  We have included this
because some references use the opposite Schubert varieties---
indeed, the formulas in \cite{Gra:02} are for $ i_x^*[\co_{X_w}] $,  where
 $X_w = \overline{B \cdot wB}$ is the opposite Schubert
 variety to $X^w$.  However, the formula of Theorem \ref{t.pullback} can be obtained
 from the formula for opposite Schubert varieties by using Proposition \ref{p.translate} below.  To prove this proposition
 we need two lemmas.  Let $*$ denote the involution of $R(T)$
 defined by $*(e^{\lambda}) = e^{- \lambda}$.   If $T$ acts on any scheme $M$, we can
 define a new action $\odot$ of $T$ by the rule $t \odot m = t^{-1}m$.  Write $K_T(M,\odot)$ to
 denote the equivariant $K$-theory of $M$ with the $\odot$ action.  Any
 coherent sheaf on $M$ which is equivariant with respect to the original $T$-action is
 equivariant with respect to the $\odot$ action.  There is a map $K_T(M) \to K_T(M,\odot)$,
 $\xi \mapsto \xi_{\odot}$, taking the class of a $T$-equivariant sheaf $\cf$ to the class
 of the same sheaf, but viewed as equivariant with respect to the $\odot$ action.  Observe
 that if $M$ is a point, then $K_T(M) = R(T)$ and $\xi_{\odot} = *\xi$.

 \begin{Lem} \label{l.dual}
 Suppose $T$ acts on a smooth scheme $X$.  Let $x \in X^T$ and let $i: \{ x \} \to X$ denote the inclusion.
 Then for any $\xi \in K_T(M)$,
 $$
 i^*(\xi_{\odot}) = *( i^*\xi).
 $$
 \end{Lem}

 \begin{proof}
 If we write $\xi = \sum_j a_j [V^j]$, where each $V^j$ is a $T$-equivariant vector bundle, then for each j,
 the fiber $V^j_x$ is a representation of $T$.  We have
 $i^* \xi = \sum a_j [V^j_x]$ and
 $$
 i^*(\xi_{\odot}) = \sum a_j ([V^j_x]_{\odot}) = * ( \sum a_j ([V^j_x])) =  *( i^*\xi).
 $$
 \end{proof}

\begin{Lem} \label{l.involution}
There exists an involution $\Psi: G \to G$ such that $\Psi(t) = t^{-1}$ for $t \in T$, $\Psi(B) = B^-$, and
$\psi(nT) = nT$ for $n \in N_G(T)$.
\end{Lem}

\begin{proof}
Given a root $\gb$ of $\fg$, let $\fg_{\beta}$ denote the corresponding root space.
For each simple root $\ga$ of $\fg$ we can find elements $x_{\ga} \in \fg_{\ga}$, $x_{-\ga} \in \fg_{-\ga}$, and
$h_{\ga} \in \ft$ such that $x_{\ga}, \{h_{\ga}, x_{-\ga} \}$ is an ${\mathfrak sl}_2$-triple (see \cite[Section 2.4]{Sam:90}.
There is an involution $\psi$ of $\fg$ which acts by multiplication by $-1$ on $\ft$ such that if $\ga$ is any
simple root, then
$\psi(x_{\ga}) = x_{-\ga}$ (see \cite[Proposition 14.3]{Hum:72}).  It follows that $\psi(\fb) = \fb^-$.
 Let $\tilde{G}$ denote the simply
 connected algebraic group with Lie algebra $\fg$, and let
 $\tilde{T}$ denote the subgroup of $\tilde{G}$ with Lie algebra $\ft$.  Because $\tilde{G}$ is simply connected,
 $\psi$ lifts to an automorphism $\Psi$ of $\tilde{G}$.  Moreover, since $\psi$ acts by multiplication
 by $-1$ on $\ft$, $\Phi$ takes any element of $\tilde{T}$ to its inverse.  The group $G$ is isomorphic to
 $G/Z_1$, where $Z_1$ is a subgroup of the center $Z$ of $\tilde{G}$.  Since $Z_1 \subset \tilde{T}$,
 and $Z_1$ is closed under inverses, $\Psi(Z_1) = Z_1$.  Therefore $\Psi$ descends to an automorphism
 (also denoted $\Psi$) of $G$.  The assertions $\Psi(t) = t^{-1}$ for $t \in T$, $\Psi(B) = B^-$ follow from
 the corresponding properties of $\psi$.  Finally, let $J_{\ga} = \frac{\pi}{2}(x_{\ga} - x_{-\ga})$.
 The simple reflection $s_{\ga}$ in $W$ is represented by
 the element $n_{\ga} = \exp(J_{\ga}) \in N_G(T)$ (see \cite[Section 2.15]{Sam:90}).
 The argument in Samelson shows that $s_{\ga}$ is also represented by the element $ \exp(-J_{\ga}) = \Psi(n_{\ga})$.
 Hence $\Psi(n_{\ga} T) = n_{\ga} T$.   Given any $n \in N_G(T)$, $nT = n_{\ga_{1}} n_{\ga_2} \cdots n_{\ga_k} T$
 for some simple roots $\ga_1, \ldots, \ga_k$.  It follows that $\Psi(nT) = nT$, as claimed.
 \end{proof}

 Let $\tilde{X} = G/B^-$, $\tilde{X}_w= \overline{B^- \cdot w B^-} \subset \tilde{X}$.  Let $\tilde{i}_x: \{ pt \} \to \tilde{X}$ be the map $\tilde{i}_x(pt) = xB^-$.

 \begin{Prop} \label{p.translate}
 $$
  i_x^*[\co_{X^w}]  = *(i_{xw_0}^*[\co_{X_{ww_0}}]).
  $$
 \end{Prop}

 \begin{proof}
 The map $\phi: G/B \to G/B^-$ defined by $\phi(gB) = g w_0 B^-$ is a $G$-equivariant isomorphism.
 Since $\phi(xB) = x w_0 B^-$ and $\phi(X^w) = \tilde{X}_{ww_0}$, we have
 \begin{equation} \label{e.translate1}
 i_x^*[\co_{X^w}] = \tilde{i}_{xw_0}[\co_{\tilde{X}_{ww_0}}].
 \end{equation}
Let $\Psi$ denote the involution of $G$ from Lemma \ref{l.involution}.
Since $\Psi(B) = B^-$, there is an induced map (which we also denote by
 $\Psi$) $G/B \to G/B^-$, $gB \mapsto \psi(g)B^-$.
 Since $\Psi(nT) = nT$ and $\Psi(B) = B^-$,
 for any $u \in W$,
 $\Psi(X_u) = \tilde{X}_u$.
 The map $\Psi$ is $T$-equivariant if $T$ acts
 by left multiplication on $G/B^-$, and by $t \odot gB = t^{-1} gB$ on $G/B$.
 Therefore, $\Psi^*[\co_{\tilde{X}_u}] = [\co_{X_u}]_{\odot}$, where the subscript
 $\odot$ indicates that we are using the $\odot$ action of $T$.  Therefore,
 \begin{equation} \label{e.translate2}
 \tilde{}i^*_x[\co_{\tilde{X}_u}] = i^*( [\co_{X_u}]_{\odot}) =  *(i^* [\co_{X_u}]),
 \end{equation}
where the second equality follows from Lemma \ref{l.dual}.
The proposition follows from \eqref{e.translate1} and  \eqref{e.translate2}, taking
$u = w w_0$.
\end{proof}

\bibliographystyle{amsalpha}
\bibliography{refs_eyd}

\end{document}